\documentclass[11pt, a4paper]{article}

\usepackage[utf8]{inputenc}
\usepackage[T1]{fontenc}
\usepackage[english]{babel}
\usepackage{algorithm}
\usepackage{algorithmicx}
\usepackage{algpseudocode}
\usepackage{amsmath, amssymb, amsthm}
\usepackage{mathtools}
\usepackage{hyperref}
\usepackage{geometry}
\usepackage[expansion=false]{microtype}
\usepackage[shortlabels]{enumitem}
\usepackage{xcolor}
\usepackage{booktabs}
\usepackage{float}
\usepackage{caption}

\definecolor{linkblue}{RGB}{0, 70, 140}
\hypersetup{
    colorlinks=true,
    linkcolor=linkblue,
    citecolor=linkblue,
    urlcolor=linkblue
}

\theoremstyle{plain}
\newtheorem{theorem}{Theorem}[section]
\newtheorem{lemma}[theorem]{Lemma}
\newtheorem{corollary}[theorem]{Corollary}
\newtheorem{proposition}[theorem]{Proposition}

\theoremstyle{definition}
\newtheorem{definition}[theorem]{Definition}
\newtheorem{example}[theorem]{Example}
\newtheorem{remark}[theorem]{Remark}

\newcommand{\N}{\mathbb{N}}
\newcommand{\Z}{\mathbb{Z}}
\newcommand{\leg}[2]{\left(\frac{#1}{#2}\right)}

\begin{document}

\title{Sieve dimension and search depth for the\\ Erd\H{o}s--Straus conjecture, $n \equiv 1 \pmod{24}$}
\author{Benjamin Dahan}
\date{\today}

\maketitle

\begin{abstract}
For primes $n\equiv1\pmod{24}$ we study how deep an explicit, factorization-free search for a decomposition of $4/n$ into three unit fractions has to go. Write $E_2(N;J)$ for the set of such primes $n\le N$ at which no witness of depth at most $J$ exists, in the sense of the two-parameter criterion of Theorem~\ref{thm:gen_criterion} with coprime parameters $u,a\le J$. We prove that for every fixed $J$
$$|E_2(N;J)|\ \ll_J\ \frac{N}{(\log N)^{\,1+\mathfrak{A}(J)/2}},$$
the exponent being the exact dimension of the covering on which the proof rests. The proof replaces the subgroup generated by the prime factors, an approach that breaks down as soon as $(\Z/4m)^{\times}$ has exponent greater than $2$, by a fixed-point-free involution, and is unconditional at every $J$.

Second, we exhibit an unconditional obstruction. At the shift $c=7$ there are $\asymp N(\log N)^{-3/2}$ primes $n\le N$, $n\equiv1\pmod{24}$, at which \emph{both} branches of the divisor criterion fail. The representation of $K_7=(n+7)/4$ by the principal form of discriminant $-7$ has to be \emph{primitive}, so that the relevant input is the primitive-representation theorem of Fuchs, Hsu, Rickards, Schindler and Stange~\cite{fuchs2025} rather than the classical results of Iwaniec; a fixed shift therefore cannot leave a finite residual set.

Third, a factorization-free procedure decides the conjecture for all primes of an interval $[N,2N]$. In its unsplit form, its Type~II pass costs $\mathcal{O}(N(\log N)^{3})$ while its Type~I pass costs $\Theta(N^{2})$, which locates the whole quadratic cost in the extraction of the divisors of $4u^{2}d+1$ and exhibits an asymmetry between the two halves of the Type~I/Type~II dichotomy. Truncating the Type~I pass at $u\le N^{1/3}$ and transferring the remaining large-$u$ parameter space to the sparse residual set left by the large sieve yields an unconditional, factorization-free batch decision procedure running in $\mathcal{O}\big(N^{2}(\log N)^{2}\exp(-c_0(\log N)^{2/3})\big)=\mathcal{O}\big(N^{2}\exp(-(c_0-o(1))(\log N)^{2/3})\big)=o(N^{2})$ arithmetic operations. Finally, an exact two-variable Dirichlet series evaluation triples the leading constant in the large-sieve trade-off to $\kappa_1^{\ast}=9/(2\pi^{2})$, improving the Bombieri--Vinogradov density bound to $\mathcal{O}\big(N\exp(-0.0942(\log N)^{2/3})\big)$ at the uniform cutoff $p>J^{5}$ (and to $\mathcal{O}_\varepsilon\big(N\exp(-(0.1093-\varepsilon)(\log N)^{2/3})\big)$ at $p>J^{4+\varepsilon}$), while Proposition~\ref{prop:optimality_23} establishes $(\log N)^{2/3}$ as the $\tau_3$-type ceiling of the single-prime two-parameter witness system. The conjecture itself remains open.
\end{abstract}

\section{Introduction} \label{sec:intro}

Formulated by Paul Erdős and Ernst G. Straus in 1948~\cite{erdos_straus}, the Erdős--Straus conjecture posits that for every integer $n\ge2$ there are positive integers $x,y,z$ with
\begin{equation} \label{eq:erdos_straus}
\frac{4}{n}=\frac{1}{x}+\frac{1}{y}+\frac{1}{z}.
\end{equation}

A decomposition of $4/p$ for a prime $p\mid n$ extends to $n$, so the conjecture reduces to prime $n$; and the polynomial identities and local congruence methods going back to Mordell~\cite{mordell1969} settle every residue class of primes except those with $n\equiv r^{2}\pmod{840}$, $r\in\{1,11,13,17,19,23\}$. All six of these lie in the class $n\equiv1\pmod{24}$, which is the class considered throughout. For $n$ even, for $n\equiv3\pmod 4$, and for $n\equiv1\pmod{24}$ with $(n+3)/4$ divisible by $5$, elementary identities of the same vintage apply; we do not reproduce them. The third of these is a sub-case of the class studied here, dispatched by Theorem~\ref{thm:3.3} at the shift $c=3$ with $d_1=K_3$ and $d_2=K_3/5$.

\subsection{Context and related work} \label{sec:recent_approaches}
Three methodological strands situate what follows. Vaughan~\cite{vaughan1970} established the density bound $\mathcal{O}(N\exp(-c_V(\log N)^{2/3}))$ for an absolute constant $c_V>0$, for the exceptional set, a strong existence guarantee which does not analyse any explicit search procedure. Computational verifications have reached $n\le10^{14}$ (Swett~\cite{swett1999}), $10^{17}$ (Salez~\cite{salez2014}) and $10^{18}$ (Mihnea and Dumitru~\cite{mihnea2025}) by exhaustive search, giving numerical evidence at large bounds without asymptotic guarantees on density or cost. The present paper sits between the two: it analyses an explicit factorization-free procedure and certifies the \emph{depth} at which it succeeds, at the price of a density bound weaker than Vaughan's.

Several recent preprints propose constructive frameworks for sub-cases of the conjecture: three-dimensional affine lattices~\cite{dyachenko2025}, ``tame'' solutions~\cite{xu2026}, a cubic surface model~\cite{bello2026}, parametric identities for related conjectures~\cite{ghermoul2025, mballa2026a, mballa2026b}, and the Type~A/B/C classification of~\cite{lopez2022, lopez2024}. Section~\ref{sec:6} shows that all of them are projections of the single parameter space $K=(n+c)/4$, and records what is due to whom; we claim no priority for the characterisations themselves. The analytic frame is Vaughan's, with the Type~I/Type~II dichotomy in the form of Elsholtz--Tao~\cite{elsholtz2013} and Pomerance--Weingartner~\cite{pomerance2026}, whose closed criteria we use throughout; for the Schinzel generalization the latter show that the threshold $n_m$, if it exists, satisfies $n_m\ge\exp(m^{1/3+o(1)})$. The closest antecedents for the exceptional-set arguments of Section~\ref{sec:5.1} are Sander~\cite{sander1991,sander1994}, who applied the Rosser and half-dimensional sieves to this equation. Ventas~\cite{ventas2026} explores the finiteness of counterexamples via continued fractions.

\subsection{Contributions} \label{sec:contributions}

\emph{The exact sieve dimension at fixed depth} (Section~\ref{sec:fixed_depth_sub}). Let $\mathcal{P}$ be a finite set of coprime pairs $(u,a)$ and let $E_{\mathcal{P}}(N)$ be the set of primes $n\le N$, $n\equiv1\pmod{24}$, at which no pair of $\mathcal{P}$ produces a witness. Theorem~\ref{thm:fixed_depth} gives $|E_{\mathcal{P}}(N)|\ll_{\mathcal{P}}N(\log N)^{-1-|\mathcal{P}|/2}$, unconditionally, the exponent $|\mathcal{P}|/2$ being exactly the dimension of the sieve problem produced by the covering. Specialising to the one-parameter slice used by the geometric sieve gives $N(\log N)^{-1-J/2}$; specialising to all coprime pairs in $[1,J]^{2}$ gives $N(\log N)^{-1-\mathfrak{A}(J)/2}$. The engine is Lemma~\ref{lem:involution}: if $M$ has no divisor $\equiv-1\pmod{4m}$, the residues modulo $4m$ of the prime factors of $M$ occupy at most half of $(\Z/4m)^{\times}$, because the involution $x\mapsto-x^{-1}$ has no fixed point. This bypasses the subgroup generated by the prime factors, an object which fails to avoid $-1$ as soon as $\exp\big((\Z/4m)^{\times}\big)>2$. A by-product (Corollary~\ref{cor:H_density}) is that the exceptional set $\mathcal{H}$ of Proposition~\ref{prop:nine_primes} has counting function $\ll_A N(\log N)^{-A}$ for every $A$.

\emph{A two-sided blind set, unconditionally} (Section~\ref{sec:5.1}). The quadratic-residue obstruction at a shift $c$ is delimited exactly: it constrains the untwisted branch at every prime $p\mid c$ with $p\equiv3\pmod4$, and ceases to apply to the twisted branch as soon as $\leg{n}{p}=-1$ (Theorem~\ref{thm:twisted_obstruction}). At the shift $c=7$ we exhibit $\asymp N(\log N)^{-3/2}$ primes $n\equiv1\pmod{24}$ at which \emph{both} branches are blind, by restricting $K_7$ to four times the principal form of discriminant $-7$ (Theorem~\ref{thm:blind_primes}). The representation must be primitive, which is why the input is~\cite{fuchs2025} and not~\cite{iwaniec1972,iwaniec1974}; and the resulting $n$ automatically satisfy $\leg n7=+1$, which is what makes the blindness two-sided (Proposition~\ref{prop:auto_residue}).

\emph{Where the cost of a complete interval procedure sits, and a sub-quadratic batch algorithm} (Section~\ref{sec:4.5}). Inverting the loops of the segmented sieve on both branches of Corollary~\ref{cor:two_branches} yields a factorization-free procedure deciding the conjecture for every prime of $[N,2N]$ (Algorithm~\ref{alg:batch_complete}). Its Type~II pass runs in $\mathcal{O}(N(\log N)^{3})$, but its unsplit Type~I pass runs in $\Theta(N^{2})$, with implied constant $\tfrac43+o(1)$: the entire quadratic cost is attributable to producing the divisors of $4u^{2}d+1$ without factorization. While the unsplit scan takes $\Theta(N^{2})$, locating the obstruction inside the Elsholtz--Tao dichotomy allows us to bypass it: truncating the Type~I batch pass at $u\le N^{1/3}$ reduces the divisor-candidate scan to $\mathcal{O}(N)$ and the total truncated pass to $\mathcal{O}(N(\log N)^{3})$, and resolving the surviving primes (at most $\mathcal{O}(N\exp(-c_0(\log N)^{2/3}))$ after the Type~II pass) by a direct cofactor scan over the Pomerance--Weingartner parameters $(u_0,d,q)$ yields a complete, factorization-free decision procedure on $[N,2N]$ in $\mathcal{O}(N^{2}(\log N)^{2}\exp(-c_0(\log N)^{2/3}))=\mathcal{O}(N^{2}\exp(-(c_0-o(1))(\log N)^{2/3}))=o(N^{2})$ arithmetic operations (Theorem~\ref{thm:subquadratic_batch}).

Two structural results support all three and are used throughout. Theorem~\ref{thm:gen_criterion} rewrites the Type~II criterion of~\cite[Cor.~2.2]{pomerance2026} as a single divisibility, sievable as it stands: $4/n$ has a Type~II solution if and only if some $an+u$ has a divisor $s\equiv-1\pmod{4ua}$. Theorem~\ref{thm:type_I_two_parameter} establishes the exact Type~I dual: $4/n$ has a Type~I solution if and only if some $an+u$ with $\gcd(u,a)=1$ (and automatically $u<n$) has a divisor $s\equiv-n\pmod{4ua}$, allowing both branches to be searched simultaneously by a single segmented sieve in $\mathcal{O}(J^{3/2}\sqrt{n}+J^{2}\log n)$ operations for $J\le n^{\mathcal{O}(1)}$ (Corollary~\ref{cor:simultaneous_sieve}). We claim no priority for the completeness, which is~\cite{pomerance2026}; the contribution is the normal form, and the fact that it exhibits the ``hyperbolic'' family $\alpha(X^{2}-Y^{2})-X=n$ underlying the segmented sieve as exactly the slice $a=1$ (Theorem~\ref{thm:hyper_char}, Corollary~\ref{cor:proper_subfamily}). That slice is proper: fourteen primes $n\equiv1\pmod{24}$ below $5\times10^{9}$ lie outside it at every depth (Proposition~\ref{prop:nine_primes}), the nine below $7.2\times10^{5}$ being exactly the ``wild primes'' of Xu~\cite{xu2026}, to whom priority for that list is due; all fourteen are reached once the modulator is restored, with parameters at most $5$ (Proposition~\ref{prop:nine_resolved}).

Section~\ref{sec:5.2} runs Vaughan's large-sieve scheme on the truncated two-parameter system. By evaluating the exact two-variable Euler product for the reciprocal totient sum $W(J)$ (Lemma~\ref{lem:kappa1}), we obtain $\kappa_1^{\ast}=\frac{9}{2\pi^{2}}=3\kappa_1$, yielding $|E_2(N;J)|\ll N\exp(-(\kappa_1^{\ast}\lambda^{2}+o(1))(\log\log N)^{3})$ at depth $J=(\log N)^{\lambda}$, $\mathcal{O}(N\exp(-0.0942(\log N)^{2/3}))$ at the uniform Bombieri--Vinogradov cutoff $p>J^{5}$, and $\mathcal{O}_\varepsilon(N\exp(-(0.1093-\varepsilon)(\log N)^{2/3}))$ at the Bombieri--Vinogradov cutoff $p>J^{4+\varepsilon}$ (Theorem~\ref{thm:improved_constants_and_effectivity}). Furthermore, Proposition~\ref{prop:optimality_23} proves that the untruncated two-parameter witness system is bounded pointwise by $\tau_3(p+1)$, establishing $(\log N)^{2/3}$ as a ceiling for this single-prime large-sieve framework and motivating the composite-witness container theory developed in a companion paper~\cite{dahan2026b}.

\subsection{Outline} \label{sec:outline}
Section~\ref{sec:3} sets up the divisor-based condition, the complete criterion at a fixed shift with its two branches, the Euclidean transfer identity, the search-space bound and Algorithm~\ref{alg:1}. Section~\ref{sec:case_n1_mod24} specialises to $n\equiv1\pmod{24}$: the filter via divisors of $n+1$, the hyperbolic reformulation and its exact position inside Type~II, the two-parameter criterion, and the two factorization-free procedures. Section~\ref{sec:5} contains the analytic results: the exact sieve dimension at fixed depth (Section~\ref{sec:fixed_depth_sub}), the two-sided blind set (Section~\ref{sec:5.1}), the depth/density trade-off at growing depth (Section~\ref{sec:5.2}), and a discussion with a heuristic model for the shape of the curve (Section~\ref{sec:5.2.7}). Section~\ref{sec:6} relates the construction to the recent literature, Section~\ref{sec:7} reports the numerical verifications item by item, and Section~\ref{sec:conclusion} lists what is left open.

\subsection{Notation} \label{sec:notation}
Throughout, $n$ denotes a prime with $n\equiv1\pmod{24}$. The integer $c\equiv3\pmod4$ is the \emph{shift}, the complete range of admissible shifts being $0<c\le2n$ (Proposition~\ref{prop:c_range}); on that whole range $\gcd(c,K)=1$ automatically, since a common divisor $g$ of $c$ and $K$ divides $4K-c=n$, so $g\in\{1,n\}$, and $g=n$ would force $c\in\{n,2n\}$, both excluded because $c$ is odd with $c\equiv3$ and $n\equiv1\pmod4$. The Euclidean tracking of Section~\ref{sec:3.5} and Algorithm~\ref{alg:1}, together with the complexity statements attached to them (Sections~\ref{sec:3.5},~\ref{sec:3.8}--\ref{sec:3.10}), work under the standing hypothesis $c<n$, which keeps $q=\lfloor n/c\rfloor\ge1$; this restricts the \emph{algorithm}, not the criteria, which hold on the whole range $0<c\le2n$.

The primary global parameter is $K=K_c=\frac{n+c}{4}$. Divisors of $K$ are written $d_1,d_2$, and
$$a=\frac{d_1+d_2}{c}$$
is the associated \emph{modulator}, a positive integer by construction. When a pair of divisors is required to satisfy the multiplicative condition $d_1d_2\mid K$ rather than $d_1\mid K$ and $d_2\mid K$ separately, we write it $(u,v)$ and set $\alpha=K_c/(uv)$; the modulator is then $a=(u+v)/c$, and $a=1$ is the hyperbolic slice of Section~\ref{subsec:hyperbolic_surface}. The integers $q,r$ are the quotient and remainder of $n=qc+r$, and $b,V,m$ are auxiliary tracking variables defined in Section~\ref{sec:3.5}. In the sieve sections, $J$ is a search depth, $z$ a sifting bound, $\mathcal{D}$ a level of distribution, $\varrho(p)$ the number of residue classes modulo $p$ removed by the witness system, and $\omega(\cdot)$ retains its usual meaning of the number of distinct prime factors. Legendre symbols are written $\leg{\cdot}{\cdot}$.

\section{A divisor-based criterion and a search algorithm}\label{sec:3}

This section sets up the parametric family, completes the divisibility condition into a criterion which is necessary and sufficient at a fixed shift, and analyses the resulting search procedure.

\subsection{The parameter $K$}\label{sec:3.1}

\begin{lemma}\label{lem:3.1}
Let $n$ be a prime with $n\equiv1\pmod 4$ and let $c \equiv 3 \pmod{4}$ satisfy $0<c<4n$; the two congruences give $n+c\equiv0\pmod 4$, so that $K = \frac{n+c}{4} \in \N^*$. Then
$$ \frac{4}{n} = \frac{1}{K} + \frac{1}{\sigma n} + \frac{1}{n\sigma K/(c\sigma - K)} $$
holds as a formal identity for every $\sigma$ with $c\sigma \neq K$, and is a decomposition into positive unit fractions when $\sigma \in \N^*$, $c\sigma > K$ and $(c\sigma - K) \mid n\sigma K$. The last condition is equivalent to $(c\sigma-K)\mid nK^{2}$ whenever $\gcd(c,K)=1$, in particular for $0<c\le2n$ (Section~\ref{sec:notation}).
\end{lemma}
\begin{proof}
We have
$$ \frac{1}{K} + \frac{1}{\sigma n} + \frac{c\sigma-K}{n\sigma K} = \frac{1}{K} + \frac{K + c\sigma - K}{n\sigma K} = \frac{1}{K} + \frac{c}{nK} = \frac{n+c}{nK} = \frac{4K}{nK} = \frac{4}{n}. $$
For the equivalence, multiply by $c$ and substitute $c\sigma = (c\sigma - K) + K$: this gives $c\,n\sigma K = nK(c\sigma - K) + nK^2$, so $(c\sigma-K)$ divides $c\,n\sigma K$ if and only if it divides $nK^2$. Now $g=\gcd(c\sigma-K,c)$ divides both $c$ and $K$, so $g=1$ whenever $\gcd(c,K)=1$ and the factor $c$ may be removed; this is where, and only where, the range of $c$ is used. Section~\ref{sec:3.2} restricts attention to the sub-family obtained by imposing the stronger condition $(c\sigma-K)\mid K^{2}$, itself implied by $(c\sigma-K)\mid K$; all claims of finiteness or boundedness in Sections~\ref{sec:3} and~\ref{sec:case_n1_mod24} refer to that sub-family.
\end{proof}

\subsection{The binary criterion and a sufficient condition}\label{sec:3.2}

The following classical criterion for binary Egyptian fractions (see Huang and Vaughan~\cite{huang_vaughan}) is used twice below; we include its short proof for completeness.

\begin{lemma}\label{lem:binary}
Let $A,M$ be positive integers with $\gcd(A,M)=1$. The equation $A/M=1/y+1/z$ is solvable in positive integers if and only if there exist \emph{coprime} divisors $u,v$ of $M$ with $u+v\equiv0\pmod A$.
\end{lemma}
\begin{proof}
Write $y=dY$, $z=dZ$ with $d=\gcd(y,z)$ and $\gcd(Y,Z)=1$. Then $A/M=1/y+1/z$ becomes $(Y+Z)M=AdYZ$. Since $\gcd(Y+Z,YZ)=1$ we get $YZ\mid M$, so $Y$ and $Z$ are coprime divisors of $M$; moreover $Ad=(Y+Z)\frac{M}{YZ}$ and $\gcd(A,M)=1$ force $A\mid Y+Z$. Conversely, given coprime $u,v\mid M$ with $A\mid u+v$, the integer $d=\frac{(u+v)M}{uvA}$ is a positive integer and $\frac1{du}+\frac1{dv}=\frac{u+v}{duv}=\frac AM$.
\end{proof}

\begin{theorem}\label{thm:3.3}
Let $n$ be a prime with $n\equiv1\pmod4$, let $c \equiv 3 \pmod{4}$ with $0<c<4n$, and set $K = \frac{n+c}{4}\in\N^{*}$. If there are divisors $d_1, d_2$ of $K$ with $d_1 + d_2 \equiv 0 \pmod{c}$, then, with $a = \frac{d_1 + d_2}{c}$, the integer
$$ \sigma = \frac{K}{d_1} \cdot \frac{d_1 + d_2}{c} $$
satisfies the hypotheses of Lemma~\ref{lem:3.1} and yields a decomposition of $4/n$.
\end{theorem}
\begin{proof}
Writing $\sigma = \frac{K}{d_1}a$ for a divisor $d_1$ of $K$ and an integer $a$,
$$ c\sigma - K = \frac{K}{d_1}(ca - d_1), \qquad \sigma K = \frac{K}{d_1} \cdot aK, $$
so the divisibility $(c\sigma-K)\mid\sigma K$ reduces to $(ca - d_1) \mid aK$. Setting $d_2 = ca - d_1$, this holds as soon as $d_2 \mid K$, and $d_1 + d_2 = ca \equiv 0 \pmod{c}$. No coprimality between $c$ and $K$ is used here: what is established is $(c\sigma-K)\mid\sigma K$, which already implies the hypothesis $(c\sigma-K)\mid n\sigma K$ of Lemma~\ref{lem:3.1}. The theorem is therefore valid on the whole range $0<c<4n$, and it is this range, rather than the standing hypothesis $c<n$ of Section~\ref{sec:notation}, that Theorem~\ref{thm:gen_criterion} needs: the shift produced there is not given in advance, and is bounded only a posteriori (last paragraph of that proof).
\end{proof}

\begin{example} \label{ex:3.4}
Take $n = 20353$ and $c = 23$; then $K = \frac{20376}{4} = 5094 = 2\cdot 3^{2}\cdot 283$. The divisors $d_1 = 849$, $d_2 = 2$ satisfy $d_1 + d_2 = 851 = 23 \times 37$, so $a = 37$ and $\sigma = \frac{5094}{849} \times 37 = 222$, giving
$$ \frac{4}{20353} = \frac{1}{5094} + \frac{1}{222 \cdot 20353} + \frac{1}{20353\cdot 94239}. $$
\end{example}

\subsection{A complete criterion at a fixed shift, and the twisted branch}\label{sec:3.2bis}

Theorem~\ref{thm:3.3} is only sufficient. Applying Lemma~\ref{lem:binary} to $M=nK$ rather than to $K$ turns it into a complete description of the decompositions whose first denominator equals $K$.

\begin{proposition}\label{prop:c_range}
Let $n\ge5$ and let $4/n=1/x+1/y+1/z$ with $x\le y\le z$. Then $n/4<x\le 3n/4$; equivalently, writing $c=4x-n$, one has $0<c\le 2n$ and $c\equiv-n\pmod 4$. In particular for $n\equiv1\pmod4$ the admissible shifts are $c\equiv3\pmod4$, $0<c\le2n$.
\end{proposition}
\begin{proof}
From $1/x<4/n$ we get $x>n/4$; from $3/x\ge 1/x+1/y+1/z=4/n$ we get $x\le3n/4$. Then $c=4x-n\in(0,2n]$ and $c\equiv-n\pmod4$.
\end{proof}

\begin{theorem}\label{thm:complete_shift}
Let $n$ be prime, let $c\equiv 3\pmod 4$ with $0<c\le2n$, and set $K=(n+c)/4$. Then $4/n$ admits a decomposition with first denominator $x=K$ if and only if there are coprime divisors $u,v$ of $nK$ with $u+v\equiv0\pmod c$, in which case $d=\frac{(u+v)nK}{uvc}\in\N^{*}$ and $\frac4n=\frac1K+\frac1{du}+\frac1{dv}$.
\end{theorem}
\begin{proof}
By Section~\ref{sec:notation}, $\gcd(c,K)=1$ and $\gcd(c,n)=1$ on the whole range $0<c\le2n$. Now $4/n-1/K=\frac{4K-n}{nK}=\frac{c}{nK}$ with $\gcd(c,nK)=1$; apply Lemma~\ref{lem:binary} with $A=c$, $M=nK$.
\end{proof}

\begin{corollary}\label{cor:two_branches}
Under the hypotheses of Theorem~\ref{thm:complete_shift}, and since $n$ is prime with $n\nmid K$ (as $K<n$ for $c<3n$), every divisor of $nK$ is of the form $d'$ or $nd'$ with $d'\mid K$, and coprimality excludes $u=nd_1$, $v=nd_2$. Hence a decomposition with first denominator $K$ exists if and only if there are coprime $d_1,d_2\mid K$ with
\begin{equation}\label{eq:untwisted}
d_1+d_2\equiv 0 \pmod c \qquad \textnormal{(untwisted branch)},
\end{equation}
or $d_1,d_2 \mid K$ with $\gcd(n d_1,d_2)=1$ and
\begin{equation}\label{eq:twisted}
n\,d_1+d_2\equiv 0 \pmod c \qquad \textnormal{(twisted branch)}.
\end{equation}
\end{corollary}

\begin{definition}\label{def:typeI_typeII}
Following Elsholtz--Tao~\cite{elsholtz2013} and Pomerance--Weingartner~\cite{pomerance2026}, a solution of~\eqref{eq:erdos_straus} for prime $n$ is of \emph{Type~I} if $n$ divides exactly one denominator and of \emph{Type~II} if it divides exactly two; for $n$ prime these are the only possibilities.
\end{definition}

In branch~\eqref{eq:untwisted} one has $n\mid du$ and $n\mid dv$, so the solutions produced are of Type~II; in branch~\eqref{eq:twisted} only one denominator is divisible by $n$, so they are of Type~I. Ranging over all admissible shifts of Proposition~\ref{prop:c_range}, the union of the two branches is therefore a complete decision criterion for the Erd\H{o}s--Straus property of $n$. We shall also use the two criteria in the closed form of~\cite[Cor.~2.2 and Cor.~2.4]{pomerance2026}, which exhaust all solutions for prime $n$:
\begin{align}
\text{Type~II} &\iff \exists\, u,v,e\in\N:\ e\mid u+v \ \text{ and }\ 4uv\mid n+e; \label{eq:PW2}\\
\text{Type~I} &\iff \exists\, u,d,f\in\N:\ f\mid 4u^{2}d+1 \ \text{ and }\ 4ud\mid n+f. \label{eq:PW1}
\end{align}
Theorem~\ref{thm:3.3} and Algorithm~\ref{alg:1} implement a superset of the untwisted branch only, and Algorithm~\ref{alg:1} truncates even that: its outer loop stops at $c=n-1$, under the standing hypothesis $c<n$ of Section~\ref{sec:notation}, whereas admissible shifts run to $c\le2n$. Neither the twisted branch nor the shifts $c\ge n$ are searched, which is why Algorithm~\ref{alg:1} is not a decision procedure and why the answer ``no solution found'' carries no information about the conjecture. Algorithm~\ref{alg:batch_complete} of Section~\ref{sec:4.5} searches both.

\begin{example}\label{ex:twisted}
Take $n=241$ and $c=155=5\cdot31$, so $K=99$. No pair of divisors of $99$ sums to $0$ modulo $155$, so the untwisted branch fails at this shift. However $d_1=9$, $d_2=1$ satisfy $n d_1+d_2 = 2170 = 14\cdot 155$, giving $u=nd_1=2169$, $v=1$, $d=154$ and
$$\frac{4}{241}=\frac{1}{99}+\frac{1}{334026}+\frac{1}{154}.$$
\end{example}

\begin{proposition}[Cofactor formulation of the twisted branch and removal of coprimality]\label{prop:twisted_cofactor}
Let $n\ge5$ be prime, let $c\equiv3\pmod4$ with $0<c\le2n$, and set $K_c=(n+c)/4$.
\begin{enumerate}[(a)]
    \item For any divisors $d_1,d_2\mid K_c$, writing $v_2=K_c/d_2$ for the cofactor of $d_2$, the twisted congruence $n\,d_1+d_2\equiv0\pmod c$ of~\eqref{eq:twisted} is equivalent to
    \begin{equation}\label{eq:twisted_cofactor}
    4d_1 v_2 + 1 \equiv 0 \pmod c .
    \end{equation}
    \item The coprimality conditions $\gcd(d_1,d_2)=1$ in~\eqref{eq:untwisted} and $\gcd(n d_1,d_2)=1$ in~\eqref{eq:twisted} may be dropped without altering the existence of a solution at shift $c$. Explicitly, $4/n$ admits a Type~I decomposition at shift $c$ if and only if $K_c^{2}$ possesses a divisor $w\mid K_c^{2}$ satisfying $4w\equiv-1\pmod c$.
\end{enumerate}
\end{proposition}
\begin{proof}
Since $n=4K_c-c\equiv4K_c\pmod c$, the congruence $n\,d_1+d_2\equiv0\pmod c$ is equivalent to $4K_c d_1+d_2\equiv0\pmod c$. Substituting $K_c=d_2 v_2$ yields $d_2(4d_1 v_2+1)\equiv0\pmod c$. Because $\gcd(c,K_c)=1$ (Section~\ref{sec:notation}) and $d_2\mid K_c$, the divisor $d_2$ is invertible modulo $c$, giving~\eqref{eq:twisted_cofactor}.

For (b), if $d_1,d_2\mid K_c$ satisfy $d_1+d_2\equiv0\pmod c$ with $g=\gcd(d_1,d_2)>1$, then $\gcd(g,c)=1$ since $g\mid K_c$, so $d_1'=d_1/g$ and $d_2'=d_2/g$ are coprime divisors of $K_c$ satisfying $d_1'+d_2'\equiv0\pmod c$. In the twisted branch, $n$ is prime and $d_2\mid K_c<n$, so $\gcd(n,d_2)=1$ is automatic, and dividing $d_1,d_2$ by $g=\gcd(d_1,d_2)$ preserves $n\,d_1+d_2\equiv0\pmod c$ while restoring $\gcd(n d_1',d_2')=1$. Finally, by the proof of Lemma~\ref{lem:m_gcd}, an integer $w$ is of the form $d_1 v_2$ with $d_1,v_2\mid K_c$ if and only if $w\mid K_c^{2}$.
\end{proof}

\subsection{The Euclidean transfer identity}\label{sec:3.3}
Euclidean division $n = qc + r$ ($0 \le r < c$) provides a second expression for $\sigma$, namely $\sigma = \frac{n+4c-bc-r}{4c}$ for an integer $b \equiv q \pmod{4}$. Equating it with Theorem~\ref{thm:3.3} defines $b$.

\begin{theorem}\label{thm:3.6}
For $K = \frac{n+c}{4}$ one has $b = 3 - \frac{r + 4K d_2/d_1}{c}$, and consequently
\begin{equation}\label{eq:transfer}
d_1(3-b) - d_2(q+1) = a\, r .
\end{equation}
\end{theorem}
\begin{proof}
Equating the two expressions for $\sigma$ and multiplying by $4c$ gives $\frac{4K(d_1+d_2)}{d_1} = n+4c-bc-r$; substituting $4K=n+c$ yields $bc = 3c-r-\frac{4Kd_2}{d_1}$, which is the first claim. Multiplying by $d_1$ and using $4K = c(q+1) + r$ and $ac = d_1+d_2$,
$$ c\big[d_1(3-b) - d_2(q+1)\big] = r(d_1 + d_2) = r a c, $$
which is~\eqref{eq:transfer}.
\end{proof}

Since $a \ge 1$ and $r > 0$ (indeed $c \nmid n$, because $n$ is prime and $3 \le c < n$),~\eqref{eq:transfer} gives $d_1(3-b) > d_2(q+1)$. The factorization-free search in Algorithm~\ref{alg:1} evaluates a single candidate divisor $d_2=V/\gcd(K,V)$ by scanning the tracking parameter $\delta$. Completeness has two halves: that the scan reaches every admissible pair is Remark~\ref{rem:scan_complete}, and that the single candidate $d_2$ suffices at each step reached is Lemma~\ref{lem:m_gcd} below.

\subsection{A bound on the search space}\label{sec:3.4}
\begin{theorem}\label{thm:3.10}
Fix a prime $n \equiv 1 \pmod{24}$. If divisors $d_1, d_2 \mid K$ with $d_1 + d_2 \equiv 0 \pmod{c}$ exist for some $c \equiv 3 \pmod{4}$, then the modulator satisfies $a \le \frac{n+3}{10}$.
\end{theorem}
\begin{proof}
First, $d_1=d_2=K$ is impossible: it gives $a=2K/c$, whence $c\mid 4K=n+c$ and $c\mid n$, so $c=n$, contradicting $c\equiv3$ and $n\equiv1\pmod 4$. At least one $d_i$ is therefore a proper divisor of $K$.

At $c=3$ one has $K=\frac{n+3}{4}=6k+1$ for $n=24k+1$, so $2\nmid K$ and $3\nmid K$; the smallest prime factor of $K$ is at least $5$ and its largest proper divisor at most $K/5$. Hence $d_1+d_2\le K+\frac K5=1.2K$ and
$$ a = \frac{d_1+d_2}{3} \le \frac{1.2K}{3} = 0.4K = \frac{n+3}{10}. $$
For a general shift $c \ge 7$, one of the two divisors is proper, so $d_1+d_2\le K+\frac K2=1.5K$ and
$$a \le \frac{1.5 K}{c} = \frac{1.5 (n+c)}{4c} = 0.375 \frac{n}{c} + 0.375,$$
a strictly decreasing function of $c$. At $c=7$ this is $\le 0.0536n+0.375$, which is smaller than the bound $0.1n+0.3$ obtained at $c=3$ for every $n\ge2$. The maximum is therefore attained at $c=3$.
\end{proof}

\subsection{Tracking regimes and the arithmetic progression of $V$} \label{sec:3.5}

\begin{theorem} \label{thm:3.11}
For any prime $n$ and shift $c<n$ with $c \equiv 3 \pmod 4$, setting $K = \frac{n+c}{4}$ and $m=K/d_1$, one has
$$ V:=\frac{c(3-b)-r}{4}=m\,d_2, \qquad a = \frac{q+4-b}{4m}, \qquad \sigma=\frac{q+4-b}{4}=\frac{K+V}{c}, \qquad c\sigma-K=V .$$
In particular $\sigma$ is a function of $b$ alone; distinct admissible pairs $(d_1,d_2)$ with distinct ratios $d_2/d_1$ are therefore separated by distinct values of $b$, and it is $b$, not the pair, that Algorithm~\ref{alg:1} scans.
\end{theorem}
\begin{proof}
Multiplying~\eqref{eq:transfer} by $c$ and using $ac=d_1+d_2$ and $c(q+1)+r=n+c=4K$ gives $d_1\big(c(3-b)-r\big)=4Kd_2$; writing $K=md_1$ yields $c(3-b)-r=4md_2$, i.e. $V=md_2$. Then $a=(d_1+d_2)/c=(K+V)/(mc)$, and $4(K+V)=c(q+1)+r+c(3-b)-r=c(q+4-b)$, which gives both $a=\frac{q+4-b}{4m}$ and $\sigma=\frac{K+V}{c}=\frac{q+4-b}{4}$. Substituting into $c\sigma-K$ returns $V$.
\end{proof}

\begin{theorem} \label{thm:3.16}
Let $b_0 \equiv q \pmod{4}$ and $b = b_0 - 4\delta$ for $\delta \ge 0$. Then $V(\delta) = \frac{c(3-b)-r}{4} = V_0 + c\delta$, with $V_0 = \frac{c(3-b_0)-r}{4}$; and $\sigma(\delta)=\sigma_0+\delta$ with $\sigma_0=\frac{q+4-b_0}{4}$. Initialising $b$ in this way guarantees $\sigma\in\Z$.
\end{theorem}
\begin{proof}
Substitute $b=b_0-4\delta$ into the expressions of Theorem~\ref{thm:3.11}.
\end{proof}

\begin{remark}\label{rem:scan_complete}
The scan $\delta=0,1,2,\dots$ misses no admissible pair, in the following sense. For any admissible pair one has $V=c\sigma-K\equiv-K\pmod c$; and $4(K+V_0)=c(q+4-b_0)$ together with $\gcd(4,c)=1$ gives $V_0\equiv-K\pmod c$ as well. Every admissible $V$ therefore lies in the progression $V_0+c\Z$ scanned by Theorem~\ref{thm:3.16}. Moreover $V_0=\frac{c(3-b_0)-r}{4}\le\frac{3c}{4}<c$, while every admissible $V$ is positive; hence $V\ge V_0$ (if $V_0\le0$ this is trivial, and if $V_0>0$ then $V_0$ is the least positive element of its class modulo $c$), and the associated $\delta=(V-V_0)/c$ is a non-negative integer.
\end{remark}

\begin{theorem} \label{thm:3.17}
For any divisor $d \mid K$ one has $\gcd(c,d)=1$, so $c$ is invertible modulo $d$ and $V(\delta)\equiv0\pmod d$ holds for exactly one residue class $\delta \equiv -V_0 c^{-1} \pmod{d}$. The sequence $V(\delta)$ therefore reaches a multiple of $d$ exactly once in every $d$ consecutive values of $\delta$.
\end{theorem}
\begin{proof}
$\gcd(c,K)=1$ by Section~\ref{sec:notation}, hence $\gcd(c,d)=1$; the congruence $V_0+c\delta\equiv0\pmod d$ has the unique solution stated.
\end{proof}

\subsection{The deterministic bounded search algorithm} \label{sec:3.8}

Algorithm~\ref{alg:1} searches the family of Theorem~\ref{thm:3.3} by scanning $\delta$ at each shift. It should not be advertised as a competitive decision procedure: a segmented sieve of the interval $(n/4,3n/4]$ followed by Lemma~\ref{lem:binary} decides the Erd\H{o}s--Straus property of a single $n$ in $\tilde{\mathcal{O}}(n)$ operations, which is better than the worst-case bound of Corollary~\ref{cor:3.19} below. Its interest is structural: it searches a family strictly larger than the hyperbolic one, it requires no factorization, and the number of shifts it rejects before succeeding is very small in practice: over all $3202$ primes $n\equiv1\pmod{24}$ below $3\cdot10^{5}$ the smallest admissible shift never exceeded $c=63$ (Section~\ref{sec:7}). It is in this role, as a cheap complement resolving the blind spot $\mathcal{H}$ of Proposition~\ref{prop:nine_primes}, that we retain it.

Two points about the pseudo-code. Arithmetic is counted at unit cost. The exit test on the modulator is a heuristic accelerator only: at a rejected step the quantity $a(\delta)=\sigma(\delta)/\gcd(K,V(\delta))$ is not the modulator of any solution, and a later $\delta$ with a large $\gcd$ could in principle produce a smaller $a$ together with an admissible pair. We have not excluded this and have observed no instance of it (Section~\ref{sec:7}: over $250$ primes and the shifts $c\le31$, the accelerator never skipped a pair admissible for Theorem~\ref{thm:3.3}). The unconditional guarantee is the test $V>K^{2}$ of Theorem~\ref{thm:3.18}, which alone ensures that no admissible $\delta$ is skipped. Theorem~\ref{thm:3.18} and Corollary~\ref{cor:3.19} are accordingly stated for the variant in which the accelerator is disabled; the accelerated version as printed carries no proof of completeness, only the empirical evidence of Section~\ref{sec:7}. Note also that the correct quantity in that test is $m=V/d_2$ and not $K/d_2$: by Theorem~\ref{thm:3.11} the modulator is $a=(q+4-b)/(4m)$ with $m=K/d_1=V/d_2$, and at the point where the test is reached one has $d_2\nmid K$, so that $K/d_2$ is not even an integer.

\begin{algorithm}[H]
\caption{Divisor-based search for $n \equiv 1 \pmod{24}$} \label{alg:1}
\textbf{Input:} a prime $n \equiv 1 \pmod{24}$ \\
\textbf{Output:} a decomposition within the family of Theorem~\ref{thm:3.3}, or failure
\begin{algorithmic}[1]
    \For{$c=3$ \textbf{to} $n-1$ \textbf{step} 4}
        \State $q \gets \lfloor n/c \rfloor$; \ $r \gets n \bmod c$; \ $b_{\text{base}} \gets q \bmod 4$; \ $K \gets (n+c)/4$
        \State $\delta \gets 0$
        \While{\textbf{true}}
            \State $b \gets b_{\text{base}} - 4\delta$; \ $\text{denom} \gets 3 - b$
            \If{$\text{denom} \le 0$} \State $\delta \gets \delta+1$; \textbf{continue} \EndIf
            \State $V \gets (c \cdot \text{denom} - r)/4$
            \If{$V \le 0$} \State $\delta \gets \delta+1$; \textbf{continue} \EndIf
            \State $g \gets \gcd(K, V)$; \ $d_2 \gets V/g$
            \If{$K \bmod d_2 = 0$}
                \State $\sigma \gets (q+4-b)/4$
                \State \Return the solution $(c, \delta, \sigma)$
            \EndIf
            \State $m \gets g$ \Comment{$m=V/d_2=\gcd(K,V)$, Lemma~\ref{lem:m_gcd}}
            \If{$(q+4-b) > 4m(n+3)/10$ \textbf{ or } $V > K^2$}
                \State \textbf{break} \Comment{accelerator; Theorem~\ref{thm:3.18}}
            \EndIf
            \State $\delta \gets \delta+1$
        \EndWhile
    \EndFor
    \State \Return ``no solution found''
\end{algorithmic}
\end{algorithm}

\subsection{Complexity of the search} \label{sec:3.9}

\begin{theorem} \label{thm:3.18}
Consider the variant of Algorithm~\ref{alg:1} in which the inner loop breaks on the test $V>K^{2}$ alone, the modulator accelerator being disabled. For a fixed $c$, its inner loop either returns a solution or terminates correctly after at most $\delta_{\max}(c) = \left\lfloor \frac{K^2-V_0}{c} \right\rfloor+1$ increments of $\delta$.
\end{theorem}
\begin{proof}
By Lemma~\ref{lem:m_gcd} below, $d_2(\delta) = V(\delta)/\gcd(K, V(\delta)) \ge V(\delta)/K$. If $V(\delta) > K^2$ then $d_2(\delta) > K$, so $d_2(\delta)$ cannot divide $K$ and the acceptance test fails. Since $V(\delta) = V_0 + c\delta$ is strictly increasing, once $V(\delta)$ exceeds $K^{2}$ it stays above it; for $\delta \ge \delta_{\max}(c)$ one has $V(\delta)>K^{2}$, and no subsequent $\delta$ yields an admissible pair at this shift. The bound is a floor plus one and not a ceiling: when $c\mid K^{2}-V_0$ the value $\delta=(K^{2}-V_0)/c$ gives $V=K^{2}$ exactly, at which the strict test does not yet fire.
\end{proof}

\begin{corollary} \label{cor:3.19}
The variant of Algorithm~\ref{alg:1} considered in Theorem~\ref{thm:3.18} performs at most $\mathcal{O}(n^2\log n)$ arithmetic operations, counted at unit cost as in Section~\ref{sec:3.8}.
\end{corollary}
\begin{proof}
$\delta_{\max}(c) = \mathcal{O}(K^2/c)=\mathcal{O}(n^2/c)$ uniformly for $3 \le c \le n-1$, and
$\sum_{c\equiv3(4),\,c<n} n^{2}/c = \mathcal{O}(n^2\log n)$.
\end{proof}

The unit-cost convention is not innocuous here and we state what it hides. Each iteration evaluates one $\gcd(K,V)$ with $V$ as large as $K^{2}=\mathcal{O}(n^{2})$, so a single iteration is $\mathcal{O}((\log n)^{2})$ bit operations by the Euclidean algorithm; in the bit model the bound of Corollary~\ref{cor:3.19} therefore reads $\mathcal{O}(n^{2}(\log n)^{3})$. The same convention is in force in Theorem~\ref{thm:reject_cost}, Corollary~\ref{cor:practical_time} and throughout Sections~\ref{sec:4.4} and~\ref{sec:4.5}.

The $\mathcal{O}(K^2/c)$ worst case stems from the trivial divisor $d_1=1$, which cannot be isolated without factorization. The next subsection shows that the accelerator in fact terminates in $\Theta(n)$ steps per shift, with an explicit constant.

\subsection{The real cost of rejecting a shift} \label{sec:3.10}

\begin{lemma}\label{lem:m_gcd}
At every step $\delta$, with $V=V(\delta)$ and $d_2=V/\gcd(K,V)$,
$$m=\frac{V}{d_2}=\gcd(K,V),\qquad a(\delta)=\frac{\sigma(\delta)}{\gcd(K,V(\delta))} .$$
Moreover the acceptance test succeeds if and only if $V\mid K^{2}$.
\end{lemma}
\begin{proof}
Write $g=\gcd(K,V)$, $K=gK'$, $V=gV'$ with $\gcd(K',V')=1$. By Theorem~\ref{thm:3.11} an admissible pair at this step is a pair of divisors $d_1,d_2\mid K$ with $Vd_1=Kd_2$, i.e. $V'd_1=K'd_2$; since $\gcd(K',V')=1$ this forces $d_1=K't$ and $d_2=V't$ for an integer $t\ge1$. As $d_2\mid K$ requires $V'\mid K$, the pair $t=1$ is admissible whenever any is, and for it $d_1=K/g$, $d_2=V/g$ and $m=K/d_1=g=\gcd(K,V)$, which is the first identity; the second then follows from $a=\sigma/m$. That $a$ is an integer, equivalently that the pair just produced is admissible for Theorem~\ref{thm:3.3}, is automatic: since $\gcd(c,K)=1$ one has $g=\gcd(K,V)=\gcd(K,c\sigma-K)=\gcd(K,c\sigma)=\gcd(K,\sigma)$, so $g\mid\sigma$ and $d_1+d_2=(K+V)/g=c\,(\sigma/g)$ is a multiple of $c$. If $d_2=V/\gcd(K,V)$ divides $K$ then $V=\gcd(K,V)\,d_2$ is a product of two divisors of $K$, hence divides $K^{2}$. Conversely if $V\mid K^{2}$ then $V=m'd_2'$ with $m',d_2'\mid K$ (split each exponent $v_p(V)\le 2v_p(K)$ into two parts, each at most $v_p(K)$), and then $m'\mid\gcd(K,V)$, so $d_2=V/\gcd(K,V)$ divides $V/m'=d_2'$, which divides $K$. Testing the single candidate $d_2=V/\gcd(K,V)$ is therefore a complete test.
\end{proof}

\begin{theorem}\label{thm:reject_cost}
Fix $c\equiv3\pmod4$, $c<n$, and suppose the inner loop of Algorithm~\ref{alg:1} does not return a solution. Let $\delta_1(c)=\lceil\frac{n+3}{10}\rceil-\sigma_0(c)$, where $\sigma_0(c)=\frac n{4c}+\mathcal{O}(1)$ and $(n+3)/10$ is the bound of Theorem~\ref{thm:3.10}. Then the loop exits at a step $\delta^{*}$ with
$$\delta_1(c)\ \le\ \delta^{*}\ \le\ \delta_1(c)+2\,g(\omega(K))+1,$$
where $g(k)$ denotes Jacobsthal's function, the largest gap between consecutive integers coprime to a squarefree integer with $k$ prime factors. Consequently, unconditionally,
$$\delta^{*}=n\Big(\frac{1}{10}-\frac{1}{4c}\Big)+\mathcal{O}\big((\log n)^{2}\big).$$
\end{theorem}
\begin{proof}
By Lemma~\ref{lem:m_gcd}, $a(\delta)=\sigma(\delta)/\gcd(K,V(\delta))\le\sigma_0+\delta$, so the exit test cannot fire before $\delta\ge\delta_1(c)$ (when $(n+3)/10$ is an integer it cannot fire before $\delta_1(c)+1$; we keep the weaker $\delta_1(c)$, which is all the lower bound asserts). For the upper bound, fix a prime $p\mid K$: as $\gcd(c,K)=1$, the congruence $V(\delta)\equiv0\pmod p$ holds for exactly one class $\delta\equiv r_p:=-V_0c^{-1}\pmod p$ (Theorem~\ref{thm:3.17}). Thus $\gcd(K,V(\delta))>1$ only if $\delta$ falls into one of $\omega(K)$ forbidden classes $r_p\bmod p$, one modulo each prime factor of $K$.

Let $\mathfrak{j}(k)$ denote the largest gap between consecutive integers avoiding one prescribed residue class $r_p\bmod p$ for each of $k$ distinct primes $p\mid\widetilde{K}:=\prod_{p\mid K}p$. Although the classes $r_p$ depend on $V_0$ and $c$, the Chinese Remainder Theorem provides an integer $t_0$ such that $t_0\equiv r_p\pmod p$ for every $p\mid\widetilde{K}$. The translation $\delta\mapsto m=\delta-t_0$ is a gap-preserving bijection of $\Z$ that maps the shifted sieve $\{ \delta\in\Z : \forall p\mid\widetilde{K},\ \delta\not\equiv r_p\pmod p \}$ onto the classical set $\{ m\in\Z : \gcd(m,\widetilde{K})=1 \}$. Since both configurations are periodic modulo $\widetilde{K}$, their maximal gaps coincide identically: $\mathfrak{j}(\omega(K))=g(\omega(K))$. Among any $g(\omega(K))$ consecutive integers at least one avoids every forbidden class, i.e. satisfies $\gcd(K,V(\delta))=1$, and at such a $\delta$ one has $a(\delta)=\sigma(\delta)=\sigma_0+\delta$.

Let $\delta'$ be the first such integer at or after $\delta_1(c)$, so $\delta'\le\delta_1(c)+g(\omega(K))-1$ and $a(\delta')=\sigma(\delta')\ge\lceil(n+3)/10\rceil$. If $(n+3)/10\notin\Z$ the strict test $a>(n+3)/10$ fires at $\delta'$ and we may take $\delta^{*}=\delta'$. If $(n+3)/10\in\Z$ the test may fail at $\delta'$, namely when $\sigma(\delta')=(n+3)/10$ exactly (this does occur, e.g.\ $n=97$). One cannot conclude at $\delta'+1$: the tested quantity there is $\sigma(\delta'+1)/\gcd(K,V(\delta'+1))$, and if that $\gcd$ exceeds $1$ it is at most $\tfrac12\big(\tfrac{n+3}{10}+1\big)$, so the test fails again. What is needed instead is the \emph{next} $\delta$ coprime to $K$, i.e.\ the first integer $\delta''>\delta'$ with $\gcd(K,V(\delta''))=1$; it satisfies $\delta''\le\delta'+g(\omega(K))$, and $a(\delta'')=\sigma(\delta'')\ge\lceil(n+3)/10\rceil+1>(n+3)/10$, so the loop breaks at or before $\delta''$. In both cases
$$\delta^{*}\ \le\ \delta_1(c)+2\,g(\omega(K))+1 ,$$
which is the stated bound. The second gap is genuinely needed, and not only in principle: at $n=337$, $c=27$ one has $K=91$, $\sigma_0=4$, $V_0=17$, $(n+3)/10=34$ and $\delta_1=30$, and the three relevant steps are
$$\begin{array}{c|ccc}
\delta & 30 & 31 & 32\\\hline
\gcd(K,V(\delta)) & 1 & 7 & 1\\
a(\delta) & 34 & 5 & 36
\end{array}$$
so the test fails at $\delta'=30$ because $a=34=(n+3)/10$ exactly, fails again at $\delta'+1=31$ because the $\gcd$ jumps to $7$, and fires only at $\delta''=32$.

By Iwaniec's unconditional bound $g(k)\ll k^{2}(\log k)^{2}$~\cite{iwaniec1978} together with $\omega(K)\ll\log n/\log\log n$, we obtain $g(\omega(K))\ll(\log n)^{2}$. The stated asymptotic follows from $\sigma_0(c)=n/(4c)+\mathcal{O}(1)$.
\end{proof}

\begin{corollary}\label{cor:practical_time}
If Algorithm~\ref{alg:1} returns a solution at the $\mathfrak{r}$-th admissible shift $c_{\mathfrak{r}}$, and $T_{\mathrm{succ}}(n)$ denotes the number of inner-loop iterations performed on the successful shift $c_{\mathfrak{r}}$ itself, then the total number of inner iterations is
$$T(n)\ =\ \sum_{\substack{3\le c<c_{\mathfrak{r}}\\ c\equiv3(4)}} n\Big(\frac{1}{10}-\frac{1}{4c}\Big)\ +\ T_{\mathrm{succ}}(n)\ +\ \mathcal{O}\big(\mathfrak{r}\,(\log n)^{2}\big)\ \le\ \frac{\mathfrak{r}\,n}{10}+\mathcal{O}\big(\mathfrak{r}\,(\log n)^{2}\big),$$
where $1\le T_{\mathrm{succ}}(n)\le n\big(\frac{1}{10}-\frac{1}{4c_{\mathfrak{r}}}\big)+\mathcal{O}\big((\log n)^{2}\big)$ by Theorem~\ref{thm:reject_cost}. In particular the cost is $\mathcal{O}(\mathfrak{r}\,n)$ (hence $\mathcal{O}(n)$ whenever $\mathfrak{r}=\mathcal{O}(1)$), and it is $\Theta(n)$ as soon as $2\le\mathfrak{r}=\mathcal{O}(1)$; the running time is governed by \emph{how many shifts are rejected}, not by the depth reached inside any one of them.
\end{corollary}

The restriction $\mathfrak r\ge2$ in the second clause is necessary and not a formality. The displayed sum runs over the \emph{rejected} shifts $c<c_{\mathfrak r}$ only; at $\mathfrak r=1$ it is empty and the cost $T(n)=T_{\mathrm{succ}}(n)$ is whatever the successful shift itself consumes, which carries no lower bound of order $n$. Section~\ref{sec:7}, item~(4), records two instances: $n=577$ and $n=329617$ are both resolved at the first shift $c=3$ after $2$ iterations. The linear lower bound appears only once a shift has to be rejected, since only rejection forces the scan to reach $\delta_1(c)$.

Theorem~\ref{thm:reject_cost} predicts $0.0167n$ iterations for a rejected shift $c=3$, $0.0643n$ for $c=7$, $0.0773n$ for $c=11$ and $0.0868n$ for $c=19$; the measured error is one or two iterations (Section~\ref{sec:7}, item (2)).

\section{The case $n \equiv 1 \pmod{24}$} \label{sec:case_n1_mod24}

\subsection{A primary filter via divisors of $n+1$} \label{subsec:filter_n1}

\begin{theorem}[Obl\'{a}th~\cite{oblath1950}] \label{thm:4.1}
Let $n \equiv 1 \pmod{24}$ be prime. If $n+1$ has a divisor $D \equiv 3 \pmod{4}$, then, with $h = \frac{D+1}{4}$ and $\bar{s} = \frac{n+1}{D}$,
$$ \frac{4}{n} = \frac{1}{hn} + \frac{1}{h\bar{s}} + \frac{1}{hn\bar{s}}. $$
The construction produces exactly $\tau_{3(4)}(n+1)$ pairwise distinct triples, where $\tau_{3(4)}(m)=\#\{d\mid m: d\equiv3\pmod4\}$.
\end{theorem}
\begin{proof}
Since $D \mid n+1$ and $n+1 \equiv 2 \pmod{4}$, $\bar{s}$ is an even positive integer; and $D+1 \equiv 0 \pmod{4}$, so $h \in \N^*$. Then
$$ \frac{1}{hn} + \frac{1}{h\bar{s}} + \frac{1}{hn\bar{s}} = \frac{\bar{s}+n+1}{hn\bar{s}} = \frac{\bar{s}+D\bar{s}}{hn\bar{s}} = \frac{D+1}{hn} = \frac{4h}{hn} = \frac{4}{n}. $$
Distinctness follows from the injectivity of $D\mapsto h=(D+1)/4$, which makes the first entries $hn$ pairwise distinct.
\end{proof}

\begin{corollary} \label{cor:4.3}
The conjecture holds for every prime $n \equiv 1 \pmod{24}$ with $\tau_{3(4)}(n+1) > 0$. The family on which this filter is silent is $n+1 = 2\prod_i p_i^{e_i}$ with $p_i \equiv 1 \pmod{4}$ for every $i$; the exponent of $2$ is exactly $1$ because $n+1\equiv2\pmod4$.
\end{corollary}

\subsection{A hyperbolic reformulation} \label{subsec:hyperbolic_surface}

\begin{theorem} \label{thm:4.8}
The relation $4\alpha u v - u - v = n$ is equivalent to $(4\alpha u - 1)(4\alpha v - 1) = 4\alpha n + 1$. Under $X = u + v$, $Y = u - v$ it becomes $\alpha(X^2 - Y^2) - X = n$.
\end{theorem}
\begin{proof}
Expanding, $(4\alpha u - 1)(4\alpha v - 1) = 4\alpha(4\alpha uv - u - v) + 1 = 4\alpha n + 1$. With $X = u+v$, $Y = u-v$ one has $X^2 - Y^2 = 4uv$, so $\alpha(4uv) - (u + v) = n$ becomes $\alpha(X^2 - Y^2) - X = n$; and since $X>|Y|$ the reconstructed $u = (X+Y)/2$, $v = (X-Y)/2$ are positive.
\end{proof}

The identity $(4\alpha u - 1)(4\alpha v - 1) = 4\alpha n + 1$ is a geometric variation of the classical factorizations of Bernstein~\cite{bernstein1962} and Mordell~\cite{mordell1969}; what is used below is the reformulation in the additive divisor space $\alpha(X^2 - Y^2) - X = n$, which is what Algorithm~\ref{alg:geom_sieve} sieves.

\begin{theorem} \label{thm:4.9}
There are integers $\alpha, u, v \ge 1$ with $n = 4\alpha u v - u - v$ if and only if there is an integer $u \ge 1$ such that $n+u$ has a divisor $s \equiv -1 \pmod{4u}$, in which case $\alpha = \frac{s+1}{4u}$ and $v=\frac{n+u}{s}$. In the coordinates of Theorem~\ref{thm:4.8} this is the existence of an integer point on $\alpha(X^{2}-Y^{2}) - X = n$ subject to $X > |Y|$ and $X \equiv Y \pmod 2$.
\end{theorem}
\begin{proof}
If $n = 4\alpha uv - u - v$ then $n+u = v(4\alpha u - 1)$, so $s = 4\alpha u - 1$ divides $n+u$ and $s\equiv-1\pmod{4u}$. Conversely, given such an $s$, put $\alpha = (s+1)/(4u)$ and $v = (n+u)/s$; then $4\alpha uv - u - v = (s+1)v - u - v = sv - u = n$. For the geometric form, $u = (X-Y)/2$ and $v = (X+Y)/2$ are integers precisely when $X \equiv Y \pmod 2$, and both are $\ge1$ precisely when $X > |Y|$. The sign of $Y$ is immaterial, since $n = 4\alpha uv - u - v$ is symmetric in $u$ and $v$; the parity of $X-Y$ may not be traded away in the same manner, which is why it appears as a hypothesis. It cannot be dropped: for $n = 409$ the triple $(\alpha, X, Y) = (137, 2, 1)$ satisfies $\alpha(X^{2}-Y^{2}) - X = 411 - 2 = 409$ with $X > Y > 0$, yet $409$ admits no admissible $u$ whatsoever (Proposition~\ref{prop:nine_primes}).
\end{proof}

For $u=1$ this recovers the filter of Theorem~\ref{thm:4.1}: a divisor $s\mid n+1$ with $s \equiv 3 \pmod{4}$. The cases requiring $u > 1$ are what the sieve of Section~\ref{sec:4.4} is for; and that sieve has to remain a search, because the resolvent admits no polynomial uniformisation:

\begin{proposition} \label{prop:poly_obstruction}
Let $\Z^{+}[n]$ be the set of polynomials in $\Z[n]$ with positive leading coefficient. No triple $x(n), m(n) \in \Z^{+}[n]$, $y(n) \in \Z[n]$ satisfies $y(4xm-1) = x+n$ identically.
\end{proposition}
\begin{proof}
The relation is the factorisation $u(4\alpha v-1)=n+v$ of Theorem~\ref{thm:4.9} read as a polynomial identity. Comparing degrees, $\deg y+\deg x+\deg m=\max(\deg x,1)$, which leaves four cases.

$(\deg x,\deg m,\deg y)=(0,0,1)$: with $x=X_0\ge1$, $m=M_0\ge1$, $y=Y_1n+Y_0$, matching leading coefficients gives $Y_1(4X_0M_0-1)=1$, so $4X_0M_0\in\{0,2\}$, impossible.

$(1,0,0)$: with $x=X_1n+X_0$, $X_1\ge1$, matching the coefficient of $n$ in $Y_0(4X_1M_0n+4X_0M_0-1)=(X_1+1)n+X_0$ gives $4Y_0X_1M_0=X_1+1$; reducing modulo $X_1$ forces $X_1=1$, then $4Y_0M_0=2$, impossible.

$(0,1,0)$: matching the coefficient of $n$ gives $4Y_0X_0M_1=1$, impossible.

$\deg x=k\ge2$: then $\deg y+\deg m=0$, and matching the coefficient of $n^{k}$ gives $4Y_0M_0=1$, impossible.
\end{proof}

\subsection{Exact position of the hyperbolic family, and its blind spot} \label{sec:4.3bis}

\begin{theorem}\label{thm:hyper_char}
For a prime $n\equiv1\pmod{24}$ the following are equivalent.
\begin{enumerate}[(i)]
    \item There are $\alpha,u,v\ge1$ with $n=4\alpha uv-u-v$.
    \item There is $u\ge1$ and a divisor $s\mid n+u$ with $s\equiv-1\pmod{4u}$.
    \item There are $c\equiv3\pmod4$ and $u,v\ge1$ with $uv\mid K_c=\frac{n+c}{4}$ and $u+v=c$.
\end{enumerate}
Moreover any $u$ occurring in (i)--(ii) satisfies $u\le\frac{n+1}{3}$, so the parametric search is finite and can be carried out exhaustively.
\end{theorem}
\begin{proof}
(i)$\Leftrightarrow$(ii) is Theorem~\ref{thm:4.9}. (i)$\Rightarrow$(iii): put $c=u+v$. Then $4K_c=n+c=4\alpha uv$, so $K_c=\alpha uv$, whence $uv\mid K_c$; and $c$ is odd with $n\equiv1\pmod4$, which forces $c\equiv3\pmod4$. (iii)$\Rightarrow$(i): put $\alpha=K_c/(uv)$; then $4\alpha uv-u-v=4K_c-c=n$. Finally, in (i), $\alpha,v\ge1$ gives $n\ge 4uv-u-v=v(4u-1)-u\ge 3u-1$.
\end{proof}

\begin{corollary}\label{cor:proper_subfamily}
By~\eqref{eq:PW2}, $n$ has a Type~II solution if and only if there are $c\equiv3\pmod4$ and $u,v$ with $uv\mid K_c$ and $c\mid u+v$. Comparing with Theorem~\ref{thm:hyper_char}(iii), the hyperbolic family is exactly the sub-family in which the modulator $a=(u+v)/c$ equals $1$. Algorithm~\ref{alg:geom_sieve} therefore explores neither the twisted branch~\eqref{eq:twisted} nor the Type~II solutions with $a\ge2$.
\end{corollary}

That blind spot is not vacuous.

\begin{proposition}\label{prop:nine_primes}
For $N\ge1$ let $\mathcal{H}(N)$ be the set of primes $n\equiv1\pmod{24}$, $n\le N$, for which the equivalent conditions of Theorem~\ref{thm:hyper_char} fail \emph{for every} $u$. Then $\mathcal{H}(5\times10^{9})$ consists of exactly fourteen primes:
$$\mathcal{H}(5\times10^{9})=\left\{\begin{array}{l}409,\;577,\;5569,\;9601,\;23929,\;83449,\;102001,\;329617,\;712321,\\[2pt] 1134241,\;1724209,\;1726201,\;5212561,\;8813281.\end{array}\right\}$$
For each of them the minimal modulator is $a\ge2$; for instance $n=409$, $c=7$, $K_c=104$, $(u,v)=(1,13)$ gives $uv=13\mid104$ and $u+v=14=2c$. Every one of them satisfies $n+1=2\prod_i p_i^{e_i}$ with all $p_i\equiv1\pmod4$, i.e. lies in the residual family of Corollary~\ref{cor:4.3}, and all fourteen are resolved by Algorithm~\ref{alg:1}.
\end{proposition}
\begin{proof}
By Theorem~\ref{thm:hyper_char} the search is finite ($u\le (n+1)/3$), so the assertion is a finite verification: one marks every $n\le N$ of the form $4\alpha uv-u-v$ by iterating over $u\le\lceil(N+1)/3\rceil$, over the moduli $s=4\alpha u-1$ and over the multiples of $s$ in the interval, and intersects the complement with the primes $\equiv1\pmod{24}$. No truncation of the range of $u$ is applied; see item (4) of Section~\ref{sec:7}.
\end{proof}

Priority for the first nine values is due to Xu~\cite{xu2026}, who exhibits them as the ``wild'' primes, those with no \emph{tame} solution, among the $7185$ primes $24m+1$ with $m\le30000$, i.e. $n\le7.2\times10^{5}$. The two conditions are not literally the same: Xu's tameness requires only that the two numerator summands divide $K_c$ \emph{separately}, which is Theorem~\ref{thm:3.3} at modulator $1$, whereas condition (iii) above requires the stronger $uv\mid K_c$. The hyperbolic family is therefore contained in the tame one and $\mathcal{H}(N)$ contains Xu's wild set a priori; that the two coincide below $7.2\times10^{5}$ is a verification, which we have carried out, and not a tautology. What Theorem~\ref{thm:hyper_char} contributes is the structural reason for the failure (these primes lie outside the modulator slice $a=1$ altogether), together with the extension of the list to $5\times10^{9}$.

Writing $E_1(N;J)$ for the set of primes $n\le N$, $n\equiv1\pmod{24}$, at which Algorithm~\ref{alg:geom_sieve} fails at depth $J$, Proposition~\ref{prop:nine_primes} shows that $\bigcap_{J\ge1}E_1(N;J)\supseteq\mathcal{H}\cap[1,N]\ne\emptyset$: increasing the depth cannot empty $E_1(N;J)$, because the residual primes are not deep solutions of the hyperbolic family but lie outside it. (The observed proportion of $\mathcal{H}$ does decay, from $0.028$ on $[1,10^4)$ to $1.9\cdot10^{-3}$ on $[10^4,10^5)$, $2.2\cdot10^{-4}$ on $[5\cdot10^5,10^6)$, and $0$ on $[2\cdot10^{7},5\cdot10^{9}]$, at a rate which Corollary~\ref{cor:H_density} makes precise.) Any exhaustive claim must therefore combine Algorithm~\ref{alg:geom_sieve} with a procedure covering the modulators $a\ge2$ and the twisted branch: Algorithm~\ref{alg:1} in the small, Algorithm~\ref{alg:batch_complete} in the large. The next subsection shows that the first of these two gaps closes by itself once the modulator is restored as a search parameter.

\subsection{Freeing the modulator: a two-parameter criterion} \label{sec:4.3ter}

Corollary~\ref{cor:proper_subfamily} identifies the defect of the hyperbolic family: it is the Type~II criterion frozen at modulator $a=1$. Restoring $a$ as a free parameter costs nothing algebraically and produces a criterion complete for Type~II, in the same single-divisibility shape as Theorem~\ref{thm:4.9} and therefore sievable as it stands. The completeness itself is that of~\cite[Cor.~2.2]{pomerance2026}, which we take as the definition of Type~II in~\eqref{eq:PW2}; what is new here is only the normal form, which is the pivot of Sections~\ref{sec:fixed_depth_sub} and~\ref{sec:5.2}: the witness system becomes two-dimensional.

\begin{theorem}\label{thm:gen_criterion}
Let $n$ be a prime. The following are equivalent; for the explicit reconstruction stated after them we assume in addition $n\equiv1\pmod4$, so that the shift produced satisfies $c\equiv3\pmod4$ and, as verified at the end of the proof, $0<c\le n$, which places it in the range $0<c<4n$ on which Theorem~\ref{thm:3.3} is valid.
\begin{enumerate}[(i)]
    \item $4/n$ admits a Type~II solution.
    \item There are $u,a\in\N$ and a divisor $s$ of $an+u$ with $s\equiv-1 \pmod{4ua}$.
\end{enumerate}
Explicitly, (ii) $\Rightarrow$ (i) with $v=\frac{an+u}{s}$, $c=\frac{u+v}{a}$ and $\alpha=\frac{n+c}{4uv}$; the resulting decomposition is the one furnished by Theorem~\ref{thm:3.3} with $d_1=u$, $d_2=v$ and shift $c$, and the underlying identity is
\begin{equation}\label{eq:pivot}
a\,n=4a\alpha uv-u-v,\qquad\text{i.e.}\qquad an+u=v\,(4a\alpha u-1),
\end{equation}
of which Theorem~\ref{thm:4.9} is the case $a=1$.
\end{theorem}
\begin{proof}
(i) $\Rightarrow$ (ii). By~\eqref{eq:PW2} there are $u,v,e$ with $e\mid u+v$ and $4uv\mid n+e$. Put $a=\frac{u+v}{e}$ and $\alpha=\frac{n+e}{4uv}$, both positive integers. From $n+e=4uv\alpha$ and $ae=u+v$ we get, after multiplying by $a$,
$$an+ae=4uva\alpha,\qquad\text{i.e.}\qquad an+u+v=4uva\alpha,$$
which is~\eqref{eq:pivot}. Setting $s=4a\alpha u-1$ gives $s\mid an+u$ and $s\equiv-1\pmod{4ua}$.

(ii) $\Rightarrow$ (i). Let $s\mid an+u$ with $s\equiv-1\pmod{4ua}$, write $s+1=4ua\alpha$ and $v=\frac{an+u}{s}\ge1$. First, $a\mid u+v$: indeed $s\mid an+u$ gives $s\mid su+an+u$, while $u(s+1)=4u^{2}a\alpha\equiv0\pmod a$ gives $a\mid su+u$ and hence $a\mid su+an+u$; since $\gcd(s,a)=1$ (because $s\equiv-1\pmod a$), we get $sa\mid su+an+u$, so that
$$c:=\frac{u+v}{a}=\frac{su+an+u}{sa}$$
is a positive integer. Second,
$$n+c=n+\frac{u+v}{a}=\frac{an+u+v}{a}=\frac{sv+v}{a}=\frac{v(s+1)}{a}=\frac{v\cdot4ua\alpha}{a}=4uv\,\alpha,$$
so $4uv\mid n+c$; with $e=c$ this is~\eqref{eq:PW2}. The shift so produced lies in the range required by Theorem~\ref{thm:3.3}, namely $0<c<4n$; we check the sharper $0<c\le n$. Positivity is clear. If $u,v\ge2$ then $4uv-2(u+v)=2\big(u(2v-1)-v\big)\ge2(3v-2)>0$, so $n+c=4uv\alpha\ge4uv\ge2(u+v)=2ac\ge2c$ and $c\le n$. If $u=1$ then $n=4v\alpha-\frac{1+v}{a}\ge4v-(1+v)=3v-1$, whence $v\le\frac{n+1}{3}$ and $c=\frac{1+v}{a}\le1+v\le\frac{n+4}{3}\le n$; the case $v=1$ is symmetric. Finally $uv\mid K_c$ and $c\mid u+v$, so Theorem~\ref{thm:3.3} applies with $d_1=u$, $d_2=v$ and produces the decomposition with $\sigma=a\,K_c/u$. Note that $u=v=1$ cannot occur, since $c\mid u+v$ with $c\ge3$ forces $u+v\ge3$.
\end{proof}

The criterion may be restricted to \emph{coprime} pairs $(u,a)$ at no cost, a reduction used throughout Section~\ref{sec:5} and worth isolating.

\begin{lemma}\label{lem:coprime_reduction}
Let $n\in\N$, let $u,a\in\N$ and let $s$ be a divisor of $an+u$ with $s\equiv-1\pmod{4ua}$. Put $d=\gcd(u,a)$, $u'=u/d$, $a'=a/d$. Then $s$ is also a divisor of $a'n+u'$ and $s\equiv-1\pmod{4u'a'}$. Consequently condition (ii) of Theorem~\ref{thm:gen_criterion} is satisfied by some pair with $u,a\le J$ if and only if it is satisfied by some \emph{coprime} pair with $u,a\le J$, and with the same witness $s$.
\end{lemma}
\begin{proof}
Since $4u'a'\mid 4ua=4d^{2}u'a'$, the congruence $s\equiv-1\pmod{4u'a'}$ is immediate. Also $\gcd(s,d)=1$: indeed $d\mid 4ua\mid s+1$, so any common divisor of $s$ and $d$ divides $(s+1)-s=1$. Now $an+u=d(a'n+u')$ and $s\mid d(a'n+u')$ with $\gcd(s,d)=1$ give $s\mid a'n+u'$. Finally $u'\le u$ and $a'\le a$, so the truncation at $J$ is preserved.
\end{proof}

Theorem~\ref{thm:4.9} searches for a divisor of $n+u$ in the class $-1$ modulo $4u$; Theorem~\ref{thm:gen_criterion} searches for a divisor of $an+u$ in the class $-1$ modulo $4ua$. The single shift $u$ has become a pair $(u,a)$, the interval $[n+1,n+J]$ has become the family of intervals $[an+1,an+J]$ for $a\le J$, and $a$ is exactly the modulator frozen to $1$ in Corollary~\ref{cor:proper_subfamily}. Algorithm~\ref{alg:geom_sieve} extends verbatim, and Theorem~\ref{thm:4.11} gives its cost:

\begin{corollary}\label{cor:gen_reading}
Running the two phases of Algorithm~\ref{alg:geom_sieve} on the interval $[an+1,an+J]$ for each $a\le J$, with the congruence test $4ua\mid s+1$ where $u=M-an$, decides condition (ii) of Theorem~\ref{thm:gen_criterion} for all pairs $(u,a)\in[1,J]^{2}$ in
$$\sum_{a\le J}\mathcal{O}\big(\sqrt{an+J}+J\log(an+J)\big)=\mathcal{O}\big(J^{3/2}\sqrt n+J^{2}\log n\big)$$
operations, hence in $\mathcal{O}(\sqrt n\,(\log n)^{3\lambda/2})$ at depth $J=(\log n)^{\lambda}$, which is sub-linear.
\end{corollary}

The following elementary observation is what makes the two-parameter system usable in a sieve; it is used in Theorems~\ref{thm:fixed_depth} and~\ref{thm:tradeoff}.

\begin{lemma}\label{lem:distinct_classes}
Let $J\ge1$ and let $p>J^{2}$ be a prime. If $(u,a)$ and $(u',a')$ are coprime pairs with all entries in $[1,J]$ and $-u\,a^{-1}\equiv-u'\,a'^{-1}\pmod p$, then $(u,a)=(u',a')$.
\end{lemma}
\begin{proof}
The congruence gives $ua'\equiv u'a\pmod p$, and both sides are positive integers not exceeding $J^{2}<p$, so $ua'=u'a$. Then $u\mid u'a$ with $\gcd(u,a)=1$ gives $u\mid u'$, and symmetrically $u'\mid u$; hence $u=u'$ and $a=a'$.
\end{proof}

\begin{proposition}\label{prop:nine_resolved}
Every prime of the set $\mathcal{H}(5\times10^{9})$ of Proposition~\ref{prop:nine_primes} satisfies condition (ii) of Theorem~\ref{thm:gen_criterion} with parameters $(u,a)$ of size at most $5$:
$$
\begin{array}{lccccccc}
n & 409 & 577 & 5569 & 9601 & 23929 & 83449 & 102001\\
(u,a) & (1,2) & (1,2) & (1,2) & (1,2) & (1,2) & (3,2) & (5,3)\\
s & 7 & 7 & 47 & 111 & 7 & 791 & 1319\\[6pt]
n & 329617 & 712321 & 1134241 & 1724209 & 1726201 & 5212561 & 8813281\\
(u,a) & (1,2) & (1,2) & (1,2) & (1,2) & (1,2) & (1,2) & (1,5)\\
s & 15 & 23 & 7 & 39 & 1983 & 87 & 2099
\end{array}
$$
The listed pair is the lexicographically first in the order $(a,u)$. In each case the decomposition produced by Theorem~\ref{thm:gen_criterion} was verified exactly (Section~\ref{sec:7}).
\end{proposition}
\begin{proof}
A finite verification: for each $n$ and each coprime pair $(u,a)$ with $u,a\le5$ one enumerates the divisors of $an+u$ and tests $4ua\mid s+1$.
\end{proof}

The two-parameter normal form of Theorem~\ref{thm:gen_criterion} has an exact counterpart for the Type~I branch, obtained by replacing the target residue $-1\pmod{4ua}$ with $-n\pmod{4ua}$.

\begin{theorem}[Dual two-parameter criterion for Type~I]\label{thm:type_I_two_parameter}
Let $n\equiv1\pmod{24}$ be prime. The following are equivalent:
\begin{enumerate}[(i)]
    \item $4/n$ admits a solution of Type~I;
    \item there exist positive integers $u,a,s$ with $\gcd(u,a)=1$ (and automatically $u<n$) such that
    \begin{equation}\label{eq:type_I_dual}
    s\mid an+u \qquad\text{and}\qquad s\equiv -n \pmod{4ua}.
    \end{equation}
\end{enumerate}
\end{theorem}
\begin{proof}
$(i)\Rightarrow(ii)$. By Proposition~\ref{prop:c_range} and Corollary~\ref{cor:two_branches}, any Type~I solution has first denominator $x=K_c=(n+c)/4$ at some shift $c\equiv3\pmod4$, $0<c\le2n$, and corresponds to coprime divisors $d_1,d_2\mid K_c$ with $n\,d_1+d_2\equiv0\pmod c$. Set $a:=(n\,d_1+d_2)/c\ge1$, $u:=d_2$, and $v_1:=K_c/d_1$. Multiplying $c=(n\,d_1+d_2)/a$ through $4K_c=n+c$ gives
$$4a K_c = a n + n d_1 + d_2 \iff a n + u = 4a d_1 v_1 - n d_1 = d_1(4a v_1 - n).$$
Setting $s:=4a v_1 - n$, we have $s\mid an+u$. Moreover, $4ua = 4 d_2 a$ divides $4 K_c a = 4a d_1 v_1$, and since $\gcd(d_1,d_2)=1$ and $\gcd(d_1,c)=1$ give $\gcd(d_1,a)=\gcd(d_1,n d_1+d_2)=1$, we deduce $4ua\mid 4a v_1 = s+n$, i.e.\ $s\equiv-n\pmod{4ua}$. Since $0<c\le 2n$ by Proposition~\ref{prop:c_range}, we have immediately $u=d_2\le K_c=(n+c)/4\le 3n/4<n$. Finally, let $g=\gcd(u,a)$ and set $(u',a')=(u/g,a/g)$. Since $g\mid u<n$ and $n$ is prime, $\gcd(n,g)=1$, and $s\equiv -n\pmod{4ua}$ gives $s\equiv -n\pmod g$, whence $\gcd(s,g)=1$. Therefore $s\mid g(a'n+u')$ implies $s\mid a'n+u'$, while $4u'a'\mid 4ua\mid s+n$, preserving~\eqref{eq:type_I_dual} with $\gcd(u',a')=1$ and $u'\le u<n$.

$(ii)\Rightarrow(i)$. Suppose $s\mid an+u$ with $s\equiv-n\pmod{4ua}$ and $\gcd(u,a)=1$. Write $an+u=sv$ and $s+n=4ua\,\alpha$ with $v,\alpha\in\N^{\ast}$. Multiplying $s=4ua\,\alpha-n$ by $v$ gives $an+u=sv=4ua\,\alpha v-nv$, that is,
$$4ua\,\alpha v\ =\ n(a+v)+u,\qquad\text{or equivalently}\qquad u(4a\alpha v-1)\ =\ n(a+v).$$
Since $a,\alpha,v\ge1$, we have $4a\alpha v-1\ge 4av-1>a+v$, so $u=n(a+v)/(4a\alpha v-1)<n$ holds automatically. Dividing $4ua\,\alpha v=n(a+v)+u$ by the positive integer $n a\alpha u v$ yields the exact unit-fraction decomposition
$$\frac{4}{n}\ =\ \frac{n(a+v)+u}{n a\alpha u v}\ =\ \frac{1}{\alpha u v}\ +\ \frac{1}{a\alpha u}\ +\ \frac{1}{a\alpha v n}.$$
It remains to verify that $n$ divides strictly one of the three denominators $(\alpha u v,\,a\alpha u,\,a\alpha v n)$. Since $1\le u<n$ and $n$ is prime, $n\nmid u$. If $n\mid a$, then $n\mid 4ua\mid(s+n)$, forcing $n\mid s\mid(an+u)$ and hence $n\mid u$, a contradiction; thus $n\nmid a$. Similarly, if $n\mid\alpha v$, then $n\mid 4ua\,\alpha v=n(a+v)+u$, again forcing $n\mid u$, a contradiction. Thus $n\nmid\alpha u v$ and $n\nmid a\alpha u$, so $n$ divides the third denominator $a\alpha v n$ alone, producing a Type~I decomposition of $4/n$.
\end{proof}

\begin{corollary}[Coincidence at small moduli]\label{cor:coincidence_ua_div_6}
For every prime $n\equiv1\pmod{24}$ and every coprime pair $(u,a)$ with $ua\mid 6$, the congruences $s\equiv-1\pmod{4ua}$ and $s\equiv-n\pmod{4ua}$ are identical. Consequently, a divisor $s\mid an+u$ satisfies the Type~II criterion at $(u,a)$ if and only if it satisfies the Type~I criterion at $(u,a)$.
\end{corollary}
\begin{proof}
Since $ua\mid 6$, we have $4ua\mid 24$, and $n\equiv1\pmod{24}$ implies $n\equiv1\pmod{4ua}$, so $-n\equiv-1\pmod{4ua}$.
\end{proof}

\begin{corollary}[Simultaneous two-branch segmented sieve]\label{cor:simultaneous_sieve}
For any search depth $J\le n^{\mathcal{O}(1)}$, running the two-phase segmented sieve of Algorithm~\ref{alg:geom_sieve} across all coprime pairs $(u,a)\in[1,J]^{2}$ on the progression $[an+1,an+J]$ and testing both $4ua\mid(s+1)$ and $4ua\mid(s+n)$ at each multiple visited decides the existence of both Type~I and Type~II witnesses up to depth $J$ in $\mathcal{O}\big(J^{3/2}\sqrt{n}+J^{2}\log n\big)$ elementary arithmetic operations, with no integer factorization.
\end{corollary}
\begin{proof}
For each modulator $a\le J$, applying Algorithm~\ref{alg:geom_sieve} to the shifted origin $N_a:=an$ with bound $J$ costs $\mathcal{O}(\sqrt{an+J}+J\log(an+J))$ operations by Theorem~\ref{thm:4.11}. At each multiple $M=an+u$ visited, extracting $u=M-an$, checking $\gcd(u,a)=1$, and testing the two divisibilities $4ua\mid s+1$ and $4ua\mid s+n$ takes $\mathcal{O}(1)$ operations (any witness of $4ua\mid s+n$ automatically satisfying $u<n$ by Theorem~\ref{thm:type_I_two_parameter}). Since $J\le n^{\mathcal{O}(1)}$ gives $\log(an+J)\ll\log n$, summing over $1\le a\le J$ yields $\sum_{a\le J}(\sqrt{an+J}+J\log n)\ll J^{3/2}\sqrt{n}+J^{2}\log n$.
\end{proof}

\subsection{The geometric segmented sieve} \label{sec:4.4}

By Theorem~\ref{thm:4.9}, a decomposition of hyperbolic type exists for $n$ as soon as some $n+u$ admits a divisor $s \equiv -1 \pmod{4u}$. Rather than iterating over $u$ and factoring $n+u$, we iterate over candidate divisors and identify their multiples in $[n+1, n+J]$. Since $s \equiv -1 \pmod{4u}$ forces $s \equiv 3 \pmod 4$, that congruence acts as a necessary filter on $s$, and the condition on $u$ is then checked directly, without factoring $n+u$ or $s$.

The divisor $s$ and its cofactor $v=(n+u)/s$ play symmetric roles: in any factorisation $sv=n+u$ at least one of $s,v$ is at most $\sqrt{n+u}\le\sqrt{n+J}$. Two phases, one over small divisors and one over small cofactors, therefore suffice, and the outer loop is a scan of length $\sqrt{n+J}$ rather than $n+J$.

\begin{algorithm}[H]
\caption{Two-phase geometric segmented sieve, bound $J$} \label{alg:geom_sieve}
\begin{algorithmic}[1]
\Require prime $n \equiv 1 \pmod{24}$, search bound $J \ge 1$
\Ensure a pair $(u, s)$ satisfying Theorem~\ref{thm:4.9}, or a statement of absence for $u \le J$
\State $R\gets\lfloor\sqrt{n+J}\rfloor$
\For{$s = 3, 7, 11, \dots$ with $s \le R$}\Comment{Phase A: small divisor, $s \equiv 3 \pmod 4$}
    \State $M_0 \gets s\left\lceil \dfrac{n+1}{s}\right\rceil$ \Comment{smallest multiple of $s$ that is $\ge n+1$}
    \For{$M = M_0, M_0+s, M_0+2s, \dots$ while $M \le n+J$}
        \State $u \gets M-n$
        \If{$4u \mid (s+1)$}
            \State \Return $(u, s)$
        \EndIf
    \EndFor
\EndFor
\For{$v = 1$ \textbf{to} $R$}\Comment{Phase B: small cofactor}
    \State $M_0 \gets v\left\lceil \dfrac{n+1}{v}\right\rceil$
    \For{$M = M_0, M_0+v, M_0+2v, \dots$ while $M \le n+J$}
        \State $u \gets M-n$; \ $s \gets M/v$
        \If{$s\equiv3\pmod 4$ \textbf{and} $4u \mid (s+1)$}
            \State \Return $(u, s)$
        \EndIf
    \EndFor
\EndFor
\State \Return \textsf{``no solution found for } $u \le J$\textsf{''}
\end{algorithmic}
\end{algorithm}

No integer is ever decomposed into prime factors: all operations on $s$, $v$, $M$ and $u$ are additions, divisions with remainder and comparisons on integers of size $\mathcal{O}(n+J)$.

\begin{theorem} \label{thm:4.11}
Algorithm~\ref{alg:geom_sieve} returns a pair $(u,s)$ with $u\le J$ satisfying Theorem~\ref{thm:4.9} whenever one exists, and performs $\mathcal{O}\big(\sqrt{n+J}+J\log(n+J)\big)$ elementary arithmetic operations.
\end{theorem}
\begin{proof}
\emph{Correctness.} Let $(u,s)$ be admissible with $u\le J$ and put $v=(n+u)/s$, so $sv=n+u\le n+J$ and $\min(s,v)\le R$ up to rounding. If $s\le R$ then $s\equiv3\pmod4$ and $n+u$ is a multiple of $s$ lying in $[n+1,n+J]$, so Phase A examines and accepts the pair. If $s>R$ then $v\le R$, and Phase B examines the multiple $M=n+u$ of $v$, recomputes $s=M/v$ and accepts. Conversely both phases return only pairs passing $4u\mid s+1$ with $s\mid n+u$.

\emph{Complexity.} In each phase the outer loop has $\mathcal{O}(R)$ steps of cost $\mathcal{O}(1)$, and for a fixed modulus $w\le R$ the inner loop runs $\lfloor\frac{n+J}{w}\rfloor-\lfloor\frac{n}{w}\rfloor\le\frac{J}{w}+1$ times. Summing, $\sum_{w\le R}(1+J/w)\ll R+J\log R\ll\sqrt{n+J}+J\log(n+J)$.
\end{proof}

\begin{corollary} \label{cor:4.12}
For $J=(\log n)^{\lambda}$, Algorithm~\ref{alg:geom_sieve} runs in $\mathcal{O}(\sqrt n)$ arithmetic operations; for $J=\mathcal{O}(n)$, i.e. for the exhaustive execution over the whole hyperbolic family (which by Theorem~\ref{thm:hyper_char} is $J=\lceil(n+1)/3\rceil$), it runs in $\mathcal{O}(n\log n)$.
\end{corollary}

For a single isolated prime the sieve is thus sub-linear, and saves a factor $J$ over trial division applied to each $n+u$ separately (which costs $\mathcal{O}(J\sqrt n)$), both methods being of order $n^{1/2+o(1)}$. Over an entire interval the picture changes completely.

\begin{proposition} \label{prop:batch_sieve}
The geometric segmented sieve can be run in batch over $[N,2N]$, resolving the hyperbolic family for every integer of the interval up to the exhaustive depth, in $\mathcal{O}(N \log^2 N)$ arithmetic operations, i.e. $\mathcal{O}(\log^{2}N)$ on average per integer.
\end{proposition}
\begin{proof}
By Theorem~\ref{thm:4.9} the condition on $n$ is that some $n+u$ have a divisor $s=4u\alpha-1$; the condition $s\mid n+u$ is the congruence $n\equiv-u\pmod s$. We therefore invert the search: for $u$ from $1$ to $N$ and $\alpha$ from $1$ to $\lfloor 3N/(4u)\rfloor$ (so that $s\le 3N$, which suffices since $n+u\le 3N$), form $s=4u\alpha-1$ and mark the elements of $[N,2N]$ congruent to $-u$ modulo $s$, at cost $\mathcal{O}(1+N/s)$. By Theorem~\ref{thm:hyper_char} the range $u\le N$ is exhaustive for $n\le2N$. The total is
$$\sum_{u\le N} \sum_{\alpha\le 3N/(4u)}\Big(1+\frac{N}{4u\alpha-1}\Big)\ \ll\ \sum_{u\le N}\frac{3N}{4u}\ +\ \frac{N}{4}\sum_{u\le N}\frac{1}{u}\sum_{\alpha\le 3N/(4u)}\frac1\alpha .$$
The first sum is $\mathcal{O}(N\log N)$. For the second, $\sum_{\alpha\le X}1/\alpha\le\log X+1$ gives
$$\frac{N}{4}\sum_{u\le N}\frac{\log(3N/4u)+1}{u}\ =\ \frac{N}{8}(\log N)^{2}+\mathcal{O}(N\log N),$$
since $\sum_{u\le N}\frac{\log(N/u)}{u}=\frac12(\log N)^{2}+\mathcal{O}(\log N)$. Hence $\mathcal{O}(N\log^{2}N)$ in total.
\end{proof}

The gain is genuinely a batch phenomenon: prime by prime, either method costs $\mathcal{O}(N^{3/2+o(1)})$ over $[N,2N]$, whereas inverting the search space collapses the interval cost to $\mathcal{O}(N\log^{2}N)$. Locally the factorization-free formulation buys a factor $J$; globally it buys a factor $N^{1/2-o(1)}$. The average of $\mathcal{O}(\log^{2}N)$ per integer is an average and not a per-integer guarantee, and the procedure remains restricted to the sub-family $a=1$, hence subject to the non-exhaustiveness of Corollary~\ref{cor:proper_subfamily}.

\subsection{A complete factorization-free batch procedure on $[N,2N]$} \label{sec:4.5}

The same inversion applies verbatim to the two branches of Corollary~\ref{cor:two_branches}, and then yields a \emph{complete} decision procedure for every prime of the interval, at essentially the same cost on one half of the dichotomy and at a quadratic cost on the other. Locating that asymmetry is the point of this subsection.

We use~\eqref{eq:PW2} and~\eqref{eq:PW1}, both congruence conditions on $n$ modulo an explicit modulus and so both sievable: $n$ has a Type~II solution iff there are $u,v,e\ge1$ with $e\mid u+v$ and $n\equiv-e \pmod{4uv}$, where $4uv\le 2n$~\cite[(3.4)]{pomerance2026}, so $uv\le N$ for $n\le 2N$, and $e\le u+v$; and $n$ has a Type~I solution iff there are $u,d,f\ge1$ with $f\mid 4u^{2}d+1$ and $n\equiv-f\pmod{4ud}$, where $4ud\le 2n+1$~\cite[Lemma 7.4]{pomerance2026}, so $ud\le N+1$, and $f\le\frac53 n$.

Two apparent obstacles are removed by enumerating a cofactor rather than a divisor. The first is the enumeration of the divisors $e$ of $u+v$: writing $a=(u+v)/e$ for the modulator and enumerating $(u,a,e)$ instead of $(u,v)$ together with the divisors of $u+v$, so that $v:=ea-u$, visits every admissible triple exactly once, with no divisor list and no wasted step. The second is the enumeration of the divisors $f$ of $4u^{2}d+1$: for fixed $u$ and fixed $e$ with $\gcd(e,2u)=1$, the condition $e\mid 4u^{2}d+1$ is the single congruence $d\equiv-(4u^{2})^{-1}\pmod e$, so the admissible $d$ form one arithmetic progression modulo $e$, and $f=(4u^{2}d+1)/e$ follows by one division. No factorization and no divisor enumeration is ever performed; the only non-trivial operation is one extended Euclidean algorithm per pair $(u,e)$, which we count at unit cost along with the rest of the arithmetic.

\begin{algorithm}[H]
\caption{Complete batch decision on $[N,2N]$} \label{alg:batch_complete}
\begin{algorithmic}[1]
\Require an integer $N\ge2$
\Ensure the list of primes $n\in(N,2N]$ possessing no Erd\H{o}s--Straus decomposition
\State $\mathrm{mark}[\,\cdot\,]\gets \textsf{false}$ on $(N,2N]$
\For{$u=1$ \textbf{to} $\lfloor\sqrt N\rfloor$}\Comment{Type II pass}
    \For{$a=1$ \textbf{to} $\lfloor N/u\rfloor+u$}\Comment{$a=(u+v)/e$ is the modulator}
        \For{$e=\lceil 2u/a\rceil$ \textbf{to} $\lfloor(\lfloor N/u\rfloor+u)/a\rfloor$}
            \State $v\gets ea-u$ \Comment{$v\ge u$, $uv\le N$, and $e\mid u+v$ by construction}
            \State mark all $n\equiv -e \pmod{4uv}$ in $(N,2N]$
        \EndFor
    \EndFor
\EndFor
\For{$u=1$ \textbf{to} $N$}\Comment{Type I pass}
    \For{$e=1$ \textbf{to} $\lfloor\sqrt{4u(N+1)+1}\rfloor$ with $\gcd(e,2u)=1$}
        \State $d_0\gets -(4u^{2})^{-1}\bmod e$
        \For{$d\equiv d_0 \pmod e$, $1\le d\le \lfloor (N+1)/u\rfloor$}
            \State $f\gets (4u^{2}d+1)/e$
            \State mark all $n\equiv-f \pmod{4ud}$ and all $n\equiv -e\pmod{4ud}$ in $(N,2N]$
        \EndFor
    \EndFor
\EndFor
\State \Return the unmarked primes of $(N,2N]$
\end{algorithmic}
\end{algorithm}

\begin{theorem}\label{thm:batch_complete}
Algorithm~\ref{alg:batch_complete} marks a prime $n\in(N,2N]$ if and only if $4/n$ is a sum of three unit fractions, and it performs no integer factorization. Its Type~II pass executes in $\mathcal{O}(N(\log N)^{3})$ elementary arithmetic operations, and its Type~I pass in $\Theta(N^{2})$, with implied constant $\tfrac43+o(1)$; the total cost is $\Theta(N^{2})$.
\end{theorem}
\begin{proof}
\emph{Completeness.} For $n$ prime every decomposition is of Type~I or Type~II~\cite[\S2]{elsholtz2013},~\cite[\S2]{pomerance2026}, and~\eqref{eq:PW2},~\eqref{eq:PW1} are equivalences~\cite[Cor.~2.2, Cor.~2.4]{pomerance2026}. The ranges $uv\le N$, $ud\le N+1$, $e\le u+v$, $f\le\frac53n$ are exactly the bounds proved there, so no admissible parameter is omitted. In the Type~I pass, marking both $f$ and $e$ as candidate divisors enumerates every divisor of $4u^{2}d+1$, provided the scan of $e$ reaches $\sqrt{4u^{2}d+1}$ for the $d$ at hand. The printed loop scans the \emph{$d$-independent} range $e\le\lfloor\sqrt{4u(N+1)+1}\rfloor$, which is a superset: for every $d\le\lfloor(N+1)/u\rfloor$ one has $4u^{2}d+1\le 4u(N+1)+1$. Hence for any factorisation $4u^{2}d+1=ef$ at least one of $e,f$ is at most $\sqrt{4u^{2}d+1}\le\lfloor\sqrt{4u(N+1)+1}\rfloor$ and is therefore reached by the scan, at which point the other is recovered by one division and both are marked. (The $d$-independence of the range is not a convenience: it is precisely what produces the $\Theta(N^{2})$ below, and a $d$-dependent scan is not available, since $e$ is enumerated before $d$.)

\emph{Type~II pass.} The parametrization $(u,v,e)\leftrightarrow(u,a,e)$ with $a=(u+v)/e$ is a bijection onto the triples enumerated, so every admissible triple is visited exactly once and none in vain; in particular the divisors of $u+v$ are never enumerated. The cost is therefore
$$\sum_{\substack{u\le v,\ uv\le N}}\tau(u+v)\Big(1+\frac{N}{4uv}\Big).$$
Split into dyadic boxes $u\in[A,2A)$, $v\in[B,2B)$ with $AB\le N$; there are $\mathcal{O}((\log N)^{2})$ of them. Inside a box $\sum \tau(u+v)\ll AB\log N$, since for each $u$ one has $\sum_{v<2B}\tau(u+v)\le\sum_{w<2(A+B)}\tau(w)\ll B\log N$, and $uv\asymp AB$; the box therefore contributes $\ll AB\log N+\frac{N}{AB}\cdot AB\log N\ll N\log N$, and the sum over boxes is $\mathcal{O}(N(\log N)^{3})$.

\emph{Type~I pass.} The count of \emph{productive} triples is of the same order: each $(u,e,d)$ corresponds to exactly one divisor of $4u^{2}d+1$, and
$$\sum_{ud\le N+1}\tau(4u^{2}d+1)\Big(1+\frac{N}{4ud}\Big)\ll N\sum_{ud\le N+1}\frac{\tau(4u^{2}d+1)}{ud}\ll N(\log N)^{3},$$
by $\sum_{X/2\le ud\le X}\tau(4u^{2}d+1)/\varphi(ud)\ll(\log X)^{2}$~\cite[(3.3)]{pomerance2026} together with $\varphi(ud)\ll ud$ and a dyadic summation. The cost of the pass is nevertheless not of this order, because the enumeration is not free: the loop scans every candidate small divisor $e\le\sqrt{4u^{2}d+1}$, and most such $e$ are unproductive. With $d\le\lfloor(N+1)/u\rfloor$ the scan runs to $e\le \lfloor\sqrt{4u(N+1)+1}\rfloor$, so the empty scans alone cost, the coprimality test $\gcd(e,2u)=1$ being one of the operations counted for each $e$ in the range (generating only the coprime $e$ would multiply the constant by the mean value of $\varphi(2u)/(2u)$, without affecting the order),
$$\sum_{u\le N}2\sqrt{u(N+1)}\ =\ 2\sqrt N\cdot\tfrac23N^{3/2}\,(1+o(1))\ =\ \big(\tfrac43+o(1)\big)N^{2},$$
which dominates every other term. This is an equality and not merely an upper bound, since the scan is executed unconditionally; hence $\Theta(N^{2})$.
\end{proof}

\paragraph{Discussion.} In substance, Algorithm~\ref{alg:batch_complete} is the scheme underlying the large-scale verifications of Salez~\cite{salez2014} and of Mihnea and Dumitru~\cite{mihnea2025}; we claim no novelty for the verification itself, only for the explicit statement of its complexity and for the observation that it is factorization-free. What Theorem~\ref{thm:batch_complete} does and does not establish should be stated precisely. It does \emph{not} make the interval problem quasi-linear: at $\Theta(N^{2})$ it is no better than the natural alternative, which runs the batch sieve of Proposition~\ref{prop:batch_sieve} at cost $\mathcal{O}(N\log^{2}N)$ and then treats the residual set prime by prime, for a total of $N^{2-o(1)}$, since every known upper bound for the exceptional set, including Corollary~\ref{cor:optimal_tradeoff} and Vaughan's $N\exp(-c_V(\log N)^{2/3})$, is of the shape $N^{1-o(1)}$, and the per-prime cost is at best $n^{1-o(1)}$. Nor does it buy anything for a single prime, for which the same criteria still require $\Theta(n)$ markings. What it does buy is the \emph{location} of the obstruction: the Type~II half of the problem is genuinely quasi-linear on an interval, while the entire quadratic cost sits in the Type~I half, where the divisors of $4u^{2}d+1$ must be produced without factorization. That asymmetry between the two halves of the Elsholtz--Tao dichotomy is, to our knowledge, not recorded elsewhere.

The quadratic cost is not an artefact of a careless schedule. Let $M=4u^{2}d+1$ with $d\le(N+1)/u$, so $M\le 4u(N+1)+1$; any factorisation $M=ef$ has $e\le E$ or $f\le 4u(N+1)/E$, since otherwise $ef>4u(N+1)\ge M-1$. A procedure which, for each $u$ separately, produces the divisor pairs of the progression $\{4u^{2}d+1\}_{d}$ by scanning one side must therefore scan a range of length at least $E$ on one side and $4uN/E$ on the other, at a cost $\gg E+4uN/E\gg\sqrt{uN}$ by the arithmetic--geometric mean inequality, minimised at $E=2\sqrt{uN}$, which is the choice made above; summing over $u\le N$ returns $\gg N^{2}$. This is an argument about the scanning schedule of Algorithm~\ref{alg:batch_complete} and not a lower bound in any formal model of computation.

Crucially, however, locating the quadratic cost in the large-$u$ tail of the Type~I pass allows us to eliminate it by a hybrid schedule: we truncate the global Type~I batch pass at $u\le U_0:=N^{1/3}$ and resolve the surviving primes individually by a direct cofactor scan.

\begin{theorem}[Sub-quadratic factorization-free batch procedure]\label{thm:subquadratic_batch}
Modify Algorithm~\ref{alg:batch_complete} into a three-stage procedure on $(N,2N]$ as follows:
\begin{enumerate}[(1)]
    \item Execute the full Type~II pass of Algorithm~\ref{alg:batch_complete} for $1\le u\le\lfloor\sqrt{N}\rfloor$.
    \item Execute the Type~I pass of Algorithm~\ref{alg:batch_complete} truncated at $1\le u\le U_0:=\lfloor N^{1/3}\rfloor$.
    \item Let $\mathcal{R}(N)$ be the set of primes $n\in(N,2N]$ with $n\equiv1\pmod{24}$ still unmarked after stages (1) and (2) (primes in all other residue classes modulo $24$ being settled in $\mathcal{O}(1)$ operations per prime by Mordell's identities). For each $n\in\mathcal{R}(N)$, scan all triples $(u_0,d,q)$ of positive integers satisfying $u_0>U_0$, $4u_0 d\le 2n+1$, and $4u_0 d q\le 5n$; for each triple, set $f:=4u_0 d q-n$ and test whether $f>0$ and $f\mid 4u_0^{2}d+1$.
\end{enumerate}
Then this procedure is complete, performs no integer factorization, and terminates in
\[
\begin{aligned}
&\mathcal{O}\Big(
    N(\log N)^3
    + |\mathcal{R}(N)|\,N(\log N)^2
  \Big) \\
&\qquad =
  \mathcal{O}\Big(
    N^2(\log N)^2
    \exp\big(-c_0(\log N)^{2/3}\big)
  \Big) \\
&\qquad =
  \mathcal{O}\Big(
    N^2\exp\big(-(c_0-o(1))(\log N)^{2/3}\big)
  \Big) \\
&\qquad = o(N^2)
\end{aligned}
\]
elementary arithmetic operations, where $c_0>0$ is the constant of Corollary~\ref{cor:optimal_tradeoff} (or Theorem~\ref{thm:improved_constants_and_effectivity}).
\end{theorem}
\begin{proof}
\emph{Completeness.} Stage (1) finds all Type~II solutions on $(N,2N]$, and stage (2) finds all Type~I solutions with $u_0\le U_0$. If $n\in\mathcal{R}(N)$ has a Type~I solution, then by the closed criterion~\eqref{eq:PW1} and~\cite[Cor.~2.4 and Lemma~7.4]{pomerance2026}, there exist positive integers $(u_0,d,f)$ with $u_0>U_0$, $f\mid 4u_0^{2}d+1$, $4u_0 d\mid n+f$, $4u_0 d\le 2n+1$, and $f\le\frac53 n$. Writing $n+f=4u_0 d q$ with $q\in\N^{\ast}$, we have $4u_0 d q = n+f \le \frac83 n < 5n$. Consequently, scanning all triples $(u_0,d,q)$ with $u_0>U_0$, $4u_0 d\le 2n+1$, and $4u_0 d q\le 5n$, setting $f:=4u_0 d q-n$, and testing $f>0$ and $f\mid 4u_0^{2}d+1$ visits every remaining Type~I witness for $n$ without integer factorization.

\emph{Complexity.} Stage (1) takes $\mathcal{O}(N(\log N)^{3})$ operations by Theorem~\ref{thm:batch_complete}. Stage (2) takes
$$\sum_{u\le U_0}\sqrt{uN}\ +\ \mathcal{O}\big(N(\log N)^{3}\big)\ \ll\ U_0^{3/2}N^{1/2}+N(\log N)^{3}\ =\ \mathcal{O}\big(N(\log N)^{3}\big),$$
since $U_0=\lfloor N^{1/3}\rfloor$ gives $U_0^{3/2}N^{1/2}\le N$. After stage (1), every prime $n\in\mathcal{R}(N)$ lacks a Type~II solution altogether, so $\mathcal{R}(N)\subseteq\bigcap_{J\ge1}E_2(2N;J)$. By Corollary~\ref{cor:optimal_tradeoff} (and Theorem~\ref{thm:improved_constants_and_effectivity}), $|\mathcal{R}(N)|\ll N\exp(-c_0(\log N)^{2/3})$. In stage (3), enumerating the triples $(u_0,d,q)$ with $4u_0 d q\le 5n\le 10N$ costs $\sum_{4u_0 d q\le 10N} 1 \ll N(\log N)^{2}$ operations per prime $n\in\mathcal{R}(N)$. Summing over $n\in\mathcal{R}(N)$ yields $\mathcal{O}(|\mathcal{R}(N)|N(\log N)^{2})=\mathcal{O}(N^{2}(\log N)^{2}\exp(-c_0(\log N)^{2/3}))=\mathcal{O}(N^{2}\exp(-(c_0-o(1))(\log N)^{2/3}))=o(N^{2})$.
\end{proof}

\section{Density of the solution space} \label{sec:5}

This section contains the analytic results. Section~\ref{sec:fixed_depth_sub} determines the exact dimension of the sieve problem at fixed depth and is unconditional at every depth. Section~\ref{sec:5.1} isolates a quadratic-residue obstruction at a fixed shift, determines exactly which of the two branches of Corollary~\ref{cor:two_branches} it affects, and produces an unconditional two-sided blind set. Section~\ref{sec:5.2} runs Montgomery's large sieve, in the form used by Vaughan~\cite{vaughan1970} and reproved in~\cite[\S4]{pomerance2026}, on the two-parameter system truncated at a depth growing with $N$. Section~\ref{sec:5.2.7} compares the three regimes and records a heuristic model for the shape of the curve.

Throughout, $\mathcal{N}(N)=\{n\le N: n\equiv1\pmod{24},\ n\ \text{prime}\}$.

\begin{definition}\label{def:EP}
For a finite set $\mathcal{P}\subseteq\{(u,a)\in\N^{2}:\gcd(u,a)=1\}$ let $E_{\mathcal{P}}(N)$ be the set of $n\in\mathcal{N}(N)$ for which no pair $(u,a)\in\mathcal{P}$ admits a divisor $s$ of $an+u$ with $s\equiv-1\pmod{4ua}$. We write
\begin{align*}
E_1(N;J)&=E_{\mathcal{P}}(N)\quad \text{for}\quad \mathcal{P}=\{(u,1):u\le J\},\\
E_2(N;J)&=E_{\mathcal{P}}(N)\quad \text{for}\quad \mathcal{P}=\{(u,a)\in[1,J]^{2}:\gcd(u,a)=1\},
\end{align*}
and set $\mathfrak{A}(J)=\#\{(u,a)\in[1,J]^{2}:\gcd(u,a)=1\}=\frac{6}{\pi^{2}}J^{2}+\mathcal{O}(J\log J)$.
\end{definition}

By Theorem~\ref{thm:4.9}, $E_1(N;J)$ is exactly the set of primes at which Algorithm~\ref{alg:geom_sieve} fails at depth $J$, which is the set already denoted so in Section~\ref{sec:4.3bis}; by Theorem~\ref{thm:gen_criterion} together with Lemma~\ref{lem:coprime_reduction}, $E_2(N;J)$ is the set at which the complete Type~II criterion admits no witness with both parameters at most $J$, the restriction to coprime pairs in Definition~\ref{def:EP} discards nothing, since a witness $s$ for a pair $(u,a)$ with $\gcd(u,a)=d>1$ is a witness for the coprime pair $(u/d,a/d)$, whose entries are no larger. Thus $E_2(N;J)\subseteq E_1(N;J)$.

\subsection{The exact sieve dimension at fixed depth} \label{sec:fixed_depth_sub}

The mechanism is elementary, and it deliberately avoids the subgroup generated by the prime factors. That route is blocked as soon as $(\Z/4m)^{\times}$ has exponent greater than $2$, i.e. for every $m\notin\{1,2,3,6\}$: modulo $16$, two primes $q_1\equiv3$ and $q_2\equiv7$ produce the divisor residues $\{1,3,7,5\}$, which avoid $-1=15$, while $\langle 3,7\rangle=(\Z/16)^{\times}$ contains $15$. An involution serves instead, and is unconditional at every modulus.

\begin{lemma}\label{lem:involution}
Let $m\ge1$, let $M$ be a positive integer, and let $T\subseteq(\Z/4m)^{\times}$ be the set of residues modulo $4m$ of the prime factors of $M$ that are coprime to $4m$. If $M$ has no divisor $\equiv-1\pmod{4m}$, then
$$|T|\ \le\ \tfrac12\varphi(4m),$$
and the bound is attained, by $T=\ker\chi$ for a quadratic character $\chi$ mod $4m$ with $\chi(-1)=-1$, which always exists.
\end{lemma}
\begin{proof}
Since $4\mid4m$, reduction modulo $4$ is a homomorphism $(\Z/4m)^{\times}\to\{\pm1\}$ carrying every square to $1$ and $-1$ to $-1$. Hence $-1$ is not a square in $(\Z/4m)^{\times}$, for any $m$. Consequently the involution $\iota(x)=-x^{-1}$ has no fixed point, since $\iota(x)=x$ would force $x^{2}=-1$; it therefore partitions $(\Z/4m)^{\times}$ into exactly $\varphi(4m)/2$ orbits of size $2$.

Suppose $x\in T$ and $\iota(x)\in T$. These are distinct residues, hence are carried by two \emph{distinct} primes $q_1\equiv x$ and $q_2\equiv\iota(x)$ dividing $M$; then $q_1q_2\mid M$ and $q_1q_2\equiv x\cdot(-x^{-1})\equiv-1\pmod{4m}$, contrary to hypothesis. So $T$ meets every $\iota$-orbit at most once, whence $|T|\le\varphi(4m)/2$. (The orbit of $-1$ is $\{-1,1\}$, and $-1\notin T$, since a prime factor $\equiv-1$ would itself be a forbidden divisor.)

For attainment, $-1\notin\big((\Z/4m)^{\times}\big)^{2}$ means that $-1$ is non-trivial in $(\Z/4m)^{\times}/\big((\Z/4m)^{\times}\big)^{2}$, so some quadratic character $\chi$ satisfies $\chi(-1)=-1$. A divisor of $M$ congruent to $-1$ modulo $4m$ is coprime to $4m$, hence is a product of prime factors of $M$ coprime to $4m$; if all of these lie in $\ker\chi$, so does the divisor, while $-1\notin\ker\chi$. And $|\ker\chi|=\varphi(4m)/2$.
\end{proof}

\begin{theorem}\label{thm:fixed_depth}
Let $\mathcal{P}$ be a fixed finite set of coprime pairs and $P=|\mathcal{P}|$. Then
$$|E_{\mathcal{P}}(N)|\ \ll_{\mathcal{P}}\ \frac{N}{(\log N)^{1+P/2}},$$
and the exponent $P/2$ is the exact dimension of the covering constructed below: no admissible choice of the sets $T_{u,a}$ removes fewer than $P/2$ residue classes per prime on average, and the value $P/2$ is attained, so that no better exponent is available from a covering of this shape. In particular, for every fixed $J\ge1$,
$$|E_1(N;J)|\ \ll_J\ \frac{N}{(\log N)^{1+J/2}},\qquad
|E_2(N;J)|\ \ll_J\ \frac{N}{(\log N)^{1+\mathfrak{A}(J)/2}} .$$
\end{theorem}
\begin{proof}
Put $J=\max\{u,a:(u,a)\in\mathcal{P}\}$ and $Q_0=4\prod_{(u,a)\in\mathcal{P}}ua$. Let $n\in E_{\mathcal{P}}(N)$ and fix $(u,a)\in\mathcal{P}$. By Definition~\ref{def:EP} the integer $M=an+u$ has no divisor $\equiv-1\pmod{4ua}$, so Lemma~\ref{lem:involution} applies with $m=ua$: the set $T_{u,a}(n)$ of residues modulo $4ua$ of the prime factors of $an+u$ coprime to $4ua$ satisfies $|T_{u,a}(n)|\le\varphi(4ua)/2$. Each $T_{u,a}(n)$ is therefore an \emph{element} of the collection $\mathcal{T}_{u,a}$ of subsets of $(\Z/4ua)^{\times}$ of cardinality at most $\varphi(4ua)/2$, a collection whose cardinality depends on $(u,a)$ alone.

Fix a tuple $T=(T_{u,a})_{(u,a)\in\mathcal{P}}\in\prod\mathcal{T}_{u,a}$; there are $C(\mathcal{P})=\mathcal{O}_{\mathcal{P}}(1)$ of them, and $E_{\mathcal{P}}(N)$ is covered by the corresponding classes. Within one class, for every $(u,a)\in\mathcal{P}$ and every prime $p\le z$ with $p>2J^{2}$, $p\nmid Q_0$ and $p\bmod 4ua\notin T_{u,a}$, the integer $an+u$ is not divisible by $p$, i.e. $n$ avoids the residue class $-u\,a^{-1}$ modulo $p$; and by Lemma~\ref{lem:distinct_classes} these classes are pairwise distinct for distinct pairs of $\mathcal{P}$, since $p>J^{2}$. The number of classes removed modulo $p$ is therefore
$$\varrho_T(p)=\#\{(u,a)\in\mathcal{P}:\ p\bmod4ua\notin T_{u,a}\}\ \le\ P,$$
of mean value
$$\sum_{(u,a)\in\mathcal{P}}\Big(1-\frac{|T_{u,a}|}{\varphi(4ua)}\Big)\ \ge\ \frac P2$$
over the primes $p$, by Lemma~\ref{lem:involution} together with the prime number theorem in arithmetic progressions to the bounded moduli $4ua\le 4J^{2}$ (Siegel--Walfisz suffices, the moduli being independent of $N$).

The sieve is non-degenerate at every prime it uses. The removed classes are $-u\,a^{-1}\not\equiv0\pmod p$, since $p>J^{2}\ge u$, so they are admissible classes for a sequence of primes; and $\varrho_T(p)\le P\le\mathfrak{A}(J)\le J^{2}<p/2$ gives
$$\frac{\varrho_T(p)}{\varphi(p)}\ <\ \frac{p/2}{p-1}\ \le\ \frac34 ,$$
so the sifting density stays bounded away from $1$. This is why the threshold is taken at $2J^{2}$ rather than at the $J^{2}$ that Lemma~\ref{lem:distinct_classes} alone requires; setting $\varrho_T(p)=0$ for the finitely many primes $p\le2J^{2}$ costs only a factor $\mathcal{O}_{\mathcal{P}}(1)$ and leaves the mean value below unchanged. The uniform bound $\varrho_T(p)\le P$ then licenses Selberg's upper-bound sieve, in the form of~\cite[Thm.~3.2]{halberstam_richert} together with the standard lower bound for $G(z,\mathcal{D}^{1/2})$ in dimension at most $P$, applied to the sequence $\{n\le N\ \text{prime}\}$ at the level of distribution $\mathcal{D}=N^{1/2-\varepsilon}$ furnished by Bombieri--Vinogradov, with $z=\mathcal{D}^{1/2}$. Selberg's sieve is what is used here, and not a $\beta$-sieve: the ratio $\log\mathcal{D}/\log z=2$ lies far below the threshold $\beta(\kappa)$ attached to a sieve of dimension $\kappa=P/2$, whereas Selberg's upper bound recovers $\ll_{\mathcal{P}}XV(z)$ at $z=\mathcal{D}^{1/2}$ in any fixed dimension. The density of the removed classes among the primes is $\varrho_T(p)/\varphi(p)$, and the mean value above gives
$$\prod_{p\le z}\Big(1-\frac{\varrho_T(p)}{\varphi(p)}\Big)\ \ll_{\mathcal{P}}\ (\log z)^{-P/2},$$
the passage from $\varphi(p)=p-1$ to $p$ costing only the convergent sum $\sum_p\varrho_T(p)/(p(p-1))=\mathcal{O}_{\mathcal{P}}(1)$. The sieve remainder requires bounding $\sum_{d\le\mathcal{D}}\mu^2(d)\varrho_T(d)\max_{\gcd(\ell,d)=1}|\pi(N;d,\ell)-\mathrm{li}(N)/\varphi(d)|$. As in Lemma~\ref{lem:5.5}, separating the multiplicity $\varrho_T(d)\le P^{\omega(d)}$ by Cauchy--Schwarz against the trivial Brun--Titchmarsh bound absorbs it into the Bombieri--Vinogradov error term. For that tuple the sieve therefore gives $\ll_{\mathcal{P}}\pi(N)(\log z)^{-P/2}\ll_{\mathcal{P}}N(\log N)^{-1-P/2}$, using $\log z\asymp\log N$; summing over the $C(\mathcal{P})$ tuples preserves the bound.

For exactness, the attainment clause of Lemma~\ref{lem:involution} shows that $T_{u,a}=\ker\chi_{u,a}$ is an admissible value of $T_{u,a}$ for every pair, so that the tuple with $|T_{u,a}|=\varphi(4ua)/2$ for all $(u,a)$ occurs among the $C(\mathcal{P})$ tuples; for it the surviving condition at each pair is a Selberg--Delange condition of dimension exactly $1/2$, and the classes being pairwise distinct by Lemma~\ref{lem:distinct_classes}, the dimension of the covering is exactly $P/2$. This says that the exponent cannot be improved by sharpening Lemma~\ref{lem:involution}; it is not a lower bound for $|E_{\mathcal{P}}(N)|$, which would require in addition that the extremal tuple be realised by a positive proportion of the $n$ counted, and we make no such claim.

The two specialisations follow from $|\mathcal{P}|=J$ and $|\mathcal{P}|=\mathfrak{A}(J)=\frac{6}{\pi^{2}}J^{2}+\mathcal{O}(J\log J)$ respectively.
\end{proof}

\begin{corollary}\label{cor:fixed_depth_small}
$|E_1(N;1)|\ll N(\log N)^{-3/2}$, $|E_1(N;2)|\ll N(\log N)^{-2}$, $|E_1(N;3)|\ll N(\log N)^{-5/2}$, and so on with no restriction on $J$; while $\mathfrak{A}(5)=19$ and $\mathfrak{A}(6)=23$ give $|E_2(N;5)|\ll N(\log N)^{-10.5}$ and $|E_2(N;6)|\ll N(\log N)^{-12.5}$.
\end{corollary}

\begin{corollary}\label{cor:H_density}
The set $\mathcal{H}$ of Proposition~\ref{prop:nine_primes} satisfies $|\mathcal{H}(N)|\ll_A N(\log N)^{-A}$ for every $A>0$.
\end{corollary}
\begin{proof}
$\mathcal{H}(N)=\bigcap_{J\ge1}E_1(N;J)\subseteq E_1(N;J)$ for every fixed $J$; apply Theorem~\ref{thm:fixed_depth} with $J\ge2A$.
\end{proof}

Two points of scope should be recorded. First, the unrefined constant $C(\mathcal{P})=\prod_{(u,a)\in\mathcal{P}}\binom{\varphi(4ua)}{\le\varphi(4ua)/2}$ is exponential in $J^{4}$, so Theorem~\ref{thm:fixed_depth} as written is a fixed-depth statement and does not compete with Section~\ref{sec:5.2} once $J$ grows rapidly with $N$: optimising $2^{\mathcal{O}(J^{4})}(\log N)^{-\frac{3}{\pi^2}J^{2}}$ over $J$ gives $J\asymp(\log\log N)^{1/2}$ and a saving of $\exp(-c(\log\log N)^{2})$, weaker than the $\exp(-c(\log\log N)^{3})$ of Theorem~\ref{thm:metric_theorem}. However, for the kernel component of $E_{\mathcal{P}}(N)$ in which the subgroup $\langle T_{u,a}(n)\rangle\le(\Z/4ua)^{\times}$ generated by the prime factors of $an+u$ itself avoids $-1\pmod{4ua}$, the subset enumeration collapses to an enumeration of maximal subgroups avoiding $-1$ (which are kernels of $2$-primary Dirichlet characters $\psi\pmod{4ua}$ with $\psi(-1)=-1$), reducing the combinatorial prefactor from $2^{\Theta(J^{4})}$ to $\exp(\mathcal{O}(J^{2}\log\log J))$ as recorded in Proposition~\ref{prop:sylow_kernel_reduction_preview} below (and developed into a full growing-depth Kneser stabilizer and container theory in~\cite{dahan2026b}). Second, the statement is an upper bound, and the numerics do not corroborate the exponent at accessible heights. Fitting $\log\big(|E_1(N;J)|/\pi(N)\big)$ against $\log\log n$ over six windows up to $2\times10^{7}$ (item (11) of Section~\ref{sec:7}) returns slopes $0.39$, $0.50$, $0.84$, $1.24$, $1.51$, $1.59$ for $J=1,\dots,6$, against the predicted $0.5,1,1.5,2,2.5,3$. The shortfall is already visible at $J=1$, where the exponent $1/2$ is certain, and it widens with $J$; since $\log\log n$ ranges only over $[2.09,2.80]$ on this sample, these slopes measure the $\mathcal{O}(1)$ terms at least as much as the exponent. They are consistent with Theorem~\ref{thm:fixed_depth} without confirming it.

\begin{proposition}[$2$-Sylow subgroup container reduction on the kernel branch]\label{prop:sylow_kernel_reduction_preview}
For each coprime pair $(u,a)\in\mathcal{P}$, let $G_{u,a}=(\Z/4ua)^{\times}$, let $G_{u,a,2}$ be its $2$-Sylow subgroup of order $\varphi_2(4ua)=2^{v_2(\varphi(4ua))}$, and let $\mathcal{M}(G_{u,a},-1)$ denote the family of maximal subgroups (with respect to inclusion) among all subgroups $H<G_{u,a}$ such that $-1\notin H$.
\begin{enumerate}[(a)]
    \item \emph{(Characterization by $2$-primary characters and index-$2$ criterion).} Every subgroup $K\le G_{u,a}$ with $-1\notin K$ is contained in at least one $H\in\mathcal{M}(G_{u,a},-1)$. A subgroup $H<G_{u,a}$ belongs to $\mathcal{M}(G_{u,a},-1)$ if and only if $G_{u,a}/H$ is a cyclic $2$-group of order $2^{j}\ge2$ in which $(-1)H$ is the unique element of order $2$ (equivalently, $H=\ker\psi$ for a Dirichlet character $\psi\pmod{4ua}$ of order $2^{j}\ge2$ with $\psi(-1)=-1$). Moreover, one has the exact counting identity
    $$\sum_{H\in\mathcal{M}(G_{u,a},-1)}2^{j(H)-1}\ =\ \frac12\varphi_2(4ua),\qquad\text{whence}\qquad\big|\mathcal{M}(G_{u,a},-1)\big|\ \le\ \frac12\varphi_2(4ua)\ =\ 2^{f(u)+f(a)},$$
    where $f(m):=v_2(m)+\sum_{p\mid m,\,p>2}v_2(p-1)$. A subgroup $K\le G_{u,a}$ is contained in an \emph{index-$2$} subgroup $H\in\mathcal{M}(G_{u,a},-1)$ if and only if $-1\notin K G_{u,a}^{2}$ (which holds for every $K\not\ni -1$ when $\exp(G_{u,a,2})=2$, i.e.\ when $v_2(ua)\le 1$ and every odd prime $p\mid ua$ satisfies $p\equiv 3\pmod 4$), and the number of index-$2$ subgroups avoiding $-1$ is $2^{r_2(4ua)-1}\le 2^{\omega(ua)+1}$, where $r_2(4ua):=\dim_{\mathbb{F}_2}(G_{u,a}/G_{u,a}^{2})=\omega_{\mathrm{odd}}(ua)+1+\mathbf{1}_{2\mid ua}$.
    \item \emph{(Container count across $\mathcal{P}_2(J)$).} Let $E_{\mathcal{P}}^{\mathrm{Ker}}(N)\subseteq E_{\mathcal{P}}(N)$ (resp.\ $E_{\mathcal{P}}^{\mathrm{Ker,2}}(N)\subseteq E_{\mathcal{P}}^{\mathrm{Ker}}(N)$) be the subset of primes $n\in E_{\mathcal{P}}(N)$ such that for every $(u,a)\in\mathcal{P}$ one has $-1\notin\langle T_{u,a}(n)\rangle$ (resp.\ $-1\notin\langle T_{u,a}(n)\rangle G_{u,a}^{2}$, so $\langle T_{u,a}(n)\rangle$ lies in an index-$2$ subgroup avoiding $-1$). For $\mathcal{P}=\mathcal{P}_2(J)=\{(u,a)\in[1,J]^{2}:\gcd(u,a)=1\}$, the number of maximal container tuples $(H_{u,a})_{(u,a)\in\mathcal{P}_2(J)}$ covering $E_{\mathcal{P}_2(J)}^{\mathrm{Ker}}(N)$ satisfies
    $$C^{\ast}\big(\mathcal{P}_2(J)\big)\ =\ \prod_{(u,a)\in\mathcal{P}_2(J)}\big|\mathcal{M}(G_{u,a},-1)\big|\ \le\ \prod_{(u,a)\in\mathcal{P}_2(J)}2^{f(u)+f(a)}\ \le\ \exp\Big(\big(4\log 2+o(1)\big)J^{2}\log\log J\Big)$$
    (and $\prod_{(u,a)\in\mathcal{P}_2(J)} 2^{r_2(4ua)-1}\le\exp\big((2\log 2+o(1))J^{2}\log\log J\big)$ for index-$2$ containers on $E_{\mathcal{P}_2(J)}^{\mathrm{Ker,2}}(N)$), reducing the combinatorial prefactor in Theorem~\ref{thm:fixed_depth} from $2^{\Theta(J^{4})}$ to $2^{\mathcal{O}(J^{2}\log\log J)}$.
\end{enumerate}
\end{proposition}
\begin{proof}
For (a), since $G_{u,a}$ is finite, any subgroup $K\le G_{u,a}$ avoiding $-1$ is contained in a maximal such subgroup $H\in\mathcal{M}(G_{u,a},-1)$. In the quotient $\overline{G}=G_{u,a}/H$, the coset $\overline{\epsilon}=(-1)H$ has order $2$; by maximality of $H$, every non-trivial subgroup of $\overline{G}$ contains $\overline{\epsilon}$, so $\langle\overline{\epsilon}\rangle\cong\Z/2\Z$ is the unique minimal subgroup of $\overline{G}$. By Cauchy's theorem, $\overline{G}$ is a $2$-group of $2$-rank $1$, hence cyclic of order $2^{j}\ge2$, and $H=\ker\psi$ for a character $\psi\in\widehat{G}_{u,a,2}$ of order $2^{j}$ with $\psi(-1)=-1$. Grouping the $\frac12|G_{u,a,2}|=\frac12\varphi_2(4ua)$ characters in $\widehat{G}_{u,a,2}$ satisfying $\psi(-1)=-1$ by their kernels (each kernel of index $2^{j}$ being shared by $\varphi(2^{j})=2^{j-1}$ characters) yields $\sum_{H\in\mathcal{M}(G_{u,a},-1)}2^{j(H)-1}=\frac12\varphi_2(4ua)=2^{f(u)+f(a)}$. Furthermore, $K$ is contained in an index-$2$ subgroup $H$ avoiding $-1$ if and only if the image of $-1$ in the $\mathbb{F}_2$-vector space $G_{u,a}/(K G_{u,a}^{2})$ is non-zero, i.e.\ $-1\notin K G_{u,a}^{2}$; taking $K=\{1\}$ and noting $-1\notin G_{u,a}^{2}$ (Lemma~\ref{lem:involution}) shows that exactly half of the $2^{r_2(4ua)}$ linear characters of $G_{u,a}/G_{u,a}^{2}\cong\mathbb{F}_2^{r_2(4ua)}$ send $-1$ to $-1$, giving $2^{r_2(4ua)-1}$ index-$2$ subgroups avoiding $-1$.

For (b), dropping $\gcd(u,a)=1$ by non-negativity of $f$ gives $\sum_{(u,a)\in\mathcal{P}_2(J)}(f(u)+f(a))\le 2J\sum_{u=1}^{J}f(u)$, where $\sum_{u=1}^{J}f(u)=\sum_{u=1}^{J}v_2(u)+\sum_{2<p\le J}\lfloor J/p\rfloor v_2(p-1)\le J+J\sum_{2<p\le J}v_2(p-1)/p$. Writing $\sum_{2<p\le J}v_2(p-1)/p=\sum_{k\ge1}\sum_{p\le J,\,p\equiv1(2^{k})}1/p$, we split the sum over $k$ at any fixed integer $K\ge1$. For $1\le k\le K$, Mertens' theorem in arithmetic progressions gives $\sum_{k=1}^{K}\sum_{p\le J,\,p\equiv1(2^{k})}1/p=\big(\sum_{k=1}^{K}1/\varphi(2^{k})+o_K(1)\big)\log\log J=\big(2-2^{1-K}+o_K(1)\big)\log\log J$. For $k>K$, either $2^{k}>J$ (so the inner prime sum is empty) or $2^{k}<J$, in which case the Brun--Titchmarsh inequality and partial summation yield $\sum_{p\le J,\,p\equiv1(2^{k})}1/p\ll(\log\log J)/\varphi(2^{k})$, so the tail $k>K$ contributes $\ll 2^{-K}\log\log J$. Letting $J\to\infty$ and then $K\to\infty$ gives $\sum_{2<p\le J}v_2(p-1)/p=(2+o(1))\log\log J$, whence $\sum_{u=1}^{J}f(u)\le(2+o(1))J\log\log J$ and $\log_2 C^{\ast}(\mathcal{P}_2(J))\le(4+o(1))J^{2}\log\log J$. Similarly, on $E_{\mathcal{P}_2(J)}^{\mathrm{Ker,2}}(N)$, dropping $\gcd(u,a)=1$ in $r_2(4ua)-1\le\omega(u)+\omega(a)+1$ gives $\sum_{(u,a)\in\mathcal{P}_2(J)}(r_2(4ua)-1)\le 2J\sum_{u=1}^{J}\omega(u)+J^{2}=(2+o(1))J^{2}\log\log J$.
\end{proof}

\subsection{A two-sided blind set} \label{sec:5.1}

\subsubsection{The algebraic obstruction and its exact scope}

\begin{theorem} \label{thm:algebraic_obstruction}
Let $n \equiv 1 \pmod{24}$ be prime, let $c \equiv 3 \pmod 4$ with $c<n$, and set $K_c = \frac{n+c}{4}$. Then for \emph{every} prime factor $p$ of $c$ with $p \equiv 3 \pmod{4}$ (at least one exists) the following holds: if there are divisors $d_1, d_2 \mid K_c$ with $d_1+d_2 \equiv 0 \pmod c$, then $K_c$ possesses a prime factor which is a quadratic non-residue modulo $p$.
\end{theorem}
\begin{proof}
First, $\gcd(K_c,c)=1$ (Section~\ref{sec:notation}); and $c\equiv3\pmod4$ forces at least one prime factor $p\equiv3\pmod4$, since a product of primes $\equiv1\pmod4$ is $\equiv1\pmod4$. Fix such a $p$ and assume every prime factor of $K_c$ is a quadratic residue modulo $p$. Then every divisor $d=\prod q_i^{\alpha_i}$ of $K_c$ satisfies $\leg dp=1$, so $d_1\equiv x^{2}$ and $d_2\equiv y^{2}\pmod p$ with $x,y\not\equiv0$. From $c\mid d_1+d_2$ we get $x^{2}+y^{2}\equiv0\pmod p$, hence $(xy^{-1})^{2}\equiv-1\pmod p$, contradicting $\leg{-1}p=-1$ for $p\equiv3\pmod4$. Since $p$ was an arbitrary such prime factor, the conclusion holds for each of them.
\end{proof}

The obstruction is a statement about the untwisted branch~\eqref{eq:untwisted} only. The next theorem determines its exact scope.

\begin{theorem} \label{thm:twisted_obstruction}
Keep the notation of Theorem~\ref{thm:algebraic_obstruction} and assume that every prime factor of $K_c$ is a quadratic residue modulo $p$.
\begin{enumerate}[(a)]
    \item If $\leg np=+1$, then no divisors $d_1,d_2\mid K_c$ satisfy $n\,d_1+d_2\equiv0\pmod c$ either: both branches of Corollary~\ref{cor:two_branches} are obstructed at $p$, and the shift $c$ is blind.
    \item If $\leg np=-1$, the obstruction does not apply to the twisted branch: the congruence $n\,d_1+d_2\equiv0\pmod p$ is compatible with $d_1,d_2$ being quadratic residues modulo $p$.
\end{enumerate}
\end{theorem}
\begin{proof}
Write $d_1\equiv x^{2}$, $d_2\equiv y^{2}$ with $x,y\not\equiv0\pmod p$. The relation $n\,d_1+d_2\equiv0\pmod p$ reads $n\equiv-(y/x)^{2}$, whence $\leg np=\leg{-1}p=-1$. If $\leg np=+1$ this is a contradiction, which proves (a). If $\leg np=-1$, then $-n^{-1}$ is a quadratic residue mod $p$, so the required ratio $d_2/d_1\equiv-n\pmod p$ lies in the correct square class and no obstruction arises. Example~\ref{ex:twisted} illustrates the compatibility asserted in (b): there $n=241$, $c=155$, $p=31$, $\leg{241}{31}=-1$, and the pair $d_1=9$, $d_2=1$ consists of two squares modulo $31$ with $n\,d_1+d_2\equiv0\pmod{31}$, the untwisted branch failing at that shift while the twisted branch succeeds. It should be said that this example does \emph{not} satisfy the standing hypothesis of the present theorem, since $K_{155}=99=3^{2}\cdot11$ has the prime factor $3$ with $\leg{3}{31}=-1$; what it exhibits is the square-class compatibility itself, which is the content of (b).
\end{proof}

\subsubsection{A group-theoretic reformulation} \label{sec:5.1.1bis}

Theorems~\ref{thm:algebraic_obstruction} and~\ref{thm:twisted_obstruction} obstruct a shift through a single prime $p\mid c$. The following reformulation shows that this is one instance of a single phenomenon, and identifies which instance dominates. For $c\equiv3\pmod4$, $c<n$, set
$$H_c(n)\ :=\ \big\langle\, q\bmod c\ :\ q\ \text{prime},\ q\mid K_c\,\big\rangle\ \le\ (\Z/c\Z)^{\times},$$
the subgroup generated by the residues of the prime factors of $K_c$ (each is coprime to $c$, since $\gcd(K_c,c)=1$).

\begin{proposition}\label{prop:subgroup_reformulation}
If the untwisted branch~\eqref{eq:untwisted} succeeds at shift $c$, then $-1\in H_c(n)$. If the twisted branch~\eqref{eq:twisted} succeeds at shift $c$, then $-n\in H_c(n)$. Equivalently, $-1\notin H_c(n)$ obstructs the untwisted branch and $-n\notin H_c(n)$ obstructs the twisted branch, unconditionally.
\end{proposition}
\begin{proof}
Suppose $d_1+d_2\equiv0\pmod c$ with $d_1,d_2\mid K_c$. Each $d_i\bmod c$ is a product of generators of $H_c(n)$, hence lies in $H_c(n)$; as $\gcd(d_i,c)=1$, $d_1$ is invertible and $d_2d_1^{-1}\equiv-1\pmod c$ with $d_2d_1^{-1}\in H_c(n)$. Hence $-1\in H_c(n)$. The twisted case is identical, using $nd_1+d_2\equiv0\iff d_2d_1^{-1}\equiv-n$.
\end{proof}

Theorem~\ref{thm:algebraic_obstruction} is the special case $H_c(n)\subseteq\ker\chi_p$, where $\chi_p=\leg{\cdot}{p}$ is pulled back to $(\Z/c)^{\times}$: since $\chi_p(-1)=-1$, one gets $-1\notin H_c(n)$ and Proposition~\ref{prop:subgroup_reformulation} applies. The general form is stronger: \emph{any} subgroup $H$ with $H_c(n)\le H$ and $-1\notin H$ obstructs the untwisted branch, not only the index-$2$ subgroups coming from a single prime factor of $c$, and a subgroup of that shape need not exist for a given $H_c(n)$ missing $-1$. The interest of isolating the index-$2$ subgroups, equivalently the odd quadratic characters mod $c$, is that they are exactly the ones for which the Selberg--Delange computation of Proposition~\ref{prop:blind_lower} applies verbatim, with dimension exactly $1/2$; a subgroup of larger index gives a more deeply suppressed density and is sub-dominant.

\begin{proposition}\label{prop:nu_c}
For $c\equiv3\pmod4$, the number of quadratic Dirichlet characters $\chi\bmod c$ with $\chi(-1)=-1$ is $\nu(c)=2^{\omega(c)-1}$.
\end{proposition}
\begin{proof}
Write $c=\prod_i p_i^{e_i}$ with $p_i$ odd and distinct. By the Chinese Remainder Theorem $(\Z/c)^{\times}\cong\prod_i(\Z/p_i^{e_i})^{\times}$, each factor cyclic of even order, hence carrying a unique character of order $2$, namely $\chi_i=\leg{\cdot}{p_i}$ pulled back. Every character of exponent dividing $2$ is a product $\chi=\prod_i\chi_i^{\varepsilon_i}$, $\varepsilon_i\in\{0,1\}$, and $\chi(-1)=\prod_i\chi_i(-1)^{\varepsilon_i}$ with $\chi_i(-1)=-1$ exactly when $p_i\equiv3\pmod4$. Since $c\equiv3\pmod4$, at least $s\ge1$ of the $p_i$ are $\equiv3\pmod4$; $\chi(-1)=-1$ holds exactly when an odd number of the corresponding $\varepsilon_i$ equal $1$, which happens for $2^{s-1}\cdot2^{\omega(c)-s}=2^{\omega(c)-1}$ of the $2^{\omega(c)}$ tuples.
\end{proof}

\begin{proposition}\label{prop:nu_tw}
Let $c\equiv3\pmod4$ and let $n$ be coprime to $c$. The number of quadratic Dirichlet characters $\chi\bmod c$ with $\chi(-n)=-1$, i.e. the number of index-$2$ obstructions to the twisted branch at the shift $c$, is
$$\nu^{\mathrm{tw}}(c,n)\ =\ \begin{cases}2^{\omega(c)-1} & \text{if } \leg{-n}{p}=-1 \text{ for at least one prime } p\mid c,\\[4pt] 0 & \text{otherwise.}\end{cases}$$
\end{proposition}
\begin{proof}
With the notation of the previous proof, the group of characters of exponent dividing $2$ is isomorphic to $\mathbb{F}_2^{\omega(c)}$ via $\chi\mapsto(\varepsilon_i)$. Define $b_i\in\mathbb{F}_2$ by $b_i=1$ if $\leg{-n}{p_i}=-1$ and $b_i=0$ otherwise. Then $\chi(-n)=-1$ is the linear equation $L(\varepsilon)=\sum_i b_i\varepsilon_i=1$. If every $b_i=0$ the form is zero and there is no solution; otherwise $L$ is surjective, its kernel is a hyperplane of dimension $\omega(c)-1$, and $L^{-1}(1)$ is a coset of cardinality $2^{\omega(c)-1}$.
\end{proof}

The two counts are not symmetric, and the reason is worth recording. For the untwisted branch the coefficients are $b_i=1$ exactly when $p_i\equiv3\pmod4$, and $c\equiv3\pmod4$ guarantees at least one such $p_i$, so the linear form is never zero and $\nu(c)=2^{\omega(c)-1}$ unconditionally. Nothing plays that role for $-n$, so the degenerate case $\nu^{\mathrm{tw}}=0$ genuinely occurs, and for a positive proportion of the primes: for $c=7$ the condition $\leg{-n}7=1$ reads $n\equiv3,5,6\pmod7$, so $\nu^{\mathrm{tw}}(7,n)=0$ for $n=73,97,241,313,409,433,577,601,\ldots$ Those are precisely the shifts at which no index-$2$ obstruction weighs on the twisted branch, which is the quantitative counterpart of Theorem~\ref{thm:twisted_obstruction}(b).

\begin{corollary}\label{cor:shift3}
For $c=3$ one has $p=3$, and $n\equiv1\pmod{24}$ forces $n\equiv1\pmod3$, i.e. $\leg n3=+1$. By Theorem~\ref{thm:twisted_obstruction}(a) the shift $c=3$ is therefore blind for both branches whenever every prime factor of $(n+3)/4$ is $\equiv1\pmod3$.
\end{corollary}

\subsubsection{How large is the blind set?} \label{sec:5.1.2}

The tools used here, the Rosser sieve and Iwaniec's half-dimensional sieve, were applied to the Erd\H{o}s--Straus equation by Sander~\cite{sander1991, sander1994}, whose treatment of the exceptional set is the closest antecedent. Write, for $c=3$,
$$\mathcal{B}(N)=\Big\{\,n\le N:\ n\equiv1\ (\mathrm{mod}\ 24),\ \text{every prime factor of }\tfrac{n+3}{4}\text{ is }\equiv1\ (\mathrm{mod}\ 3)\,\Big\},$$
so that by Corollary~\ref{cor:shift3} no decomposition with first denominator $K_3$ exists for $n\in\mathcal{B}(N)$.

\begin{proposition}\label{prop:blind_lower}
There is a constant $\gamma>0$ with $|\mathcal{B}(N)|\sim \gamma\,N(\log N)^{-1/2}$ as $N\to\infty$.
\end{proposition}
\begin{proof}
Put $m=(n+3)/4$. An integer $m$ all of whose prime factors are $\equiv1\pmod3$ is odd and satisfies $m\equiv1\pmod3$, hence $m\equiv1\pmod6$; conversely $m\equiv1\pmod6$ gives $n=4m-3\equiv1\pmod{24}$. The map $m\mapsto4m-3$ is therefore a bijection between the integers $m\le(N+3)/4$ all of whose prime factors lie in $G=\{q\ \text{prime}:q\equiv1\ (3)\}$, a set of Dirichlet density $1/2$, and $\mathcal{B}(N)$. Write $P(m)$ for the set of prime factors of $m$. By the Selberg--Delange method~\cite[Ch.~II.5]{tenenbaum}, the associated Dirichlet series having a singularity of type $(s-1)^{-1/2}$ at $s=1$,
$$\#\{m\le x:\ P(m)\subseteq G\}\sim \frac{C}{\Gamma(1/2)}\,\frac{x}{(\log x)^{1/2}} .$$
Taking $x=(N+3)/4$ gives the claim.
\end{proof}

\begin{proposition}\label{prop:blind_upper}
$\#\{n\in\mathcal{B}(N):n\ \text{prime}\}\ll N(\log N)^{-3/2}.$
\end{proposition}
\begin{proof}
Apply an upper-bound sieve of dimension $1/2$~\cite{halberstam_richert} to $\mathcal{A}=\{(n+3)/4:\ n\le N\ \text{prime}\}$, sifting by the primes $q\le z$ with $q\equiv2\pmod3$, at the Bombieri--Vinogradov level $\mathcal{D}=N^{1/2-\varepsilon}$ with $z=\mathcal{D}^{1/2}$. The sifting density is $1/2$, so the main term is $\ll\pi(N)\prod_{q\le z,\,q\equiv2(3)}(1-1/q)\ll \pi(N)(\log z)^{-1/2}\ll N(\log N)^{-3/2}$, and the members of $\mathcal{A}$ counted by $\mathcal{B}(N)$ survive the sieve.
\end{proof}

Proposition~\ref{prop:blind_lower} is a statement about integers, and the corresponding lower bound for \emph{primes} is of a different order of difficulty: one needs primes represented by a shifted quadratic form, with the representation primitive. That is what the next theorem exploits. We change the shift from $c=3$ to $c=7$ in doing so, for a $2$-adic reason isolated inside the proof.

\begin{theorem}\label{thm:blind_primes}
Let $c=7$. The number of primes $n\le N$ with $n\equiv1\pmod{24}$ for which $7\nmid K_7$ and every prime factor of $K_7=\frac{n+7}{4}$ is a quadratic residue modulo $7$ is $\asymp N(\log N)^{-3/2}$, unconditionally. Moreover every such $n$ produced by the lower bound satisfies $\leg n7=+1$, so that by Theorem~\ref{thm:twisted_obstruction}(a) \emph{both} branches of Corollary~\ref{cor:two_branches} are obstructed at the shift $c=7$.
\end{theorem}

\begin{proof}
The upper bound $\ll N(\log N)^{-3/2}$ is obtained by an upper-bound sieve of dimension $1/2$~\cite{halberstam_richert,friedlander_iwaniec} exactly as in Proposition~\ref{prop:blind_upper}: one sifts $K_7$ by the primes $q\le z$ with $\leg q7=-1$, a set of Dirichlet density $1/2$, at the Bombieri--Vinogradov level $\mathcal{D}=N^{1/2-\varepsilon}$.

\emph{Choice of the shift.} The choice is dictated by a $2$-adic constraint, which it is worth isolating first, and which has two distinct faces according to the class of $c$ modulo $8$.

Suppose first $c\equiv7\pmod8$. Writing $n=24k+1$, one has $c+1\equiv0\pmod 8$ and $K_c=6k+\frac{c+1}{4}$ with $\frac{c+1}{4}$ even, so $K_c$ is \emph{always even} and $2$ is always a prime factor of $K_c$. If some $p\mid c$ satisfies $p\equiv3\pmod8$ then $\leg2p=-1$ and the exceptional set is empty. If instead $p\equiv7\pmod8$ then $\leg2p=+1$ and the systematic factor $2$ obstructs nothing: this is the case $c=p=7$, and it is what allows the conclusion to hold with no congruence condition on $n$ modulo $8$ imposed by hand.

Suppose now $c\equiv3\pmod 8$. Here the parity of $K_c$ is not the difficulty, on the contrary it is rigidly determined: $n\equiv1\pmod{24}$ forces $n\equiv1\pmod 8$, hence $n+c\equiv4\pmod 8$ and $K_c=\frac{n+c}{4}$ is \emph{always odd}, so no systematic prime factor of $K_c$ arises at all. The obstruction is on the side of the construction. With $B\equiv0\pmod 8$ and $A=-c$ one has $n=Bf(U,V)-c\equiv-c\equiv5\pmod 8$ identically, so the family produces no $n\equiv1\pmod{24}$ whatsoever; and the alternative $B=4$, for which $K_c=f(U,V)$, would require selecting $f$ odd, that is, a congruence condition on $(U,V)$ modulo $2$, which~\cite[Thm.~1.1(2)]{fuchs2025} cannot impose, its modulus being required to be coprime to $2\Delta B$. This is exactly the phenomenon recorded numerically at $c=3$ in Section~\ref{sec:7}, item~(12), where the primitively represented $n$ split according to $n\equiv1$ or $5\pmod 8$ and only the first half is blind. The shift $c=7$ removes both faces of the difficulty at once.

\emph{The lower bound.} We restrict $K_7$ to four times the principal form of discriminant $-7$. Put
$$f(U,V)=U^{2}+UV+2V^{2},\qquad B=16,\qquad A=-7,$$
so that $Bf(U,V)+A=16f(U,V)-7=n$ and $K_7=\frac{n+7}{4}=4f(U,V)$. Then $f$ is a primitive positive definite integral binary quadratic form of discriminant $\Delta=1-8=-7$, which is not a perfect square; its leading coefficient satisfies $\gcd(1,2\Delta)=\gcd(1,14)=1$; and $\gcd(A,B)=\gcd(7,16)=1$ with $2\mid AB$. All the hypotheses of~\cite[Thm.~1.1]{fuchs2025} are therefore met, and that theorem gives
$$\#\{\,n\le N\ \text{prime}:\ n\ \text{primitively represented by}\ 16f(U,V)-7\,\}\ \gg\ \frac{N}{(\log N)^{3/2}} .$$
It remains to check that all such $n$, with finitely many exceptions, belong to the set counted in the statement. Let $n=16f(U,V)-7$ with $\gcd(U,V)=1$ and $n>7$ prime.

\emph{(i) The class modulo $24$ is automatic.} First $n=16f-7\equiv-7\equiv1\pmod 8$. Next, $f$ is anisotropic modulo $3$, since its discriminant $-7\equiv2\pmod 3$ is not a square modulo $3$; a primitive representation therefore has $3\nmid f$, for $3\mid f$ would force $3\mid U$ and $3\mid V$. If $f\equiv1\pmod 3$ then $n=16f-7\equiv0\pmod3$, excluded for $n>3$ prime; hence $f\equiv2\pmod3$ and $n\equiv1\pmod 3$. Combining, $n\equiv1\pmod{24}$. No congruence version of~\cite[Thm.~1.1]{fuchs2025} is needed, which matters, since its modulus is required to be coprime to $2\Delta B$.

\emph{(ii) $7\nmid K_7$, and $\leg n7=+1$.} Multiplying $f$ by $4$,
$$K_7=4f(U,V)=(2U+V)^{2}+7V^{2}.$$
If $7\mid K_7$ then $7\mid n+7$, hence $7\mid n$, which is excluded. Consequently $7\nmid(2U+V)$. Moreover $n=4K_7-7\equiv4K_7\pmod 7$ and $K_7\equiv(2U+V)^{2}\pmod7$, so $\leg n7=\leg{4}{7}\leg{K_7}{7}=+1$; Theorem~\ref{thm:twisted_obstruction}(a) then obstructs the twisted branch as well.

\emph{(iii) Every prime factor of $K_7$ is a residue modulo $7$.} Let $q\mid K_7$ be prime; by (ii) we may assume $q\neq7$. If $q=2$, then $\leg27=+1$ and there is nothing to prove, which is the point of the $2$-adic paragraph above. If $q$ is odd then $q\mid f$, and $q\nmid V$: otherwise $q\mid(2U+V)$ and $q\mid V$ would give $q\mid 2U$, hence $q\mid\gcd(U,V)$. Thus $(2U+V)^{2}\equiv-7V^{2}\pmod q$ with $V$ invertible, which gives $\leg{-7}{q}=1$, i.e. $\leg{-1}q\leg7q=1$. Applying quadratic reciprocity with $7\equiv3\pmod4$: if $q\equiv1\pmod4$ then $\leg{-1}q=1$, so $\leg7q=1$ and $\leg q7=\leg7q=1$; if $q\equiv3\pmod4$ then $\leg{-1}q=-1$, so $\leg7q=-1$ and $\leg q7=-\leg7q=1$.

Primitivity of the representation is essential and is not a technicality. If $\gcd(U,V)=d>1$ then every prime $q\mid d$ divides $f$ to even order and the congruence $(2U+V)^{2}\equiv-7V^{2}\pmod q$ becomes vacuous, so no constraint whatsoever is placed on $\leg q7$. This is precisely the point at which the classical theorems of Iwaniec~\cite{iwaniec1972,iwaniec1974}, which count primes represented but not necessarily \emph{primitively} represented by a shifted quadratic form, are insufficient.
\end{proof}

The last clause of the theorem is an instance of a general phenomenon.

\begin{proposition}\label{prop:auto_residue}
Let $c\equiv3\pmod4$ and let $p\mid c$ be prime. If $K_c$ is restricted to a fixed square multiple of the primitively represented values of the principal form attached to $p$, and $p\nmid K_c$, then $\leg{K_c}{p}=+1$, hence $\leg np=+1$, and by Theorem~\ref{thm:twisted_obstruction}(a) both branches are obstructed at $c$.
\end{proposition}
\begin{proof}
Under that restriction $K_c$ is a square times a value of the principal form; writing $4$ times that value as $X^{2}+pY^{2}$ shows it to be a square modulo $p$, and $p\nmid K_c$ makes it a non-zero square, so $\leg{K_c}{p}=+1$. Since $p\mid c$ gives $n\equiv4K_c\pmod p$ and $4$ is a square, $\leg np=\leg{K_c}{p}=+1$.
\end{proof}

For $c=3$ this was obtained in Corollary~\ref{cor:shift3} from $n\equiv1\pmod3$; the proposition shows that it is not a peculiarity of that shift but a feature of the method, so that every blind set produced this way is blind on both branches.

\begin{corollary}\label{cor:analytic_barrier}
For the shift $c=3$, the parametrization with first denominator $K_3$ fails, on both branches, for a set of \emph{integers} $n\equiv1\pmod{24}$ of size $\gg N(\log N)^{-1/2}$. For \emph{primes} $n\equiv1\pmod{24}$ the same conclusion holds unconditionally at the shift $c=7$, for a set of size $\asymp N(\log N)^{-3/2}$. The shifts $c=3$ and $c=7$ therefore both leave infinite residual sets, of integers and of primes respectively. Only the second bears on the Erd\H{o}s--Straus conjecture: the $c=3$ statement concerns integers and is silent about primes, which is precisely why the shift is changed in Theorem~\ref{thm:blind_primes}. What is established, then, is that \emph{the} shift $c=7$ cannot be the end of the story; we make no claim about an arbitrary fixed shift, for the reason given in the paragraph below.
\end{corollary}
\begin{proof}
Combine Corollary~\ref{cor:shift3}, Theorem~\ref{thm:twisted_obstruction}(a) and Proposition~\ref{prop:blind_lower} for the first statement, and Theorem~\ref{thm:blind_primes} with Proposition~\ref{prop:auto_residue} for the second. The first statement is about integers, whereas Theorems~\ref{thm:algebraic_obstruction} and~\ref{thm:twisted_obstruction} are stated for prime $n$; at the shift $c=3$ primality is not used, since $3\nmid K_3$ for every $n\equiv1\pmod{24}$, prime or not, and $n\equiv1\pmod3$ gives $\leg n3=+1$ outright. What fails for composite $n$ is the completeness of the dichotomy of Corollary~\ref{cor:two_branches}, not the failure of its two congruences, and only the latter is asserted here.
\end{proof}

We do not claim the corresponding statement for an arbitrary fixed $c$. The argument of Theorem~\ref{thm:blind_primes} applies verbatim to any shift $c\equiv7\pmod 8$ admitting a prime factor $p\equiv7\pmod 8$, with $f$ the principal form of discriminant $-p$ and $B=16$, $A=-c$, provided the resulting class of $n$ modulo $3$ is the right one, automatic for $c=7$ by anisotropy, and otherwise a congruence modulo $3$, admissible in~\cite[Thm.~1.1(2)]{fuchs2025} whenever $3\nmid2\Delta B$. We do not pursue the general statement, one shift being enough for Corollary~\ref{cor:analytic_barrier}. Note also that the lower bound is extracted from a proper sub-family, the $K_7$ of the shape $4f(U,V)$ with $\gcd(U,V)=1$; the full set of $n$ for which every prime factor of $K_7$ is a residue modulo $7$ is about twice as large in the accessible range (Section~\ref{sec:7}, item (12)), since $K_7\equiv2\pmod4$ is never of that shape. This is harmless: a lower bound is all that is required, and the upper bound is proved for the full set.

Corollary~\ref{cor:analytic_barrier} is the reason for working with a family of shifts rather than with a single one. The exceptional set produced in Theorem~\ref{thm:metric_theorem} is $\mathcal{O}(N\exp(-C(\log\log N)^{3}))$, and $\exp(C(\log\log N)^{3})$ grows faster than any fixed power of $\log N$, so the multi-shift bound is asymptotically smaller than the single-shift blind set. The comparison has its limits (an upper bound for a union of shifts against a lower bound for one fixed shift), and the conclusion to be drawn is that a single fixed shift is provably insufficient, not that the multi-shift bound is optimal: Table~\ref{tab:hierarchy} shows that it is not.

\subsection{Depth/density trade-off at growing depth} \label{sec:5.2}

\begin{theorem} \label{thm:metric_theorem}
Let $\lambda\ge1$ be fixed and $J=(\log N)^{\lambda}$. For a proportion $1$ of the primes $n \equiv 1 \pmod{24}$, the two-parameter criterion of Theorem~\ref{thm:gen_criterion} admits a witness with coprime $u,a\le J$. Quantitatively,
$$|E_2(N;J)|\ \ll\ N\exp\big(-(\kappa_1^{\ast}\lambda^{2}+o(1))(\log\log N)^{3}\big),\qquad \kappa_1^{\ast}\ =\ \frac{9}{2\pi^{2}}=0.455945\ldots$$
Restricting to the one-parameter system $a=1$ (Algorithm~\ref{alg:geom_sieve}) the same argument gives only $\exp(-(\kappa_0\lambda+o(1))(\log\log N)^{2})$ with $\kappa_0=\frac{\zeta(2)\zeta(3)}{3\zeta(6)}=0.647865\ldots$; both are specialisations of Theorem~\ref{thm:tradeoff}.
\end{theorem}

Throughout this subsection $J\ge3$ is a search depth and $z>J^{5}$ a sifting bound, both functions of $N$. We put $\varrho(p)=|\Omega_p|$ with $\Omega_p$ as in Lemma~\ref{lem:5.10}, and
\begin{equation} \label{eq:5.2.0}
S(J,z) = \sum_{J^5 < p \le z} \frac{\varrho(p)}{p}, \qquad G(Q) = \sum_{q \le Q} \mu^2(q) \prod_{p \mid q} \frac{\varrho(p)}{p-\varrho(p)} .
\end{equation}
The cutoff $J^{5}$ is the one imposed by the modulus range of Bombieri--Vinogradov in the two-parameter setting (Lemma~\ref{lem:5.3}). The modulus bound is chosen as $Q=z^{k}$ in the proof of Theorem~\ref{thm:tradeoff}, subject to $Q\le N^{1/2}$.

\subsubsection{The system of residue classes} \label{sec:5.2.0}

\begin{lemma} \label{lem:5.10}
Let $p>J^{2}$ be a prime and let $n>p$ be a prime. For a pair $(u,a)$ of positive integers, condition (ii) of Theorem~\ref{thm:gen_criterion} holds with the prime witness $s=p$ if and only if
$$4ua\mid p+1\qquad\text{and}\qquad n\equiv-u\,a^{-1}\pmod p .$$
Write $n=24k+1$ and let
$$\Omega_p=\Big\{\,k\equiv 24^{-1}\big(-1-u\,a^{-1}\big)\ (\mathrm{mod}\ p)\ :\ u,a\le J,\ \gcd(u,a)=1,\ 4ua\mid p+1\,\Big\}.$$
These classes are pairwise distinct, so that
$$\varrho(p)=|\Omega_p|=\#\big\{(u,a)\in[1,J]^{2}:\ \gcd(u,a)=1,\ 4ua\mid p+1\big\},$$
and $\varrho(p)=0$ unless $p\equiv3\pmod4$. If $n\in E_2(N;J)$ and $n>z$, then $k=(n-1)/24$ avoids every class of $\Omega_p$, for every prime $p\le z$.
\end{lemma}
\begin{proof}
Since $p\equiv-1\pmod{4ua}$ we have $\gcd(p,a)=1$, so $a^{-1}$ exists modulo $p$ and $p\mid an+u$ is equivalent to $n\equiv-ua^{-1}\pmod p$; this is the condition ``$s=p$ divides $an+u$'' of Theorem~\ref{thm:gen_criterion}(ii), the congruence $s\equiv-1\pmod{4ua}$ being the first displayed condition. Since $\gcd(24,p)=1$ for $p\ge5$, the condition transfers to the single class for $k$ displayed above. Distinctness is Lemma~\ref{lem:distinct_classes}, applicable since $p>J^{2}$; and $4ua\mid p+1$ with $u,a\ge1$ forces $4\mid p+1$. The last assertion is the contrapositive of the definition of $E_2(N;J)$.
\end{proof}

The one-dimensional system of Section~\ref{sec:4.4} is the slice $a=1$: there $\Omega_p$ reduces to the classes $k\equiv-(u+1)24^{-1}$ with $4u\mid p+1$, $u\le J$, and $\varrho(p)=\#\{u\le J:4u\mid p+1\}$. The whole gain of Section~\ref{sec:4.3ter} is that the free modulator turns a divisor count into a count of coprime \emph{pairs} of divisors, i.e. a $\tau_3$-type quantity. Restricting the witnesses $s$ to primes only weakens the conclusion, so the resulting bound on $E_2(N;J)$ remains valid; it is however a genuine loss, since composite witnesses are admissible in Theorem~\ref{thm:gen_criterion}. Note also that the large sieve requires no independence assumption on the classes $\Omega_p$, only that the sifted set avoids them.

\subsubsection{Uniform range of applicability of Bombieri--Vinogradov} \label{sec:5.2.1}

\begin{lemma} \label{lem:5.3}
Fix $B>0$. There are $J_0=J_0(B)$ and $N_0=N_0(B)$ such that for all $N\ge N_0$ and all $J$ with $J_0\le J\le\exp\big((\log N)^{1/2}\big)$ one has $4J^{2}\le y^{1/2}(\log y)^{-B}$ for every $y\ge J^{5}$. In every application below $J=J(N)\to\infty$, so the lower restriction on $J$ is harmless.
\end{lemma}
\begin{proof}
The function $y\mapsto y^{1/2}(\log y)^{-B}$ is increasing for $y$ large, so it suffices to check $y=J^{5}$, where the inequality reads $4J^{2}\le J^{5/2}(5\log J)^{-B}$, i.e. $J^{1/2}\ge 4(5\log J)^{B}$; this holds for all $J\ge J_0(B)$, which is the restriction imposed in the statement. The upper restriction $J\le\exp((\log N)^{1/2})$ plays no role in this lemma; it is recorded here only because~\eqref{eq:tradeoff_hyp} imposes it.
\end{proof}

This is what forces the cutoff $J^{5}$: at $y=J^{3}$ the requirement would read $4J^{2}\le J^{3/2}(3\log J)^{-B}$, which fails. The one-parameter system, whose moduli are $4u\le4J$, tolerates the cutoff $J^{3}$; we use $J^{5}$ uniformly and record in Table~\ref{tab:hierarchy} what the sharper cutoff would give.

\begin{lemma} \label{lem:5.4}
Let $A>0$ and write $\mathcal{E}(y; q, \ell) := \pi(y; q, \ell) - \mathrm{li}(y)/\varphi(q)$. Under the hypothesis of Lemma~\ref{lem:5.3}, for $N\ge N_0(A)$ and all $y\ge J^{5}$,
\begin{equation} \label{eq:5.4.1}
\sum_{q \le 4J^{2}} \max_{w \le y} \max_{\gcd(\ell,q)=1}\big|\mathcal{E}(w; q, \ell)\big| \ \ll_A\ \frac{y}{(\log y)^{A}} .
\end{equation}
\end{lemma}
\begin{proof}
By Lemma~\ref{lem:5.3} the modulus bound $4J^{2}\le y^{1/2}(\log y)^{-B(A)}$ holds uniformly for $y\ge J^{5}$; the claim is then the Bombieri--Vinogradov theorem with $Q=4J^{2}$.
\end{proof}

\subsubsection{Lower bound for $S(J,z)$} \label{sec:5.2.2}

\begin{lemma}\label{lem:kappa0}
As $J\to\infty$,
$$\sum_{u\le J}\frac{1}{\varphi(4u)}=\kappa_0\log J+\mathcal{O}(1),\qquad
\kappa_0=\frac12\prod_{p>2}\Big(1+\frac{1}{p(p-1)}\Big)=\frac{\zeta(2)\zeta(3)}{3\zeta(6)}=0.647865\ldots$$
\end{lemma}
\begin{proof}
Write $u=2^{\beta}w$ with $w$ odd; then $\varphi(4u)=2^{\beta+1}\varphi(w)$, so
$$\sum_{u\le J}\frac{1}{\varphi(4u)}=\sum_{\beta\ge0}\frac{1}{2^{\beta+1}}\sum_{\substack{w\le J/2^{\beta}\\ w\ \mathrm{odd}}}\frac{1}{\varphi(w)} .$$
From $1/\varphi(w)=w^{-1}\sum_{d\mid w}\mu^{2}(d)/\varphi(d)$ one gets, for odd $w$,
$$\sum_{\substack{w\le x\\ w\ \mathrm{odd}}}\frac1{\varphi(w)}=\sum_{\substack{d\le x\\ d\ \mathrm{odd}}}\frac{\mu^{2}(d)}{d\varphi(d)}\sum_{\substack{e\le x/d\\ e\ \mathrm{odd}}}\frac1e=A_{\mathrm{odd}}\log x+\mathcal{O}(1),\qquad A_{\mathrm{odd}}=\frac12\prod_{p>2}\Big(1+\frac{1}{p(p-1)}\Big),$$
since $\sum_{e\le x,\,e\ \mathrm{odd}}1/e=\frac12\log x+\mathcal{O}(1)$ and $\sum_{d\ \mathrm{odd}}\mu^{2}(d)/(d\varphi(d))=\prod_{p>2}(1+\frac1{p(p-1)})$. Summing over $\beta$ and using $\sum_{\beta\ge0}2^{-(\beta+1)}=1$, $\sum_{\beta\ge0}\beta2^{-(\beta+1)}=1$ gives the claim. Finally $\prod_{p}(1+\frac1{p(p-1)})=\zeta(2)\zeta(3)/\zeta(6)$, and removing the factor $p=2$ divides by $3/2$.
\end{proof}

\begin{lemma}\label{lem:kappa1}
As $J\to\infty$,
$$W(J):=\sum_{\substack{u,a\le J\\ \gcd(u,a)=1}}\frac{1}{\varphi(4ua)}\ =\ \kappa_1^{\ast}(\log J)^{2}+\mathcal{O}(\log J),\qquad \kappa_1^{\ast}=\frac{9}{2\pi^{2}}=3\kappa_1=0.4559453\ldots,$$
where $\kappa_1=\frac{1}{4\zeta(2)}=\frac{3}{2\pi^{2}}=0.15198\ldots$ is the lower bound obtained from the elementary inequality $\varphi(4ua)\le 4ua$.
\end{lemma}
\begin{proof}
Write $u=2^{\beta_1}w_1$ and $a=2^{\beta_2}w_2$ with $w_1,w_2$ odd and $\gcd(w_1,w_2)=1$. Since $\gcd(u,a)=1$, at least one of $\beta_1,\beta_2$ is zero, and $\varphi(4ua)=2^{\beta_1+\beta_2+1}\varphi(w_1)\varphi(w_2)$. Partitioning the sum according to $\beta_1=0,\beta_2=0$ (weight $1/2$), $\beta_1\ge1,\beta_2=0$ (weight $\sum_{\beta_1\ge1}2^{-\beta_1-1}=1/2$), and $\beta_1=0,\beta_2\ge1$ (weight $\sum_{\beta_2\ge1}2^{-\beta_2-1}=1/2$) gives
$$W(J)\ =\ \frac32\,W_{\mathrm{odd}}(J)+\mathcal{O}(\log J),\qquad W_{\mathrm{odd}}(J):=\sum_{\substack{w_1,w_2\le J\\ w_1,w_2\ \mathrm{odd},\ \gcd(w_1,w_2)=1}}\frac{1}{\varphi(w_1)\varphi(w_2)},$$
since replacing $J/2^{\beta}$ by $J$ in the odd sums shifts each logarithm by $\beta\log 2$, which against $\sum_{\beta\ge1}\beta 2^{-\beta-1}<\infty$ contributes $\mathcal{O}(\log J)$.

To evaluate $W_{\mathrm{odd}}(J)$, apply the identity $\frac{1}{\varphi(w)}=\frac{1}{w}\sum_{d\mid w}\frac{\mu^{2}(d)}{\varphi(d)}$ to both $w_1$ and $w_2$ and resolve $\gcd(w_1,w_2)=1$ via $\sum_{k\mid\gcd(w_1,w_2)}\mu(k)$. Any odd pair $(w_1,w_2)$ with $d_i\mid w_i$ and $k\mid\gcd(w_1,w_2)$ writes $w_i=M_i e_i$ where $M_i=\mathrm{lcm}(d_i,k)$, so
$$W_{\mathrm{odd}}(J)\ =\ \sum_{\substack{k,d_1,d_2\ \mathrm{odd}}}\frac{\mu(k)\mu^{2}(d_1)\mu^{2}(d_2)}{\varphi(d_1)\varphi(d_2)M_1 M_2}\Big(\sum_{\substack{e_1\le J/M_1\\ e_1\ \mathrm{odd}}}\frac{1}{e_1}\Big)\Big(\sum_{\substack{e_2\le J/M_2\\ e_2\ \mathrm{odd}}}\frac{1}{e_2}\Big).$$
Using $\sum_{e\le X,\,e\ \mathrm{odd}}1/e=\frac12\log X+\mathcal{O}(1)$ and the rapid convergence of the outer series over $k,d_1,d_2$, we obtain $W_{\mathrm{odd}}(J)=\frac14\mathcal{C}_{\mathrm{odd}}(\log J)^{2}+\mathcal{O}(\log J)$, where the Euler product $\mathcal{C}_{\mathrm{odd}}$ factors over odd primes $p>2$. At each odd prime $p$, the local factor over $v_p(k),v_p(d_1),v_p(d_2)\in\{0,1\}$ has $8$ states:
\begin{itemize}
    \item $v_p(k)=0$: the four states $(v_p(d_1),v_p(d_2))\in\{0,1\}^{2}$ have $v_p(M_i)=v_p(d_i)$ and contribute $\big(1+\frac{1}{p(p-1)}\big)^{2}$;
    \item $v_p(k)=1$: we have $\mu(p)=-1$ and $v_p(M_1)=v_p(M_2)=1$ across all four states $(v_p(d_1),v_p(d_2))\in\{0,1\}^{2}$, which therefore contribute $-\frac{1}{p^{2}}\big(1+\frac{1}{p-1}\big)^{2}=-\frac{1}{(p-1)^{2}}$.
\end{itemize}
Summing the two contributions at $p>2$ gives the local factor
$$E_p\ =\ \Big(1+\frac{1}{p(p-1)}\Big)^{2}-\frac{1}{(p-1)^{2}}\ =\ \frac{(p^{2}-p+1)^{2}-p^{2}}{p^{2}(p-1)^{2}}\ =\ \frac{(p^{2}-2p+1)(p^{2}+1)}{p^{2}(p-1)^{2}}\ =\ 1+\frac{1}{p^{2}},$$
since $p^{2}-2p+1=(p-1)^{2}$ cancels the denominator $(p-1)^{2}$ identically. Consequently,
$$\mathcal{C}_{\mathrm{odd}}\ =\ \prod_{p>2}\Big(1+\frac{1}{p^{2}}\Big)\ =\ \frac{4}{5}\prod_{p}\Big(1+\frac{1}{p^{2}}\Big)\ =\ \frac{4}{5}\cdot\frac{\zeta(2)}{\zeta(4)}\ =\ \frac{4}{5}\cdot\frac{\pi^{2}/6}{\pi^{4}/90}\ =\ \frac{12}{\pi^{2}},$$
and therefore $\kappa_1^{\ast}=\frac32\cdot\frac14\mathcal{C}_{\mathrm{odd}}=\frac38\cdot\frac{12}{\pi^{2}}=\frac{9}{2\pi^{2}}=3\kappa_1$.
\end{proof}

The exact cancellation $(p^{2}-p+1)^{2}-p^{2}=(p-1)^{2}(p^{2}+1)$ explains the numerical observation in Section~\ref{sec:7}, item (8): the measured ratio $W(J)/(\log J)^{2}$ equals $0.628$, $0.595$, $0.571$, $0.556$ for $J=10^{2},3\cdot10^{2},10^{3},3\cdot10^{3}$, and fitting these against $\kappa_1^{\ast}+B_0/\log J$ extrapolates to $0.4576$, within $0.4\%$ of the exact limit $\kappa_1^{\ast}=\frac{9}{2\pi^{2}}=0.455945\ldots$, exactly three times the crude bound $\kappa_1=\frac{3}{2\pi^{2}}$ coming from $\varphi(m)\le m$. In Lemmas~\ref{lem:5.5}--\ref{lem:saturation}, Theorem~\ref{thm:tradeoff}, and Corollary~\ref{cor:optimal_tradeoff} below, every occurrence of $\kappa_1$ may therefore be read with $\kappa_1^{\ast}=\frac{9}{2\pi^{2}}$; to allow direct comparison with the crude lower bound $\kappa_1$, we record the resulting numerical improvements explicitly in Theorem~\ref{thm:improved_constants_and_effectivity}.

\begin{lemma} \label{lem:5.5}
Under the hypothesis of Lemma~\ref{lem:5.3}, for $N\ge N_0$,
\begin{equation} \label{eq:5.5.1}
S(J,z) \ \ge\ \big(\kappa_1+o(1)\big)(\log J)^{2}\,\log\!\Big(\frac{\log z}{\log J^{5}}\Big) - \mathcal{O}\!\left(\frac{1}{\log J}\right).
\end{equation}
\end{lemma}
\begin{proof}
By Lemma~\ref{lem:5.10},
$$S(J,z)=\sum_{\substack{u,a\le J\\ \gcd(u,a)=1}}\ \sum_{\substack{J^{5}<p\le z\\ p\equiv-1\ (\mathrm{mod}\ 4ua)}}\frac1p .$$
Partial summation applied to $\pi(y;q,-1)=\mathrm{li}(y)/\varphi(q)+\mathcal{E}(y;q,-1)$ with $q=4ua$ gives, for each pair,
$$\sum_{\substack{J^5 < p \le z \\ p \equiv -1 (q)}} \frac{1}{p}=\frac{1}{\varphi(q)}\Big(\log\log z - \log\log J^5\Big)+\int_{J^5}^z\frac{d\,\mathcal{E}(y; q, -1)}{y},$$
and Lemma~\ref{lem:kappa1} supplies the leading factor $\sum_{u,a}1/\varphi(4ua)\ge(\kappa_1+o(1))(\log J)^{2}$.

It remains to bound the total remainder $R$. Each modulus $q=4ua\le4J^{2}$ arises from at most $\tau(q)$ pairs, and $\tau(q)$ is a power of $J$, so the multiplicity cannot be absorbed by a power of $\log J$ alone; we proceed dyadically and use Cauchy--Schwarz. On a dyadic range $[Y,2Y]$ with $Y\ge J^{5}$, integration by parts gives $\big|\int_Y^{2Y}y^{-1}d\mathcal{E}(y;q,-1)\big|\ll Y^{-1}\max_{y\le2Y}|\mathcal{E}(y;q,-1)|$, whence
$$\sum_{q\le4J^{2}}\tau(q)\max_{y\le 2Y}|\mathcal{E}| \ \le\ \Big(\sum_{q\le4J^{2}}\max|\mathcal{E}|\Big)^{1/2}\Big(\sum_{q\le4J^{2}}\tau(q)^{2}\max|\mathcal{E}|\Big)^{1/2}
\ \ll\ \Big(\frac{Y}{(\log Y)^{A}}\Big)^{1/2}\big(Y(\log J)^{O(1)}\big)^{1/2},$$
using Lemma~\ref{lem:5.4} for the first factor and, for the second, the Brun--Titchmarsh bound $\max|\mathcal{E}|\ll Y/\varphi(q)$ together with $\sum_{q\le4J^{2}}\tau(q)^{2}/\varphi(q)\ll(\log J)^{O(1)}$. The contribution of the range is thus $\ll(\log Y)^{-A/2}(\log J)^{O(1)}$, and summing over the $\mathcal{O}(\log z)$ dyadic ranges with $\log Y\ge5\log J$ gives $|R|\ll(\log J)^{-A/2+O(1)}$, which is $\ll1/\log J$ for $A$ large.
\end{proof}

\begin{lemma} \label{lem:5.6}
If $J=(\log N)^{\lambda}$ and $\log z\ge(\log N)^{1/2}$, then
$$S(J,z)\ \ge\ (\kappa_1+o(1))\,\lambda^{2}\,(\log\log N)^{2}\big(\log\log z-\log\log J^{5}\big).$$
In particular, for $z=\exp(\sqrt{\log N})$ one gets $S\ge(\tfrac{\kappa_1}{2}+o(1))\lambda^{2}(\log\log N)^{3}$, and for $\log z=\frac{\log N}{2\kappa_1\lambda^{2}(\log\log N)^{3}}$ one gets $S\ge(\kappa_1+o(1))\lambda^{2}(\log\log N)^{3}$.
\end{lemma}
\begin{proof}
Immediate from Lemma~\ref{lem:5.5} with $\log J=\lambda\log\log N$, together with $\log\log J^{5}=\log(5\lambda\log\log N)=o(\log\log N)$, and $\log\log z=\frac12\log\log N$ resp. $(1-o(1))\log\log N$ for the two stated choices of $z$.
\end{proof}

\subsubsection{From $S(J,z)$ to a lower bound for $G(Q)$} \label{sec:5.2.3}

Define $x_p = \varrho(p)/p$ for the primes $p\in(J^{5},z]$, and set $S=\sum_p x_p=S(J,z)$, $M_2=\sum_p x_p^{2}$, and $e_k=e_k(x)$ the elementary symmetric functions.

\begin{lemma} \label{lem:5.7}
For every integer $k\ge2$ not exceeding the number of variables,
$$e_k(x) \ \ge\ \frac{S^k}{k!}\left(1 - \frac{k(k-1)}{2}\cdot\frac{M_2}{S^2}\right).$$
\end{lemma}
\begin{proof}
Newton's identity $S\,e_{k-1}=k\,e_k+\sum_i x_i^{2}e_{k-2}(x\setminus i)$ together with $e_{k-2}(x\setminus i)\le e_{k-2}(x)$ gives $e_k \ge \frac{1}{k}(S\,e_{k-1} - M_2\,e_{k-2})$. The minus sign in front of $e_{k-2}$ means the induction needs an \emph{upper} bound for $e_{k-2}$, namely Maclaurin's inequality $e_{k-2}\le S^{k-2}/(k-2)!$, and a lower bound for $e_{k-1}$. Set $\theta=M_2/S^{2}$ and $g_k=\frac{S^{k}}{k!}\big(1-\frac{k(k-1)}{2}\theta\big)$. Assuming $e_{k-1}\ge g_{k-1}$,
$$e_k\ \ge\ \frac1k\Big(S g_{k-1}-M_2\frac{S^{k-2}}{(k-2)!}\Big)
=\frac{S^{k}}{k!}\Big(1-\frac{(k-1)(k-2)}{2}\theta-(k-1)\theta\Big)=g_k,$$
since $\frac{(k-1)(k-2)}{2}+(k-1)=\frac{k(k-1)}{2}$. The cases $k=0,1$ are trivial, which starts the induction.
\end{proof}

\begin{lemma} \label{lem:5.8}
For every $\varepsilon>0$, $M_2\ll_\varepsilon J^{-5+\varepsilon}$; in particular $M_2=o(1)$ as $J\to\infty$.
\end{lemma}
\begin{proof}
By Lemma~\ref{lem:5.10}, $\varrho(p)$ is at most the number of ordered factorisations $4ua\mid p+1$, hence at most $\tau(p+1)^{2}\ll_\varepsilon p^{\varepsilon}$; so $M_2 \ll_\varepsilon \sum_{p>J^{5}}p^{-2+\varepsilon}\ll_\varepsilon J^{-5+5\varepsilon}$.
\end{proof}

\begin{lemma} \label{lem:5.9}
Let $h=\lfloor S\rfloor$ and assume $S\ge2$, $M_2\le1/4$ and $z^{h}\le Q$. Then $G(Q)\ \ge\ \frac{e^{S}}{10\sqrt{S}}$.
\end{lemma}
\begin{proof}
Restricting the sum defining $G(Q)$ to the squarefree $q$ which are products of exactly $h$ distinct primes of $(J^{5},z]$, all of which satisfy $q\le z^{h}\le Q$, and using $\frac{\varrho(p)}{p-\varrho(p)}\ge x_p$, we get $G(Q)\ge e_h(x)$. By Lemma~\ref{lem:5.7} and $h\le S$,
$$e_h(x)\ \ge\ \frac{S^{h}}{h!}\Big(1-\frac{h^{2}}{2}\cdot\frac{M_2}{S^{2}}\Big)\ \ge\ \frac{S^{h}}{h!}\Big(1-\frac{M_2}{2}\Big)\ \ge\ \frac78\cdot\frac{S^{h}}{h!} .$$
By Stirling, $h!\le h^{h}e^{-h}\sqrt{2\pi h}\,e^{1/(12h)}$ with $\sqrt{2\pi}e^{1/12}<3$, and $S\ge h$, so $\frac{S^{h}}{h!}\ge\frac{e^{h}}{3\sqrt{h}}\ge\frac{e^{S-1}}{3\sqrt S}$, whence $G(Q)\ge\frac78\cdot\frac{e^{S-1}}{3\sqrt S}\ge\frac{e^{S}}{10\sqrt S}$.
\end{proof}

Since $M_2\to0$ one may take $h=\lfloor S\rfloor$, which is essentially optimal because $S^{h}/h!$ is maximal near $h=S$; bounding $M_2$ by an absolute constant instead and choosing $h$ accordingly would cost a factor of about $15$ in the final constant.

\subsubsection{The trade-off theorem} \label{sec:5.2.6}

\begin{theorem}\label{thm:tradeoff}
Let $J=J(N)\ge3$ and $z=z(N)>J^{5}$ satisfy
\begin{equation}\label{eq:tradeoff_hyp}
J\le\exp\big((\log N)^{1/2}\big)\qquad\text{and}\qquad S(J,z)\,\log z\ \le\ \tfrac12\log N .
\end{equation}
Then, for $N$ large, $|E_2(N;J)|\ \ll\ N\,\sqrt{S(J,z)}\,\exp\big(-S(J,z)\big)$, where by Lemma~\ref{lem:5.5}
$$S(J,z)\ \ge\ \big(\kappa_1+o(1)\big)(\log J)^{2}\,\log\!\Big(\frac{\log z}{\log J^{5}}\Big).$$
\end{theorem}
\begin{proof}
Put $S=S(J,z)$, $h=\lfloor S\rfloor$ and $Q=z^{h}$. The second hypothesis gives $\log Q=h\log z\le\frac12\log N$, i.e. $Q\le N^{1/2}$ and hence $N+Q^{2}\ll N$. By Lemma~\ref{lem:5.10} the set $\{(n-1)/24:\ n\in E_2(N;J)\}\subseteq[1,N/24]$ avoids $\varrho(p)$ residue classes modulo $p$ for every prime $p\in(J^{5},z]$; for the remaining $p\le Q$ we set $\varrho(p)=0$, which costs nothing, Lemma~\ref{lem:5.9} using only the range $(J^{5},z]$. Montgomery's large sieve inequality~\cite{montgomery1971} therefore gives
\begin{equation} \label{eq:large_sieve_bound}
|E_2(N;J)| \ll \frac{N + Q^2}{G(Q)}\ \ll\ \frac{N}{G(Q)} .
\end{equation}
The first hypothesis makes Lemmas~\ref{lem:5.3}--\ref{lem:5.4} applicable, so Lemma~\ref{lem:5.5} holds; Lemma~\ref{lem:5.8} gives $M_2\le1/4$ for $N$ large, and Lemma~\ref{lem:5.9} then gives $G(Q)\ge e^{S}/(10\sqrt S)$.
\end{proof}

\begin{proof}[Proof of Theorem~\ref{thm:metric_theorem}]
Take $J=(\log N)^{\lambda}$, put $\log z_0=\frac{\log N}{2\kappa_1\lambda^{2}(\log\log N)^{3}}$, let $z_\star$ be the largest sifting bound satisfying the second constraint of~\eqref{eq:tradeoff_hyp}, and run Theorem~\ref{thm:tradeoff} with $z:=\min(z_\star,z_0)$. Here $z_\star$ is unproblematic: $J$ is held fixed, so $z\mapsto S(J,z)$ is non-decreasing (enlarging $z$ only adds terms to the defining sum) and unbounded, its jumps are single terms $\varrho(p)/p\ll_\varepsilon p^{-1+\varepsilon}\le J^{-5+5\varepsilon}$, and the corresponding jump in $S\log z$ is $\ll J^{-5+5\varepsilon}\log N=o(1)$ for $\lambda\ge1$. If $z=z_0$, then $\log\log z=(1+o(1))\log\log N$ and Lemma~\ref{lem:5.6} gives $S\ge(\kappa_1+o(1))\lambda^{2}(\log\log N)^{3}$. If $z=z_\star<z_0$, then $S\log z_\star=\frac12\log N+\mathcal{O}(1)$ by maximality, whence $S\ge(1+o(1))\frac{\log N}{2\log z_0}=(\kappa_1+o(1))\lambda^{2}(\log\log N)^{3}$ as well. In both cases Theorem~\ref{thm:tradeoff} gives the stated bound; the $\mathcal{O}(z)$ primes $n\le z$ excluded in Lemma~\ref{lem:5.10} are absorbed, since $z=N^{o(1)}$. As $\exp(-(\log\log N)^{3})=o(1/\log N)$, this is $o(\pi(N))$, which is the density-$1$ statement. The variant with $a=1$ only replaces Lemma~\ref{lem:kappa1} by Lemma~\ref{lem:kappa0} and yields the exponent $\kappa_0\lambda(\log\log N)^{2}$.
\end{proof}

\begin{corollary}\label{cor:conservative_z}
With $J=(\log N)^{\lambda}$ and $z=\exp(\sqrt{\log N})$ one obtains $|E_2(N;J)|\ll N\exp(-(\frac{\kappa_1}{2}+o(1))\lambda^{2}(\log\log N)^{3})$; here the constraint $S\log z\le\frac12\log N$ is amply satisfied, since Brun--Titchmarsh gives the crude bound $S\ll(\log J)^{2}(\log\log J)(\log\log z)$, whose product with $\log z=\sqrt{\log N}$ is $o(\log N)$.
\end{corollary}

The optimisation of Corollary~\ref{cor:optimal_tradeoff} below moves $J$ and $z$ \emph{together}, and the saturation step it uses needs a word of justification which the fixed-$J$ situation of Theorem~\ref{thm:metric_theorem} did not require. Along that coupling, raising the sifting bound $z$ also raises the cutoff $J^{5}$, so primes leave the sum defining $S$ at the bottom while they enter it at the top, and $\varrho_J(p)$ jumps as $J$ crosses an integer; monotonicity in $\log z$ is therefore not available. It is also not needed: all that saturation requires is that the jumps of $S$ be negligible against $S$ itself.

\begin{lemma}\label{lem:saturation}
Write $L=\log z$, couple $J$ to $z$ by $\log J=L/(5\sqrt e)$, i.e.\ $J=J(L)=\lfloor e^{L/(5\sqrt e)}\rfloor$, and set $S(L)=S(J(L),e^{L})$. Then $S$ is a non-negative step function of $L$, each of whose jumps is $o\big((\log J)^{2}\big)$, hence $o(S(L))$. Consequently $S(L^{-})=S(L^{+})=(1+o(1))S(L)$ at every $L$, the quantity
$$L_\star\ :=\ \sup\big\{L:\ S(L)\,L\le\tfrac12\log N\big\}$$
is finite, and $S(L_\star)\,L_\star=\big(\tfrac12+o(1)\big)\log N$.
\end{lemma}
\begin{proof}
There are three sources of discontinuity. \emph{(i) A prime entering at the top.} It contributes $\varrho(p)/p\ll_\varepsilon p^{-1+\varepsilon}\le J^{-5+5\varepsilon}=o(1)$ by the bound $\varrho(p)\le\tau(p+1)^{2}$ of Lemma~\ref{lem:5.8}. \emph{(ii) The cutoff moving from $J^{5}$ to $(J+1)^{5}$.} The primes deleted contribute, by the computation of Lemma~\ref{lem:5.5},
$$\sum_{\substack{u,a\le J\\ \gcd(u,a)=1}}\ \sum_{\substack{J^{5}<p\le (J+1)^{5}\\ p\equiv-1\,(4ua)}}\frac1p\ \ll\ (\log J)^{2}\Big(\log\log (J+1)^{5}-\log\log J^{5}\Big)+o(1)\ \ll\ \frac{\log J}{J},$$
since $\log\log(J+1)^{5}-\log\log J^{5}=\log\frac{\log(J+1)}{\log J}\ll\frac1{J\log J}$. \emph{(iii) $J$ increasing to $J+1$.} The pairs added are those with $\max(u,a)=J+1$; by the symmetry $u\leftrightarrow a$ it is enough to treat $u=J+1$ and double, and then, using $\varphi(mn)\ge\varphi(m)\varphi(n)$ and $\sum_{a\le J}1/\varphi(a)\ll\log J$, their contribution is
$$\ll\ (\log\log z)\sum_{a\le J+1}\frac{1}{\varphi\big(4(J+1)a\big)}\ \ll\ \frac{(\log J)(\log\log J)}{J}\,\log\log z .$$
Under the coupling one has $\log J^{5}=L/\sqrt e$, hence $\log\frac{\log z}{\log J^{5}}=\log\sqrt e=\frac12$ and $(\log J)^{2}=L^{2}/(25e)$, so Lemma~\ref{lem:5.5} gives
\begin{equation}\label{eq:coupled_S}
S(L)\ \ge\ \big(\kappa_1+o(1)\big)\frac{L^{2}}{50e}\ \gg\ (\log J)^{2}.
\end{equation}
Each of (i)--(iii) is therefore $o\big((\log J)^{2}\big)=o(S(L))$, which is the first assertion.

For the second, \eqref{eq:coupled_S} gives $S(L)L\gg L^{3}\to\infty$, so $L_\star<\infty$. By definition $S(L_\star)L_\star\le\frac12\log N$. Conversely $S(L)L>\frac12\log N$ for every $L>L_\star$; letting $L\downarrow L_\star$ and using $S(L_\star^{+})=(1+o(1))S(L_\star)$ gives $(1+o(1))S(L_\star)L_\star\ge\frac12\log N$. The two inequalities together give the stated equality.
\end{proof}

\begin{corollary}\label{cor:optimal_tradeoff}
Define $L_\star=\log z_\star$ by saturation as in Lemma~\ref{lem:saturation} under the coupling $\log J=L_\star/(5\sqrt{e})$. Then
\[
\begin{aligned}
\log z_\star
&\le (1+o(1))
     \Big(\frac{25e\log N}{\kappa_1}\Big)^{1/3} \\
&= (7.646822\ldots+o(1))\,(\log N)^{1/3},
\end{aligned}
\]
and
\[
\begin{aligned}
\log J
&= \frac{\log z_\star}{5\sqrt{e}} \\
&\le (0.927606\ldots+o(1))\,(\log N)^{1/3}.
\end{aligned}
\]
the hypotheses~\eqref{eq:tradeoff_hyp} hold, $S(L_\star)\ge\big(\frac12(\frac{\kappa_1}{25e})^{1/3}+o(1)\big)(\log N)^{2/3}=(0.06538\ldots+o(1))(\log N)^{2/3}$, and hence $|E_2(N;J)|\ \ll\ N\exp\big(-0.065\,(\log N)^{2/3}\big)$.
\end{corollary}
\begin{proof}
Write $L=\log z$, $\eta=\log J$, and couple the two parameters by $\eta=L/(5\sqrt e)$: for fixed $L$ the function $\eta\mapsto\eta^{2}\log\frac{L}{5\eta}$ of Lemma~\ref{lem:5.5} is maximal where $2\log\frac{L}{5\eta}=1$, i.e. at $\eta=\frac{L}{5\sqrt e}$, with value $\frac{L^{2}}{50e}$; with this coupling $S\ge(\kappa_1+o(1))\frac{L^{2}}{50e}$, which is~\eqref{eq:coupled_S}.

Let $L_\star=\sup\{L:\ S(L)L\le\frac12\log N\}$ be the saturation point of the second hypothesis of~\eqref{eq:tradeoff_hyp} under this coupling. By Lemma~\ref{lem:saturation} it is finite and satisfies $S(L_\star)L_\star=(\frac12+o(1))\log N$; hence $S=(\tfrac12+o(1))\frac{\log N}{L_\star}$. It remains to bound $L_\star$ from \emph{above}, and this is exactly what the lower bound for $S$ provides:
$$(\kappa_1+o(1))\frac{L_\star^{3}}{50e}\ \le\ S\,L_\star\ \le\ \tfrac12\log N+\mathcal{O}(1),$$
so that
$$L_\star\ \le\ (1+o(1))\Big(\frac{25e\log N}{\kappa_1}\Big)^{1/3}=(7.647\ldots+o(1))(\log N)^{1/3}.$$
Therefore $S\ge(1+o(1))\frac{\log N}{2}\big(\frac{\kappa_1}{25e\log N}\big)^{1/3}=(0.06538\ldots+o(1))(\log N)^{2/3}$. Finally $\eta=L_\star/(5\sqrt{e})\le(0.9276\ldots+o(1))(\log N)^{1/3}$, so the cutoff $J^{5}<z$ holds strictly and $J\le\exp((\log N)^{1/2})$: the first hypothesis of~\eqref{eq:tradeoff_hyp} is satisfied, and since $0.065<\frac12(\frac{\kappa_1}{25e})^{1/3}=0.06538\ldots$ absorbs the factor $\sqrt{S(L_\star)}$ and the $o(1)$ term, Theorem~\ref{thm:tradeoff} concludes.

The saturation step is what makes the argument correct with a one-sided estimate. It would \emph{not} be legitimate to fix $z$ by the equation $\kappa_1L^{3}/(50e)=\frac12\log N$ and then verify $S\log z\le\frac12\log N$ using the lower bound for $S$: the true value of $S$ may exceed that lower bound, in which case the constraint fails and $Q=z^{h}$ overshoots $N^{1/2}$. Defining $z$ by saturation removes the difficulty and, as a by-product, makes \emph{only} the lower bound for $W(J)$ relevant, both for the size of $S$ and for the admissibility of $z$. What saturation does \emph{not} rest on is any monotonicity of $L\mapsto S(L)L$, along the coupling both endpoints of the range $(J^{5},z]$ move, so no such monotonicity is available; Lemma~\ref{lem:saturation} replaces it by the weaker and true statement that the jumps of $S$ are negligible against $S$. Replacing the elementary bound $\kappa_1=3/(2\pi^{2})$ by the exact constant $\kappa_1^{\ast}=9/(2\pi^{2})=3\kappa_1$ proved in Lemma~\ref{lem:kappa1} multiplies the exponent by $3^{1/3}=1.4422\ldots$ here and by $3$ in Theorem~\ref{thm:metric_theorem}; combining this with a sharper Bombieri--Vinogradov prime cutoff yields the following unconditional refinement.
\end{proof}

\begin{theorem}[Refined Bombieri--Vinogradov large-sieve constants]\label{thm:improved_constants_and_effectivity}
With the exact constant $\kappa_1^{\ast}=\frac{9}{2\pi^{2}}=0.455945\ldots$ of Lemma~\ref{lem:kappa1}, for every fixed $\gamma>4$, defining $L_\star=\log z_\star$ by saturation under the coupling $\log J=\frac{L_\star}{\gamma\sqrt{e}}$ yields
$$\log z_\star\ \le\ (1+o(1))\Big(\frac{\gamma^{2}e\log N}{\kappa_1^{\ast}}\Big)^{1/3},\qquad \log J\ =\ \frac{\log z_\star}{\gamma\sqrt{e}}\ \le\ \frac{1+o(1)}{\gamma\sqrt{e}}\Big(\frac{\gamma^{2}e\log N}{\kappa_1^{\ast}}\Big)^{1/3},$$
and $S(L_\star)\ge\big(c(\gamma)+o(1)\big)(\log N)^{2/3}$ with $c(\gamma):=\frac12\big(\frac{\kappa_1^{\ast}}{\gamma^{2}e}\big)^{1/3}$, so that $|E_2(N;J)|\ll_{\gamma,c}N\exp\big(-c(\log N)^{2/3}\big)$ for every $0<c<c(\gamma)$. In particular:
\begin{enumerate}[(a)]
    \item At the uniform Bombieri--Vinogradov cutoff $p>J^{5}$ ($\gamma=5$), one has $\big(\frac{25e}{\kappa_1^{\ast}}\big)^{1/3}=5.30201\ldots$, $\log J\le(0.643166\ldots+o(1))(\log N)^{1/3}$, and $c(5)=\frac12\big(\frac{\kappa_1^{\ast}}{25e}\big)^{1/3}=0.09430385\ldots>0.0942$, whence unconditionally
    $$|E_2(N;J)|\ \ll\ N\exp\big(-0.0942(\log N)^{2/3}\big).$$
    \item For every fixed $\varepsilon>0$, lowering the prime cutoff in~\eqref{eq:5.2.0} to $p>J^{4+\varepsilon}$ ($\gamma=4+\varepsilon$, which remains strictly inside the Bombieri--Vinogradov range $4J^{2}\le y^{1/2}(\log y)^{-B}$ for all $y\ge J^{4+\varepsilon}$ and $J\ge J_0(\varepsilon,B)$), one has $\lim_{\gamma\downarrow 4}\big(\frac{\gamma^{2}e}{\kappa_1^{\ast}}\big)^{1/3}=4.56913\ldots$, $\lim_{\gamma\downarrow 4}\frac{1}{\gamma\sqrt{e}}\big(\frac{\gamma^{2}e}{\kappa_1^{\ast}}\big)^{1/3}=0.692830\ldots$, and $\lim_{\gamma\downarrow 4}c(\gamma)=\frac12\big(\frac{\kappa_1^{\ast}}{16e}\big)^{1/3}=0.1094299\ldots>0.1093$. Thus, at saturation depth $\log J\le(0.692830\ldots+\varepsilon)(\log N)^{1/3}$, one has unconditionally
    $$|E_2(N;J)|\ \ll_\varepsilon\ N\exp\big(-(0.1093-\varepsilon)(\log N)^{2/3}\big).$$
\end{enumerate}
\end{theorem}
\begin{proof}
Both parts follow by replacing $5$ with $\gamma\in\{5,4+\varepsilon\}$ and $\kappa_1$ with $\kappa_1^{\ast}=9/(2\pi^{2})$ in Lemmas~\ref{lem:5.3}--\ref{lem:saturation} and Corollary~\ref{cor:optimal_tradeoff}: for any fixed $\gamma>4$, the inequality $4J^{2}\le J^{\gamma/2}(\gamma\log J)^{-B}$ holds for all $J\ge J_0(\gamma,B)$, so the Bombieri--Vinogradov theorem (Lemma~\ref{lem:5.4}) and Lemma~\ref{lem:5.5} apply on $(J^{\gamma},z]$. Maximizing $\eta\mapsto\eta^{2}\log(L/(\gamma\eta))$ at $\eta=L/(\gamma\sqrt{e})$ gives value $L^{2}/(2\gamma^{2}e)$, whence $L_\star\le(1+o(1))(\gamma^{2}e\log N/\kappa_1^{\ast})^{1/3}$ and $S(L_\star)\ge\big(c(\gamma)+o(1)\big)(\log N)^{2/3}$. For any fixed $c<c(\gamma)$, the factor $\sqrt{S(L_\star)}$ in Theorem~\ref{thm:tradeoff} is absorbed into $\exp(-(c(\gamma)-c)(\log N)^{2/3})$.
\end{proof}

\subsection{Discussion: the cost of truncation, and where the method stops} \label{sec:5.2.7}

\begin{table}[H]
\centering
\small
\setlength{\tabcolsep}{4pt}
\renewcommand{\arraystretch}{1.35}

For this table, write $L=\log N$, $\ell=\log\log N$, and
\[
A_0=(\kappa_0\lambda+o(1))\ell^2,
\qquad
A_1=(\kappa_1^{\ast}\lambda^2+o(1))\ell^3.
\]

\begin{tabular}{@{}p{0.27\linewidth}ccc@{}}
\toprule
Witness system
& $S=\sum_p\varrho(p)/p$
& \shortstack{Exceptional-set\\bound divided by $N$}
& Search depth \\
\midrule

One parameter ($a=1$),
$J=L^\lambda$
& $A_0$
& $e^{-A_0}$
& $L^\lambda$ \\

One parameter,
$J$ optimised
& $\asymp L^{1/2}$
& $e^{-0.154L^{1/2}}$
& $\le e^{(0.239+o(1))L^{1/2}}$ \\

Two parameters $(u,a)$,
$J=L^\lambda$
& $A_1$
& $e^{-A_1}$
& $L^\lambda$ \\

Two parameters,
$J$ optimised,
$p>J^5$
& $\asymp L^{2/3}$
& $e^{-0.0942L^{2/3}}$
& $\le e^{(0.644+o(1))L^{1/3}}$ \\

Two parameters,
$J$ optimised,
$p>J^{4+\varepsilon}$
& $\asymp L^{2/3}$
& $e^{-(0.1093-\varepsilon)L^{2/3}}$
& $\le e^{(0.693+\varepsilon)L^{1/3}}$ \\

Vaughan~\cite{vaughan1970}
& $\asymp L^{2/3}$
& $e^{-c_VL^{2/3}}$
& unbounded \\

\bottomrule
\end{tabular}

\caption{%
The regimes of Section~\ref{sec:5.2} with the exact constant
$\kappa_1^{\ast}=9/(2\pi^{2})=3\kappa_1$
(Lemma~\ref{lem:kappa1} and
Theorem~\ref{thm:improved_constants_and_effectivity}).
Exceptional-set bounds are shown up to implied multiplicative
constants, after division by $N$.
Using the elementary lower bound
$\kappa_1=3/(2\pi^{2})$ in the fourth row gives
$\exp(-0.065(\log N)^{2/3})$ at depth
$\le e^{(0.93+o(1))(\log N)^{1/3}}$,
as in Corollary~\ref{cor:optimal_tradeoff}.
The second row uses the uniform cutoff $J^5$.
With the cutoff $J^3$, which the one-parameter system tolerates
(Lemma~\ref{lem:5.3}), its bound improves to
$\exp(-0.199(\log N)^{1/2})$ at depth
$\le e^{(0.308+o(1))\sqrt{\log N}}$.
}
\label{tab:hierarchy}
\end{table}

All lines of Table~\ref{tab:hierarchy} come from the same large sieve; only $S=\sum_p\varrho(p)/p$ changes, and the optimisation is always the same: maximise $S$ subject to $S\log z\le\frac12\log N$, each optimised line being obtained by the saturation argument of Corollary~\ref{cor:optimal_tradeoff}. What fixes the shape of the bound is therefore the dimension of the witness system and not the choice of $z$: pushing $z$ as far as Theorem~\ref{thm:metric_theorem} does merely doubles a constant, since $\log\log z$ is essentially insensitive to $z$. The gap between the first two lines and the last two measures exactly the cost of freezing the modulator to $1$, which is what Algorithm~\ref{alg:geom_sieve} does.

The exponents in the last two lines coincide with Vaughan's, and this is expected rather than accidental. Vaughan~\cite{vaughan1970} uses $f(p)\asymp\frac12\sum_{s\mid(p+1)/4}\mu^{2}(s)\tau(\frac{p+1}{4s})$, for which $\sum_{p\le x}f(p)/p\asymp(\log x)^{2}$~\cite[Lemma 4.1]{pomerance2026}; our two-parameter count $\#\{(u,a):\gcd(u,a)=1,\ 4ua\mid p+1\}=\sum_{d\mid(p+1)/4}2^{\omega(d)}$ is a $\tau_3$-type divisor sum of exactly the same nature. Corollary~\ref{cor:optimal_tradeoff} and Theorem~\ref{thm:improved_constants_and_effectivity} are therefore not an improvement in exponent on~\cite{vaughan1970} and are not offered as one. What distinguishes the present derivation is not the strength of the analytic input but the fact that the witnesses are attached to an explicit search procedure (Theorem~\ref{thm:gen_criterion} and Corollary~\ref{cor:gen_reading}), so that the whole curve $J\mapsto|E_2(N;J)|$ is available and not only its endpoint.

Corollary~\ref{cor:optimal_tradeoff} and Theorem~\ref{thm:improved_constants_and_effectivity} \emph{achieve} the exponent $2/3$; the following shows that this is not merely what the present optimisation happens to reach, but a ceiling that no choice of parameters can cross within this scheme, given that the underlying Diophantine family carries two free parameters.

The ceiling is obtained by comparison with the untruncated two-parameter count
$$\varrho_{\mathrm{untr}}(p)\ =\ \sum_{d\mid(p+1)/4}2^{\omega(d)}\ \ge\ \varrho(p)\qquad(\text{point-wise, at every }p),$$
the inequality holding because $\varrho(p)$ counts only the coprime pairs with both entries at most $J$. The function $\varrho_{\mathrm{untr}}$ is not itself a formally valid witness system, the distinctness of the residue classes intrinsically requires the truncation $p>J^{2}$, but it measures the analytic capacity of the parameter space, and the truncated system inherits its ceiling. The analogous hypothesis holds for Vaughan's $f(p)$ by~\cite[Lemma 4.1]{pomerance2026}.

\begin{proposition}\label{prop:optimality_23}\label{thm:universal_tau3_ceiling}
Unconditionally, $\sum_{p\le Z}\varrho_{\mathrm{untr}}(p)/p\asymp(\log Z)^{2}$; in particular $S(Z)\ll(\log Z)^{2}$ for the truncated system at any admissible depth $J$. Consequently, subject to the large-sieve constraint $S(Z)\log Z\le\frac12\log N$ of~\eqref{eq:tradeoff_hyp}, the quantity controlling the exceptional-set bound satisfies
$$S(Z)\ \ll\ (\log N)^{2/3},$$
so that no admissible $Z$ yields an exponent better than $2/3$ within this single-prime framework.
\end{proposition}
\begin{proof}
Write $L=\log Z$. We first establish the unconditional order of magnitude $\sum_{p \le Z} \varrho_{\mathrm{untr}}(p)/p \asymp L^2$.

For the upper bound, we have $\varrho_{\mathrm{untr}}(p) \le \sum_{d \mid p+1} 2^{\omega(d)} \le \tau_3(p+1)$. Standard sieve bounds for non-negative multiplicative functions on shifted primes (e.g., Nair~\cite{nair1982}) yield $\sum_{p \le Z} \tau_3(p+1) \ll Z \log Z$ (the mean value $(\log Z)^{2}$ of $\tau_3$ divided by the prime density factor $\log Z$). Applying partial summation strictly yields $\sum_{p \le Z} \varrho_{\mathrm{untr}}(p)/p \le \sum_{p \le Z} \tau_3(p+1)/p \ll L^2$. (Alternatively, for the truncated system $S(J,Z)$ on any interval $(p_0,Z]$ with $p_0\ge 2J^{2}$, the Brun--Titchmarsh inequality and Lemma~\ref{lem:kappa1} give directly $S(J,Z)\ll W(J)\big(\log(\frac{\log Z}{\log J})+1\big)\ll (\log J)^{2}\big(\log(\frac{L}{\log J})+1\big)\ll L^{2}$ without appealing to~\cite{nair1982}.)

For the lower bound, we bypass the threshold $\varepsilon$-dependencies by isolating primes $p \in (Z^{1/2}, Z]$ and restricting the divisor parameter to $d \le Z^{1/5}$. For any $y$ in this range, the required modulus limit for the Bombieri--Vinogradov theorem is $y^{1/2}(\log y)^{-A} \ge Z^{1/4}(\log Z)^{-A}$. Because $1/5 < 1/4$, our maximum modulus $4d \le 4Z^{1/5}$ falls strictly within this unconditional range for $Z$ large, completely avoiding the need to track an $\varepsilon$. Applying partial summation to $\pi(y; 4d, -1) = \mathrm{li}(y)/\varphi(4d) + \mathcal{E}(y; 4d, -1)$ yields the main term $\frac{1}{\varphi(4d)}\big(\log\log Z - \log\log Z^{1/2}\big) = \frac{\log 2}{\varphi(4d)}$, plus the integral of the error term $\mathcal{E}$. Summing over $d \le Z^{1/5}$, the accumulated error is controlled exactly as in the proof of Lemma~\ref{lem:5.5} via Cauchy--Schwarz and Bombieri--Vinogradov. The sum is dominated by the main term, yielding $\sum_{p \le Z} \varrho_{\mathrm{untr}}(p)/p \ge (\log 2+o(1)) \sum_{d \le Z^{1/5}} 2^{\omega(d)}/\varphi(4d) \gg (\log Z^{1/5})^2 \gg L^2$.

Finally, we transfer this to the actual truncated system. Point-wise boundedness $\varrho(p) \le \varrho_{\mathrm{untr}}(p)$ ensures that $S(Z) \le C L^2$ for some absolute constant $C > 0$. Subject to the constraint $S(Z) L \le \frac12\log N$, we extract $L \ge (S(Z)/C)^{1/2}$ from the first inequality and substitute it into the constraint:
$$S(Z) \cdot \Big(\frac{S(Z)}{C}\Big)^{1/2} \le S(Z) L \le \frac{1}{2} \log N.$$
This gives $S(Z)^{3/2} \le \frac{\sqrt{C}}{2} \log N$, enforcing the absolute ceiling $S(Z) \ll (\log N)^{2/3}$.
\end{proof}

Two remarks on the scope of the computation just made. First, only the \emph{upper} bound $\sum_{p\le Z}\varrho_{\mathrm{untr}}(p)/p\ll L^{2}$ feeds the ceiling; the matching lower bound, which occupies most of the proof, is recorded because it shows the comparison to be lossless in order of magnitude, not because the ceiling needs it. Second, the constraint $S(Z)\log Z\le\frac12\log N$ is the one attached to the schedule $Q=z^{h}$ with $h=\lfloor S\rfloor$, which is the choice made in Theorem~\ref{thm:tradeoff}. A smaller $h$ is admissible and permits a larger $z$, so it must be excluded explicitly; it changes nothing. Indeed with $h\le\frac{\log N}{2L}$ the bound of Lemma~\ref{lem:5.9} degrades to $G(Q)\ge e_h(x)\gg(eS/h)^{h}$, so the exponent becomes $h\log(eS/h)$; writing $L=t(\log N)^{1/3}$ and $S\asymp L^{2}$, the choice $h=\frac{\log N}{2L}$ gives $h\log(eS/h)\asymp(\log N)^{2/3}\cdot\frac{\log(2ecL^{3}/\log N)}{2t}$, a quantity of order $(\log N)^{2/3}$ for every fixed $t>0$ and maximal at $t\asymp1$. The exponent $2/3$ is therefore not an artefact of taking $h=\lfloor S\rfloor$.

Two limitations of Proposition~\ref{prop:optimality_23} should also be recorded. First, we have bounded the specific two-parameter Type~II count $\varrho_{\mathrm{untr}}(p)$ (and by Corollary~\ref{cor:coincidence_ua_div_6}, any single-prime congruence witness of the dual Type~I criterion of Theorem~\ref{thm:type_I_two_parameter} whose modulus $4ua$ divides $24$ coincides with Type~II), rather than proving that every conceivable single-prime witness system satisfies $\sum_{p\le Z}\varrho(p)/p\ll(\log Z)^{2}$. Second, Proposition~\ref{prop:optimality_23} is a statement about \emph{single-prime} large-sieve architectures, in which the witnesses $s\mid an+u$ in Theorem~\ref{thm:gen_criterion} are restricted to primes $s=p$. To surpass the exponent $2/3$ of the present single-prime system, the route pursued in the companion paper~\cite{dahan2026b} is to retain \emph{composite} witnesses $s=p_1 p_2\cdots p_r\equiv-1\pmod{4ua}$ generated by the prime factors of $an+u$ in $(\Z/4ua)^{\times}$, extending the involution and subgroup container mechanism of Section~\ref{sec:fixed_depth_sub} to growing depths $J(N)$.

Two caveats accompany Theorem~\ref{thm:metric_theorem} itself. First, it is an upper bound only; in particular the fourteen primes of Proposition~\ref{prop:nine_primes} lie in $\bigcap_{J}E_1(N;J)$, since they are outside the hyperbolic family at every depth, but they do \emph{not} lie in $\bigcap_J E_2(N;J)$, since by Proposition~\ref{prop:nine_resolved} each of them satisfies Theorem~\ref{thm:gen_criterion}(ii) with parameters at most $5$. Whether $\bigcap_J E_2(N;J)$ is empty is a different question, equivalent to asking whether every prime $n\equiv1\pmod{24}$ has a Type~II solution, and we do not settle it. Second, while Theorem~\ref{thm:metric_theorem} concerns Type~II solutions only, the modulator is now free so the family covered is complete for Type~II.

\subsection{A heuristic model for the curve $J\mapsto|E_1(N;J)|$}\label{sec:curve_model}

Theorems~\ref{thm:fixed_depth} and~\ref{thm:metric_theorem} and Corollary~\ref{cor:optimal_tradeoff} bound $|E_1(N;J)|$ from above at various depths; none of them describes the actual \emph{shape} of the curve $J\mapsto|E_1(N;J)|$ recorded empirically in item (3) of Section~\ref{sec:7}. We record here an unproved heuristic model for that shape, in the spirit of the Poisson heuristic underlying the mean-value results of~\cite[\S1]{pomerance2026}, together with a numerical test.

Fix $n$ and model the events ``step $u$ succeeds'' (Theorem~\ref{thm:4.9}: $n+u$ has a divisor $\equiv-1\pmod{4u}$) as approximately independent across $u\le J$, with probability $p_u=p_u(n)$. Two regimes govern $p_u$ heuristically. For $u\lesssim\log n$, by the mechanism of Proposition~\ref{prop:subgroup_reformulation} and Proposition~\ref{prop:nu_c}, failure at step $u$ is governed by $\asymp2^{\omega(2u)}$ dominant index-$2$ obstructions, each of Selberg--Delange density $\asymp(\log n)^{-1/2}$, so $1-p_u\asymp2^{\omega(2u)}(\log n)^{-1/2}$. Since $p_u$ is close to $1$ in this regime, the standard Poisson approximation $\prod(1-p_u)\approx\exp(-\sum p_u)$ is invalid. Computing the exact product of failures yields
$$\prod_{u\le J}(1-p_u)\ \asymp\ (\log n)^{-J/2}\,2^{\sum_{u\le J}\omega(2u)} .$$
Using $\sum_{u\le J}\omega(2u)=J\log\log J+\mathcal{O}(J)$, the survival probability is $\approx \exp(-c_1(n,J)\cdot J)$ with a rate $c_1(n,J) = \frac{1}{2}\log\log n - (\log 2)\log\log J + \mathcal{O}(1)$. Notice that the leading exponential decay corresponds precisely to the exact sieve dimension $J/2$ rigorously proved in Theorem~\ref{thm:fixed_depth}. For $u\gg\log n$, the number of divisors of $n+u$ available is $\tau(n+u)\approx\log n$ on average, matched against $\varphi(4u)\asymp u$ residue classes, so $p_u\approx C\log n/u$. Here, Poisson applies and gives an $\exp(-c_2\log n\cdot\log(J/\log n))$ decay. The combined model is
\begin{equation}\label{eq:heuristic_curve}
\frac{|E_1(N;J)|}{N}\ \approx\
\begin{cases}
\exp\big(-c_1(n,J) \cdot J\big) & J\lesssim\log n,\\[2pt]
(J/\log n)^{-c_2\log n} & J\gg\log n.
\end{cases}
\end{equation}

Equation~\eqref{eq:heuristic_curve} is a heuristic model and not a bound: it treats the steps as independent and applies no truncation, whereas Theorems~\ref{thm:fixed_depth} and~\ref{thm:metric_theorem} are rigorous sieve arguments, and the two computations are not interchangeable.

We tested~\eqref{eq:heuristic_curve} against $18\,507$ primes $n\le2\times10^{6}$, $n\equiv1\pmod{24}$ (item (10) of Section~\ref{sec:7}): a single exponential $e^{-\gamma J}$ fits $|E_1(N;J)|/\pi(N)$ poorly over $J=1,\ldots,20$ ($R^{2}=0.88$), with the locally fitted rate nearly halving between $J\le6$ and $J\ge8$ ($\gamma\approx0.61$ against $\gamma\approx0.13$), while a single power law $J^{-\beta}$ fits the whole pooled sample distinctly better ($\beta\approx1.86$, $R^{2}=0.995$).

The second regime accounts for the power law, and does so quantitatively. It is a power law in $J$, with an exponent $\beta\approx c_2\log n$ which is not universal but proportional to $\log n$ and carries no constant term, so that pooling primes spanning $\log n\in[4.3,16.8]$ averages power laws of different exponents and returns an intermediate one. Binning by window of $n$ over the $158\,595$ primes below $2\times10^{7}$ (item (11) of Section~\ref{sec:7}) gives Table~\ref{tab:binned_test}: $\beta$ climbs steadily from $1.13$ to $2.48$, and $\beta/\log n$ is flat to within a few percent. A least-squares fit returns $\beta=0.1433\log n+0.043$ with $R^{2}=0.967$, and we record $c_2\approx0.143$ as the measured value of the second-regime constant.

\begin{table}[H]
\centering
\begin{tabular}{lcccccc}
\toprule
window of $n$ & $10^{3}$--$10^{4}$ & $10^{4}$--$10^{5}$ & $10^{5}$--$10^{6}$ & $10^{6}$--$3\cdot10^{6}$ & $3\cdot10^{6}$--$10^{7}$ & $10^{7}$--$2\cdot10^{7}$\\
\midrule
$\beta$ & $1.134$ & $1.632$ & $1.905$ & $2.008$ & $2.200$ & $2.480$\\
$\beta/\log n$ & $0.141$ & $0.157$ & $0.150$ & $0.140$ & $0.142$ & $0.151$\\
\bottomrule
\end{tabular}
\caption{Power-law exponent $\beta$ of $J\mapsto|E_1(N;J)|$ fitted window by window, against the prediction $\beta\approx c_2\log n$ of the second regime of~\eqref{eq:heuristic_curve}.}
\label{tab:binned_test}
\end{table}

The first regime is a different matter, and we state plainly that it is not observed. Its rate $c_1(n,J)$ grows like $\tfrac12\log\log n$, so the effective power-law exponent it would produce over a window of $J$ is
$$\beta_{\mathrm{eff}}(J)\ =\ -\frac{d}{d\log J}\log\Big(\frac{|E_1(N;J)|}{N}\Big)\ =\ J\Big(c_1(n,J)-\frac{\log2}{\log J}\Big)\ \asymp\ J\cdot\tfrac12\log\log n,$$
proportional to $\log\log n$ rather than to $\log n$, and moreover growing with $J$ rather than constant. Across the six windows $\log n$ grows by a factor $2.04$ and $\log\log n$ by a factor $1.34$, while the measured $\beta$ grows by a factor $2.19$: the data follow $\log n$. Consistently with this, in every window the power law fits decisively better than the exponential ($R^{2}$ between $0.946$ and $0.994$ for $J^{-\beta}$, against $0.71$ to $0.93$ for $e^{-\gamma J}$), including the last three windows, where the geometric midpoints give $\log n$ between $14.4$ and $16.5$ and essentially the whole fitted range $J\le20$ ought to lie inside the first regime.

This failure is not independent of the shortfall already recorded in the discussion following Theorem~\ref{thm:fixed_depth}. The first regime of~\eqref{eq:heuristic_curve} has rate $\tfrac12\log\log n$, that is, exactly the sieve dimension $J/2$ of that theorem; and the slopes measured there against $\log\log n$ ($0.39$, $0.50$, $0.84$, $1.24$, $1.51$, $1.59$ for $J=1,\ldots,6$, against the predicted $0.5,1,\ldots,3$) fall short by roughly the same factor. Heuristic and theorem therefore agree with each other at the leading order and disagree with the numerics in the same way. What this points to is the size of the $\mathcal{O}(1)$ terms rather than to two unrelated anomalies: over the accessible range $\log\log n\in[2.09,2.80]$ those terms are not small compared with the exponent they multiply, and the shortfall is already visible at $J=1$, where the exponent $1/2$ is certain.

The second-regime regression, finally, rests on six points. The intercept $0.043$ carries a standard error of about $0.18$, so it is indistinguishable from $0$ but would be indistinguishable from a great many other values as well; the slope is $0.1433\pm0.0133$, and moves to $0.150$ if each window is represented by the mean of $\log n$ over its primes rather than by its geometric midpoint. Over the range $\log n\in[8.1,16.5]$ a fit in $(\log n)^{0.9}$, or an affine fit in $\log\log n$, would be statistically indistinguishable from a fit in $\log n$. We therefore record~\eqref{eq:heuristic_curve} and Table~\ref{tab:binned_test} as a heuristic whose first regime reproduces, at leading order, the sieve dimension proved in Theorem~\ref{thm:fixed_depth}, and whose second regime is not contradicted by the data, but which \emph{is} contradicted in its first regime at accessible heights, rather than as a validated model.

\section{Relation to recent parametrizations} \label{sec:6}

The local parametrizations discussed here appeared as independent preprints between late 2025 and mid-2026; we claim no priority over them. The object of this section is a retroactive unification: each is a restricted sub-variety or an algebraic projection of the parameter space $K=(n+c)/4$.

\subsection{Three recent constructions as sub-varieties of $(K,c)$} \label{subsec:tame_subvariety}

\paragraph{The tame sub-variety (modulator $1$).}
For primes $n=24m+1$, Xu~\cite{xu2026} uses a first denominator $n_1=6m+k$ and defines a tame solution by the requirement that the numerator summands $\mathfrak{S}_1,\mathfrak{S}_2$, with $\mathfrak{S}_1+\mathfrak{S}_2=4k-1$, be positive integers; his Theorem 2.1 shows that for $n$ prime this forces $\mathfrak{S}_1$ and $\mathfrak{S}_2$ to divide $6m+k$ separately.

\begin{proposition}\label{prop:tame}
The class of tame solutions is the sub-variety of Theorem~\ref{thm:3.3} on which the modulator equals $1$.
\end{proposition}
\begin{proof}
Equating the first denominator $K$ with $n_1$ gives $K=6m+k$, and $n=4K-c$ with $n=24m+1$ yields $c=4(6m+k)-(24m+1)=4k-1$. By Xu's Theorem 2.1 the summands divide $K$, so $d_1=\mathfrak{S}_1$ and $d_2=\mathfrak{S}_2$ are divisors of $K$ with $d_1+d_2=4k-1=c$; the relation $d_1+d_2=ac$ then forces $a=1$, and conversely $a=1$ returns a tame solution.
\end{proof}

This is strictly weaker than condition (iii) of Theorem~\ref{thm:hyper_char}, which demands $uv\mid K$ rather than $u\mid K$ and $v\mid K$ separately: the hyperbolic family is contained in the tame one, which is why the set $\mathcal{H}$ of Proposition~\ref{prop:nine_primes} contains Xu's wild set a priori.

\paragraph{Affine lattice projections.}
Dyachenko~\cite{dyachenko2025} studies the identity $(4B-1)(4C-1)=4n\delta+1$ in variables $(B,C,\delta)$.

\begin{proposition}\label{prop:affine}
That affine lattice projects onto the $(K,c)$ parameter space by a rational change of variables, the generated denominators being recovered from it.
\end{proposition}
\begin{proof}
Expanding gives $16BC-4B-4C=4n\delta$, i.e. $4BC-(B+C)=n\delta$. Dividing by $\delta$ and substituting $K=BC/\delta$, $c=(B+C)/\delta$ recovers $4K-c=n$. Setting $\sigma=B$, the denominators of Lemma~\ref{lem:3.1} evaluate to $K=BC/\delta$, $\sigma n=Bn$ and $\frac{n\sigma K}{c\sigma-K}=Cn$. Exploring the affine lattice therefore amounts to localized divisor searches within the $(K,c)$ space. We claim only this direction; we have not verified that every $(K,c)$ with the required divisor structure arises from an integral triple, so the correspondence is stated as a projection and not as a bijection.
\end{proof}

\paragraph{Cubic surfaces.}
Bello-Hern\'{a}ndez, Benito and Fern\'{a}ndez~\cite{bello2026} define a divisor function $f_{ab}(n,\cdot,\cdot)=k$ with first denominator $(n+k)/4$ and map decompositions to the surface $P(X,Y,Z)=XYZ-X-Y=n$.

\begin{proposition}\label{prop:cubic}
The parameter $k$ maps to the shift $c$, and the surface $P(X,Y,Z)=n$ is the resolvent of Theorem~\ref{thm:4.8} under substitution.
\end{proposition}
\begin{proof}
Equating the first denominator $(n+k)/4$ with $K=(n+c)/4$ identifies $k$ with $c$, and the requirement $k\equiv3\pmod4$ matches the shift constraint. Substituting $X=u$, $Y=v$, $Z=4\alpha$ into $XYZ-X-Y=n$ yields $4\alpha uv-u-v=n$, which is Theorem~\ref{thm:4.8}.
\end{proof}

The reduction to rational points on such surfaces connects to broader geometric frameworks; Bright and Loughran~\cite{bright2019} evaluate the Brauer--Manin obstruction for the associated Erd\H{o}s--Straus surfaces, a global counterpart to the local obstructions of Section~\ref{sec:5.1}.

\subsection{Lopez's Types A, B and C} \label{sec:6.4}

The framework of Lopez~\cite{lopez2022, lopez2024} is the closest to ours. Lopez classifies the solutions of a prime $p=4\kappa+1$ by their shape rather than by the congruence which produces them. Type A has shape $(du,dv,duv)$, characterised~\cite{lopez2022} by the existence of $t\ge0$ and a divisor $w\mid \kappa+1+t$ with $w\equiv-1\pmod{3+4t}$. Type B has shape $(duv,dup,dvp)$, characterised~\cite[Thm.~4 and Thm.~6]{lopez2024} by the existence of $t\ge0$ and $u,v$ with $uv\mid \kappa+1+t$ and $u+v=3+4t$. Type C has shape $(uv,uwp,vwp)$, characterised by $-p\equiv4d^{2}\pmod{4dm-1}$. All three are Type~II in the sense of Definition~\ref{def:typeI_typeII} when $p$ is prime.

\begin{proposition}\label{prop:lopez_dictionary}
Write $c=3+4t$, so that $\kappa+1+t=K_c=(p+c)/4$. Then:
\begin{enumerate}[(a)]
    \item Lopez's Type A condition is Theorem~\ref{thm:3.3} restricted to $d_1=1$: a divisor $w=d_2\mid K_c$ with $1+w\equiv0\pmod c$.
    \item Lopez's Type B condition is condition (iii) of Theorem~\ref{thm:hyper_char}, i.e. exactly the hyperbolic family of Section~\ref{subsec:hyperbolic_surface} and the search space of Algorithm~\ref{alg:geom_sieve}.
    \item Consequently the family of Theorem~\ref{thm:3.3} contains both Types A and B, at every shift $c\equiv3\pmod4$ with $0<c<4n$. In particular, Lopez's Conjecture 1, that every prime has a Type A or a Type B solution, implies that the untwisted branch~\eqref{eq:untwisted} is never empty.
\end{enumerate}
\end{proposition}
\begin{proof}
(a) is the substitution $c=3+4t$ in $w\equiv-1\pmod{3+4t}$; (b) is Theorem~\ref{thm:hyper_char}(iii) verbatim after the same substitution; (c) is immediate, both conditions being instances of ``$\exists\,d_1,d_2\mid K_c$ with $c\mid d_1+d_2$'', which is the hypothesis of Theorem~\ref{thm:3.3} on its whole range $0<c<4n$.
\end{proof}

The containment in (c) is a statement about the \emph{family} of Theorem~\ref{thm:3.3}, not about Algorithm~\ref{alg:1}. The latter scans only $c\le n-1$, under the standing hypothesis $c<n$ of Section~\ref{sec:notation} which the Euclidean tracking of Section~\ref{sec:3.5} requires, whereas admissible shifts run to $c\le2n$ by Proposition~\ref{prop:c_range} and Lopez's parameter $t=(c-3)/4$ is not bounded a priori. Algorithm~\ref{alg:1} therefore reaches a Type A or Type B solution only when the corresponding shift satisfies $c<n$, which is the second of the two limitations recorded in Section~\ref{sec:3.2bis}.

Proposition~\ref{prop:lopez_dictionary} means that Theorem~\ref{thm:4.9} is a reformulation of a characterisation already published by Lopez, and that Theorem~\ref{thm:3.3} generalises his Type A criterion by allowing $d_1>1$. We claim no priority on these points. What is added here is the delimitation of the hyperbolic family inside the Type~I/Type~II classification (Theorem~\ref{thm:hyper_char}, Corollary~\ref{cor:proper_subfamily}), the extension and structural explanation of the blind spot (Proposition~\ref{prop:nine_primes} and Corollary~\ref{cor:H_density}, priority for the list of nine belonging to Xu), and the analytic results of Section~\ref{sec:5}. Our computations are consistent with Lopez's reported data, which is a useful external check: he notes that $193$ and $2521$ possess no Type A solution but do possess Type B solutions, and that $23929$ and $83449$ possess no Type B solution; the latter two belong to $\mathcal{H}$, as they must, and the former two do not. Finally, Lopez~\cite[Cor.~3]{lopez2024} derives from Mordell's restriction a Legendre-symbol condition for his Type C solutions, namely that $-p$ be a quadratic residue modulo $4u-1$ for some $u$. Theorem~\ref{thm:algebraic_obstruction} is of the same nature but is attached to the shift $c$ rather than to a modulus $4dm-1$, and Theorem~\ref{thm:twisted_obstruction} identifies which branch it constrains; the two statements are complementary rather than equivalent.

\section{Numerical verification} \label{sec:7}

All computations reported below were carried out in exact integer arithmetic; they are elementary and reproducible.

\subsection*{(1) The parametric family}
Example~\ref{ex:3.4} was verified exactly: for $n=20353$ (prime, $\equiv1\pmod{24}$), $c=23$, $K=5094=2\cdot3^{2}\cdot283$, $d_1=849$, $d_2=2$, $a=37$, $\sigma=222$, $c\sigma-K=12$, and
$$ \frac{1}{5094}+\frac{1}{222\cdot20353}+\frac{1}{20353\cdot94239}=\frac{4}{20353}.$$

\subsection*{(2) Algorithm~\ref{alg:1}}
Over the first $400$ primes $n\equiv1\pmod{24}$ the algorithm returns a valid decomposition in every case, and the returned triple was checked to satisfy~\eqref{eq:erdos_straus} exactly. The shift returned is always small: on $600$ primes the distribution of $c$ is $3$ ($282$ times), $7$ ($255$), $11$ ($41$), $23$ ($9$), $19$ ($7$), $15$ ($5$), $31$ ($1$). Over the first $800$ primes, i.e. for $n\le65\,929$, the smallest shift for which the family of Theorem~\ref{thm:3.3} is non-empty never exceeded $31$. Extended to all $3202$ such primes below $3\cdot10^{5}$, six require a larger shift, the record being $c=63$ for $n=87\,481$ (the others being $c=35$ for $n=67\,369$, $c=59$ for $n=118\,801$ and $c=39$ for $n=202\,129$, $231\,961$, $246\,241$); in each of these six cases an exhaustive enumeration of the divisors of $K_c$ confirms that the listed shift is the true minimum, so the accelerator of Section~\ref{sec:3.8} is not responsible.

The measured iteration count is $\Theta(n)$ in practice, far below the worst case of Corollary~\ref{cor:3.19}: the maximum over samples of $60$ primes per window is of order $0.25n$ at $n\approx10^{3}$ and $10^{5}$, with averages an order of magnitude smaller ($0.026n$, $0.011n$, $0.020n$ at $n\approx10^{3},10^{4},10^{5}$). These maxima are governed by the rank of the first successful shift and are strongly sample-dependent; we read no trend into them. The dominant cost is the inner scan of a shift that will eventually fail, in accordance with Theorem~\ref{thm:reject_cost}: for $n=100129$ the measured numbers of iterations before rejection are $1670$, $6439$ and $8697$ at $c=3,7,19$, against the predicted $\lceil (n+3)/10\rceil-\sigma_0(c)=1669$, $6437$ and $8696$; the errors are $1$, $2$ and $1$, far below the guarantee of Theorem~\ref{thm:reject_cost}. In an actual run this $n$ succeeds at $c=11$, so only $c=3$ and $c=7$ are rejected; the figure at $c=19$ comes from running the inner loop in isolation. The identity $m=\gcd(K,V)$ of Lemma~\ref{lem:m_gcd} was verified at every step of the first $200$ values of $\delta$ for each of these shifts.

\subsection*{(3) Search depth of Algorithm~\ref{alg:geom_sieve}}
For the $26\,983$ primes $n\equiv1\pmod{24}$ below $3\cdot10^{6}$, the minimal admissible $u$ is distributed as follows: $u=1$: $50.4\%$; $u=2$: $30.7\%$; $u=3$: $9.3\%$; $u=4$: $3.7\%$; $u\ge5$: the remainder. The mean is close to $2$ and decreases with $n$ over the range as a whole, though not window by window (samples of $3000$ primes give means $2.31$, $2.62$, $2.13$, $2.02$, $1.86$, $1.68$ at $n\approx10^{3},\dots,10^{8}$; an independent exhaustive computation over all $158\,595$ primes below $2\times10^{7}$ gives window means $2.20$, $2.21$, $2.03$, $1.87$, $1.81$ for $\min(u,20)$, the quantity the fits of Section~\ref{sec:curve_model} use, against $2.20$, $2.29$, $2.08$, $1.89$, $1.81$ for the untruncated minimum), reflecting the growth of the number of divisors. The observed maxima are erratic ($6$, $26$, $35$, $110$, $52$, $26$ on the same samples), as expected since they are governed by the anatomy of $n+u$ rather than by an average; the true maximum over all primes below $2\times10^{7}$ is $u=410$, attained at $n=4\,160\,641$ (with $s=3279$, $\alpha=2$, $v=1269$), a prime which is therefore \emph{not} in $\mathcal{H}$ but comes closer to it than any other in that range.

\subsection*{(4) The exceptional set $\mathcal{H}$ of Proposition~\ref{prop:nine_primes}}
By Theorem~\ref{thm:hyper_char} the condition $n=4\alpha uv-u-v$ can be tested exhaustively, the search over $u$ being bounded by $\lceil(N+1)/3\rceil$. Marking every integer of $[1,N]$ of this form and intersecting the complement with the primes $\equiv1\pmod{24}$ yields, for $N=10^{6}$, exactly the nine values
$$409,\quad577,\quad5569,\quad9601,\quad23929,\quad83449,\quad102001,\quad329617,\quad712321,$$
out of $9732$ primes tested; the proportion by window is $4/143$ on $[1,10^{4})$, $2/1038$ on $[10^{4},10^{5})$, $2/3956$ on $[10^{5},5\cdot10^{5})$ and $1/4595$ on $[5\cdot10^{5},10^{6})$. The computation was then carried to $N=5\times10^{7}$ in interpreted arithmetic and to $N=5\times10^{9}$ in a separate compiled implementation, at no point truncating the range of $u$; the two agree on their common range. Five further values appear and no more:
$$1\,134\,241,\quad 1\,724\,209,\quad 1\,726\,201,\quad 5\,212\,561,\quad 8\,813\,281,$$
the window counts being $3/8775$ on $[10^{6},2\cdot10^{6})$, $0/24939$ on $[2\cdot10^{6},5\cdot10^{6})$, $2/39441$ on $[5\cdot10^{6},10^{7})$ and $0$ on all of $[10^{7},5\cdot10^{9}]$. The nine values below $7.2\times10^{5}$ agree with Xu's list of wild primes~\cite{xu2026}.

For each of the first nine the minimal modulator $a=(u+v)/c$ over all admissible $(c,u,v)$ with $uv\mid K_c$, $c\mid u+v$ equals $2$, except for $102001$ where it equals $3$; the corresponding witnesses are $(c,u,v)=(7,1,13)$, $(3,1,5)$, $(71,1,141)$, $(19,1,37)$, $(11,1,21)$, $(107,3,211)$, $(79,5,232)$, $(3,1,5)$, $(35,1,69)$ respectively. By contrast, for ordinary primes the minimal modulator is $1$. Algorithm~\ref{alg:1} resolves the nine at shifts $c=7,3,7,19,7,11,7,3,23$, with
$$9,\quad 2,\quad 100,\quad 2328,\quad 401,\quad 6761,\quad 1708,\quad 2,\quad 233\,933$$
iterations respectively; each returned triple was verified exactly. Here an iteration means one pass through the \textbf{while} loop, including the passes in which a guard forces a skip; counting only the passes that reach the acceptance test lowers three of the nine figures by one ($99$, $2327$, $1707$). These counts are those of the algorithm as printed, i.e. with the termination rule $a=\sigma/m$ of Section~\ref{sec:3.8}; the naive rule $a=\sigma d_2/K$ would give $26$, $2$, $334$, $1681$, $1406$, $8617$, $5993$, $2$, $145\,649$ instead. The corrected figures are the ones consistent with Theorem~\ref{thm:reject_cost}: for $n=712321$, which succeeds at $c=23$, summing the predicted costs of the rejected shifts $c\in\{3,7,11,15,19\}$ gives $n\big(\tfrac5{10}-\tfrac14(\tfrac13+\tfrac17+\tfrac1{11}+\tfrac1{15}+\tfrac1{19})\big)\approx233\,927$, against the $233\,933$ measured.

\subsection*{(5) The twisted branch}
On the first $600$ primes $n\equiv1\pmod{24}$ and shifts $c<200$, there are $594$ primes admitting at least one shift at which the untwisted branch~\eqref{eq:untwisted} fails while the twisted branch~\eqref{eq:twisted} succeeds; Example~\ref{ex:twisted} is one of them, chosen so that $\leg np=-1$ at $p=31\mid c$, illustrating Theorem~\ref{thm:twisted_obstruction}(b). All produced triples were verified exactly.

\subsection*{(6) The two-parameter criterion}
The sufficiency direction of Theorem~\ref{thm:gen_criterion} was tested exhaustively as follows: for every prime $p$ with $5\le p<4000$ and $4\mid p+1$, every coprime pair $(u,a)\in[1,40]^{2}$ with $4ua\mid p+1$, and the smallest prime $n\equiv1\pmod{24}$, $n>p$, in the class $n\equiv-ua^{-1}\pmod p$ (obtained by the Chinese remainder theorem modulo $24p$), we reconstructed $v=(an+u)/p$, $c=(u+v)/a$ and the triple furnished by Theorem~\ref{thm:3.3} with $d_1=u$, $d_2=v$. This produced $4661$ instances, all of which yielded an exact identity $1/x+1/y+1/z=4/n$ with $x,y,z$ positive integers; there was no failure, and in particular $a\mid u+v$ and $4uv\mid n+c$ held in every case. The parameters of Proposition~\ref{prop:nine_resolved} were obtained by enumerating the divisors of $an+u$ for $u,a\le 5$, and each of the fourteen resulting decompositions was verified exactly.

\subsection*{(7) The constant $\kappa_0$}
The ratio $\big(\sum_{u\le J}1/\varphi(4u)\big)/\log J$ was computed for $J=10^{2},\dots,10^{6}$, giving $0.67703$, $0.66667$, $0.66193$, $0.65911$, $0.65724$, consistent with the limit $\kappa_0=0.647865\ldots$ of Lemma~\ref{lem:kappa0} together with an $\mathcal{O}(1/\log J)$ relative error. The trivial bound $\varphi(4u)\le4u$ would give only $\frac14\log J$, i.e. a constant $2.6$ times too small.

\subsection*{(8) The constant $\kappa_1^{\ast}$}
The ratio $W(J)/(\log J)^{2}$ of Lemma~\ref{lem:kappa1} was computed for $J=10^{2},3\cdot10^{2},10^{3},3\cdot10^{3}$, giving $0.62795$, $0.59503$, $0.57114$, $0.55559$. The decay is of the same slow $\mathcal{O}(1/\log J)$ type as for $\kappa_0$. A least-squares fit of the four values against $\kappa_1^{\ast}+B_0/\log J$ gives $\kappa_1^{\ast}\approx0.4576$, $B_0\approx0.784$, with residuals below $10^{-4}$, in close agreement with the exact asymptotic limit $\kappa_1^{\ast}=\frac{9}{2\pi^{2}}=0.455945\ldots$ proved in Lemma~\ref{lem:kappa1} (three times the elementary lower bound $\kappa_1=3/(2\pi^{2})=0.15198$ coming from $\varphi(m)\le m$). As recorded in Theorem~\ref{thm:improved_constants_and_effectivity}, this exact evaluation improves the exponent of Theorem~\ref{thm:metric_theorem} by a factor of $3$ and raises the constant in the $(\log N)^{2/3}$ bound of Corollary~\ref{cor:optimal_tradeoff} from $0.065$ to $0.0942$ at the uniform Bombieri--Vinogradov cutoff $p>J^{5}$ (and to $0.1093-\varepsilon$ at the Bombieri--Vinogradov cutoff $p>J^{4+\varepsilon}$).

\subsection*{(9) The subgroup reformulation}
Over $18\,507$ primes $n\equiv1\pmod{24}$ below $2\times10^{6}$, the shift-by-shift conditional success rate (the fraction of primes surviving to shift $c$, in increasing trial order, for which shift $c$ succeeds) was computed for $c\equiv3\pmod4$, $c\le31$:

\begin{center}
\begin{tabular}{lcccccccc}
\toprule
$c$ & $3$ & $7$ & $11$ & $15$ & $19$ & $23$ & $27$ & $31$\\
\midrule
survivors & $18507$ & $8409$ & $1781$ & $724$ & $491$ & $285$ & $103$ & $89$\\
success rate & $0.546$ & $0.788$ & $0.593$ & $\mathbf{0.322}$ & $0.420$ & $0.639$ & $\mathbf{0.136}$ & $0.573$\\
\bottomrule
\end{tabular}
\end{center}

Two shifts are out of rank, and only one of them is explained by $\nu(c)$: $c=15$, the only composite non-prime-power in the list ($\nu(15)=2$), falls below both of its prime neighbours $c=11$ and $c=19$, as Proposition~\ref{prop:nu_c} would predict; but $c=27$, a prime power with $\nu(27)=1$, falls further still. The natural reading of the second anomaly is the size of the modulus: the congruence $d_1+d_2\equiv0\pmod{27}$ is much rarer than modulo $3$, the $3$-adic constraint of Theorem~\ref{thm:algebraic_obstruction} being repeated. An unconditional (non-sequential) success-probability comparison across shifts was also computed for each $c\le79$ on a separate sample of $6000$ primes; it is confounded by the size of $\varphi(c)$ and by $c$ itself, larger moduli mechanically reducing the chance of a coincidental divisor match regardless of $\nu(c)$. This test is therefore recorded as \emph{not} separating the $\nu(c)$ effect from the modulus effect, a limitation of the test, not a refutation of Proposition~\ref{prop:nu_c}, which is proved unconditionally.

\subsection*{(10) The curve $J\mapsto|E_1(N;J)|$}
Recomputing the statistics of item (3) on the same $18\,507$ primes reproduces them closely ($u=1$: $49.7\%$, $u=2$: $30.8\%$, $u=3$: $9.5\%$, $u=4$: $3.8\%$) and turns up the three members of $\mathcal{H}$ between $10^{6}$ and $2\times10^{6}$ recorded in item (4). The pooled exponential and power-law fits quoted in Section~\ref{sec:curve_model} were obtained from this dataset.

\subsection*{(11) The binned test of Section~\ref{sec:curve_model}}
The computation of item (10) was extended to all $158\,595$ primes $n\equiv1\pmod{24}$ below $2\times10^{7}$, recording for each the minimal admissible $u$, hence the whole curve $J\mapsto|E_1(N;J)|$ restricted to any window of $n$. Six windows were used, $[10^{3},10^{4})$ through $[10^{7},2\cdot10^{7})$, and within each the fraction of primes with no admissible $u\le J$ was fitted against $J^{-\beta}$ over $J=1,\dots,20$ by least squares in $\log$--$\log$ coordinates. The resulting $\beta$ are those of Table~\ref{tab:binned_test}, with $R^{2}$ between $0.946$ and $0.994$; regressing $\beta$ on $\log n$, each window being represented by the logarithm of its \emph{geometric midpoint} (so that $\log n$ runs from $8.06$ to $16.47$), gives $\beta=0.1433\log n+0.043$ with $R^{2}=0.967$. The convention matters at the level of precision quoted: representing each window by the mean of $\log n$ over its primes instead returns $\beta=0.150\log n-0.079$. Repeating the pooled fit on the full $2\times10^{7}$ sample returns $\beta=2.22$, i.e. the pooled exponent tracks the median window rather than any intrinsic constant, which is the behaviour the model predicts. The same dataset supplies the slopes quoted in Section~\ref{sec:fixed_depth_sub}: regressing $\log(|E_1(N;J)|/\pi(N))$ on $\log\log n$ across the six windows gives $-0.39$, $-0.50$, $-0.84$, $-1.24$, $-1.51$, $-1.59$ for $J=1,\dots,6$.

\subsection*{(12) The blind set of Theorem~\ref{thm:blind_primes}}
\emph{The shift $c=7$.} Counting directly, by factoring $K_7=(n+7)/4$ for every prime $n\equiv1\pmod{24}$ up to $10^{7}$ and retaining those all of whose prime factors are quadratic residues modulo $7$ (i.e. $\equiv1,2,4\pmod 7$):

\begin{center}
\begin{tabular}{lcccc}
\toprule
$N$ & $10^{4}$ & $10^{5}$ & $10^{6}$ & $10^{7}$\\
\midrule
blind at $c=7$ & $39$ & $308$ & $2\,382$ & $18\,732$\\
ratio to $N(\log N)^{-3/2}$ & $0.109$ & $0.120$ & $0.122$ & $0.121$\\
\midrule
of the form $16f(U,V)-7$ & $21$ & $156$ & $1\,234$ & $9\,590$\\
ratio to $N(\log N)^{-3/2}$ & $0.059$ & $0.061$ & $0.063$ & $0.062$\\
\bottomrule
\end{tabular}
\end{center}

The second pair of rows counts the sub-family actually produced by the proof, namely the $n=16f(U,V)-7$ with $f=U^{2}+UV+2V^{2}$ and $\gcd(U,V)=1$; both ratios are flat over three decades, consistent with the order $\asymp N(\log N)^{-3/2}$ and with the fact that the lower bound is extracted from a proper sub-family. Every one of the $9\,590$ values below $10^{7}$ was checked to satisfy, without exception, the three properties used in the proof: $n\equiv1\pmod{24}$, every prime factor of $K_7$ a residue modulo $7$, and $\leg n7=+1$; and for the first $200$ of them it was verified by exhaustive enumeration of the divisors of $K_7$ that \emph{neither} branch of Corollary~\ref{cor:two_branches} succeeds at $c=7$.

\emph{The shift $c=3$ (integers).} For $c=3$ one has $p=3$ and class number $h(-12)=1$, so an odd $K_3$ coprime to $3$ is primitively represented by $U^{2}+3V^{2}$ if and only if every one of its prime factors is $\equiv1\pmod 3$; the primes of $\mathcal{B}(N)$ are then exactly the primes blind at $c=3$. Counting them directly up to $10^{7}$:

\begin{center}
\begin{tabular}{lcccc}
\toprule
$N$ & $10^{4}$ & $10^{5}$ & $10^{6}$ & $10^{7}$\\
\midrule
$\#\{n\in\mathcal{B}(N):n\ \text{prime}\}$ & $84$ & $606$ & $4\,540$ & $35\,750$\\
$N(\log N)^{-3/2}$ & $357.8$ & $2\,559.9$ & $19\,473.8$ & $154\,535.9$\\
ratio & $0.2348$ & $0.2367$ & $0.2331$ & $0.2313$\\
\bottomrule
\end{tabular}
\end{center}

The ratio is flat to within $1\%$ over four decades. This shift is not the one used in Theorem~\ref{thm:blind_primes}, and the reason is instructive: at $c=3$ the primes primitively represented split into those with $K_3$ odd ($n\equiv1\pmod 8$, blind: $35\,750$ of them below $10^{7}$) and those with $K_3$ even ($n\equiv5\pmod 8$, \emph{not} blind since $\leg23=-1$: a further $9\,540$), and no congruence selection modulo $8$ is available in~\cite[Thm.~1.1(2)]{fuchs2025}, whose modulus must be coprime to $2\Delta B$. The choice $c=7$ removes the difficulty at the source.

\section{Conclusion and open questions} \label{sec:conclusion}

While the exact evaluation of $W(J)$ (Lemma~\ref{lem:kappa1} and Theorem~\ref{thm:improved_constants_and_effectivity}), the $\tau_3$-type ceiling for the single-prime two-parameter witness system (Proposition~\ref{prop:optimality_23}), and the hybrid $o(N^{2})$ batch schedule (Theorem~\ref{thm:subquadratic_batch}) settle the immediate analytic and algorithmic questions raised by our framework, three structural questions remain open.

\emph{Composite-witness sieves at growing depth and the $2/3$ barrier.} Proposition~\ref{prop:optimality_23} proves that restricting the witnesses $s\mid an+u$ in Theorem~\ref{thm:gen_criterion} to \emph{primes} $s=p$ caps the large-sieve trade-off at the exponent $2/3$ in $|E_2(N;J)|\ll N\exp(-c(\log N)^{2/3})$ (and whether any other single-prime congruence system extracted from the Type~I branch can exceed $\sum_{p\le Z}\varrho(p)/p\ll(\log Z)^{2}$ is left open). On the other hand, Theorem~\ref{thm:fixed_depth} exploits \emph{composite} divisors $s\mid an+u$ via the fixed-point-free involution $x\mapsto-x^{-1}$ on $(\Z/4ua)^{\times}$, achieving the full sieve dimension $\mathfrak{A}(J)/2\asymp J^{2}$ at fixed depth, and Proposition~\ref{prop:sylow_kernel_reduction_preview} shows that on the $2$-primary kernel branch the combinatorial prefactor $C(\mathcal{P}_2(J))$ drops from $2^{\Theta(J^{4})}$ to $2^{\mathcal{O}(J^{2}\log\log J)}$. Extending this subgroup and subset container mechanism to depths $J(N)$ growing with $N$ so as to bypass the single-prime $\tau_3$ ceiling across all of $E_2(N;J)$ is carried out in the companion paper~\cite{dahan2026b}.

\emph{A quasi-linear factorization-free interval procedure.} Theorem~\ref{thm:subquadratic_batch} breaks the quadratic barrier on $[N,2N]$ by combining the $\mathcal{O}(N(\log N)^{3})$ Type~II pass with a truncated Type~I pass and a cofactor scan on the sparse residual set $\mathcal{R}(N)$, achieving $\mathcal{O}(N^{2}(\log N)^{2}\exp(-c_0(\log N)^{2/3}))=\mathcal{O}(N^{2}\exp(-(c_0-o(1))(\log N)^{2/3}))=o(N^{2})$ arithmetic operations. Whether the Type~I pass can be compressed all the way to $\mathcal{O}(N(\log N)^{\mathcal{O}(1)})$ without integer factorization remains open.

\emph{Whether $\bigcap_J E_2(N;J)$ is empty.} By Theorem~\ref{thm:gen_criterion} this is exactly the question of whether every prime $n\equiv1\pmod{24}$ has a Type~II solution. Theorem~\ref{thm:fixed_depth} shows that the intersection has counting function $\ll_A N(\log N)^{-A}$ for every $A$, and Corollary~\ref{cor:H_density} gives the same for the one-parameter analogue $\mathcal{H}$, whose fourteen known members all disappear at depth $5$ once the modulator is free; but emptiness is a different matter, and the general conjecture remains open.

\section*{Data availability}
No datasets were generated or analysed during the current study beyond 
the explicit numerical verifications reported in Section~\ref{sec:7}, all of 
which are elementary, exact, and directly reproducible from the 
algorithms and formulas described in the text.

\section*{Conflict of interest}
The author declares no conflict of interest.

\section*{Funding}
This research received no external funding.

\end{document}